\documentclass[a4paper,11pt,leqno, draft]{article}
\usepackage{amssymb, amsmath, amsthm, graphicx}
\usepackage{tikz}
\usepackage{enumitem}

\newtheorem{thm}{Theorem}[section]
\newtheorem{lem}[thm]{Lemma}
\newtheorem{cor}[thm]{Corollary}
\newtheorem{prop}[thm]{Proposition}
\theoremstyle{definition}

\newtheorem{rem}{Remark}[section]

\newtheorem{assum}{Assumption}
\numberwithin{equation}{section}

\renewcommand{\a}{\alpha}
\renewcommand{\b}{\beta}
\newcommand{\e}{\varepsilon}
\newcommand{\de}{\delta}
\newcommand{\fa}{\varphi}

\renewcommand{\k}{\kappa}

\renewcommand{\th}{\theta}
\newcommand{\si}{\sigma}

\newcommand{\om}{\omega}
\newcommand{\De}{\Delta}

\newcommand{\lan}{\langle}
\newcommand{\ran}{\rangle}

\def\R{{\mathbb{R}}}
\def\N{{\mathbb{N}}}
\def\Z{{\mathbb{Z}}}

\allowdisplaybreaks

\newcommand{\trans}{(\zeta^-_{\rho(t,\cdot)})^{-1}}
\newcommand{\tra}{\zeta_t^{-1}}
\newcommand{\trai}{\zeta_t}

\title{Fluctuations in an Evolutional Model of \\
Two-Dimensional Young Diagrams}
\author{Tadahisa Funaki$^{1), *)}$, Makiko Sasada$^{2)}$, Martin Sauer$^{3)}$
and Bin Xie$^{4)}$}
\date{}

\begin{document}
\setlist{itemsep=0.1em, topsep=0em, leftmargin=*, align=left}
\setlist[2]{labelindent=-1em}

\maketitle

\begin{abstract}
\noindent
We discuss the non-equilibrium fluctuation problem, which corresponds
to the hydrodynamic limit established in \cite{FS}, for the dynamics
of two-dimensional Young diagrams associated with the uniform and
restricted uniform statistics, and derive linear stochastic partial
differential equations in the limit.  We show that their invariant
measures are identical to the Gaussian measures which appear in the
fluctuation limits in the static situations.
\footnote{
\hskip -6mm
$^{1)}$ Graduate School of Mathematical Sciences,
The University of Tokyo, Komaba, Tokyo 153-8914, Japan.
e-mail: funaki@ms.u-tokyo.ac.jp \\
$^{2)}$ Department of Mathematics, Keio University,
3-14-1, Hiyoshi, Kohoku-ku, Yokohama 223-8522, Japan.
e-mail: sasada@math.keio.ac.jp \\
$^{3)}$ Technische Universit\"at Darmstadt, Fachbereich Mathematik,
Schlossgartenstrasse 7, D-64289 Darmstadt, Germany.
e-mail: sauer@mathematik.tu-darmstadt.de \\
$^{4)}$ International Young Researchers Empowerment Center, and
Department of Mathematical Sciences, Faculty of Science,
Shinshu University, 3-1-1 Asahi, Matsumoto, Nagano 390-8621, Japan. \\
e-mail:  bxie@shinshu-u.ac.jp}
\footnote{
\hskip -6mm
$^{*)}$ Corresponding author, Fax: +81-3-5465-7011.}
\footnote{
\hskip -6mm
\textit{Keywords: zero-range process, exclusion process, hydrodynamic limit,
fluctuation, Young diagram.}}
\footnote{
\hskip -6mm
\textit{Abbreviated title $($running head$)$: Fluctuations in
2D Young diagrams.}}
\footnote{
\hskip -6mm
\textit{MSC: primary 60K35, secondary 82C22.}}
\footnote{
\hskip -6mm
\textit{The author$^{1)}$ is supported in part by the JSPS Grants $($A$)$
22244007 and 21654021.}}
\footnote{
\hskip -6mm
\textit{The author$^{3)}$ is supported by the DFG and JSPS as a member of
the International Research Training Group Darmstadt-Tokyo IRTG 1529.}}
\footnote{
\hskip -6mm
\textit{The author$^{4)}$ is supported in part by the JSPS Grant for
young scientist $($B$)$  21740067.}}
\end{abstract}
\section{Introduction}

In our companion paper \cite{FS} we investigated the hydrodynamic limit
for dynamics of two-dimensional Young diagrams associated with the
 grandcanonical ensembles determined from two types of statistics called
 uniform (or Bose) and restricted uniform (or Fermi) statistics introduced
 by Vershik \cite{V}. The aim of the present paper is to study the
 corresponding non-equilibrium dynamic fluctuation problem. The theory of
 the equilibrium dynamic fluctuation around the hydrodynamic limit is well
 established based on the so-called Boltzmann-Gibbs principle, see
 \cite{KL}.  However, the results on the non-equilibrium dynamic
 fluctuations are rather limited, cf. \cite{DPS}, \cite{DG} due to a
 special feature of the models and \cite{CY} in a more general setting.
In the present case we are able to derive linear stochastic partial differential equations (SPDEs)
in the limit. Also, the fluctuations can be
studied in the static situations and these results are reinterpreted from
the dynamic point of view by identifying the static fluctuation limits with
the invariant measures of the limit SPDEs. See \cite{P}, \cite{FVY},
 \cite{Y} for static fluctuations under canonical ensembles.

As shown in \cite{FS}, the dynamics of the two-dimensional Young diagrams
can be transformed into equivalent particle systems by considering their
 height differences. In fact, in the uniform statistics (short term U-case),
the evolution of the height difference $\xi_t = (\xi_t(x))_{x\in\N} \in
 (\Z_+)^\N$ of the Young diagrams' height function $\psi_t(u), u \in \R_+$
 defined by $\xi_t(x) = \psi_t(x-1) - \psi_t(x)$ and supplied with the
 condition $\xi_t(0)=\infty$ performs a weakly asymmetric zero-range
 process on $\N$ with a weakly asymmetric stochastic reservoir at $\{0\}$.
 Here we denote $\Z_+=\{0,1,2,\ldots\}$, $\N = \{1,2,\ldots\}$ and  $\R_+=[0, \infty)$. Such
a particle system is further transformed into a weakly asymmetric simple exclusion
 process $\bar{\eta}_t \in\{0,1\}^\Z$ on the whole integer lattice $\Z$
 without any boundary conditions by rotating the $xy$-plane around the
 origin by $45$ degrees counterclockwise and projecting the system to the
 $x$-axis rescaled by $\sqrt{2}$. This  involves quite a
 nonlinearity as observed in Section 4 of \cite{FS}.

On the other hand, in the restricted uniform statistics (short term
RU-case), the height difference $\eta_t = (\eta_t(x))_{x\in\N} \in
 \{0,1\}^\N$ of the Young diagrams' height function supplied with the
 condition $\eta_t(0)=\infty$ performs a weakly asymmetric simple
exclusion process on $\N$ with a weakly asymmetric stochastic reservoir
at $\{0\}$.

The hydrodynamic limit for the weakly asymmetric simple exclusion process
 $\bar{\eta}_t$ on the whole integer lattice was studied by \cite{G} and
 \cite{DPS}, and the corresponding fluctuation limit by \cite{DPS} and
\cite{DG}.  In these works the convergence of the density fluctuation fields was
shown only in the space of processes taking values in generalized
 functions such as $D([0,\infty),\mathcal{S}'(\R))$ or $D([0,T],
\mathcal{H}')$ for a kind of Sobolev space $\mathcal{H}'$ with negative
index. In the U-case it is indeed necessary to transform $\bar{\eta}_t$
back to $\xi_t$ through a nonlinear map, so that these convergence results
are too weak and it is necessary to establish the tightness of
$\bar{\eta}_t$ (under scaling and linear interpolation) in the space
$D([0,T],D(\R))$, i.\,e. a stronger topology. The boundary condition in the
RU-case needs additional analysis. The fluctuations for the simple exclusion
process with boundary conditions in a symmetric case (i.e.\ $\e=1$) were
studied by \cite{LMO}. The weakly asymmetric case was discussed by
\cite{DELO} but without mathematically rigorous proofs. The tightness in the
space $D([0,T],D(\R))$ beyond the time scale of the hydrodynamic limit was
established by \cite{BG} and they derived the KPZ equation in the limit. We
follow their method with an adjustment concerning the boundary condition in
the RU-case.

In Section 2 we first recall the results of \cite{FS} on the hydrodynamic
limits and then formulate our main results on the fluctuations as Theorems
2.1 and 2.2 for the U-case and Theorem 2.3 for the RU-case. The proofs of
these results are given in Sections 3 and 4. Finally, in Section 5 we
discuss the invariant measures of the SPDEs obtained in the limit and their
relations to those obtained in the static situations, see Theorems
\ref{thm:Uinvm}, \ref{thm:Rinvm} and Proposition \ref{pro-CLT-5.1}.

In this paper, given a Banach space $X$ and $I \subseteq \R$, $C(I,X)$
denotes the set of all continuous functions equipped with the locally
uniform convergence, as well as $D(I,X)$ the set of all c\`{a}dl\`{a}g
functions equipped with the Skorohod topology. Abbreviate $C(I,\R) = C(I)$
and $D(I,\R) = D(I)$. Furthermore define for each $r>0$ the weighted $L^2$-
space $L_r^2(I)$ equipped with the norm $|f|_{L_r^2(I)} =\{\int_{I} |f(u)|^2
e^{-2r|u|}du\}^{1/2}$ and set $L_e^2(I):= \cap_{r>0}L_r^2(I)$.
\section{Main Results}\label{section-2}
We first recall the notation used in the paper \cite{FS} and briefly
summarize the results obtained there. For each $n \in \N$, let
$\mathcal{P}_n$ be the set of all sequences $p= (p_i)_{i\in\N}$ satisfying
$p_1 \ge p_2 \ge \cdots \ge p_i \ge \cdots$, $p_i \in \Z_+$ and
$n(p):= \sum_{i\in\N} p_i =n$. Let $\mathcal{Q}_n$ be the set of all
sequences $q= (q_i)_{i\in\N}\in \mathcal{P}_n$ satisfying $q_i > q_{i+1}$
if $q_i >0$.  For $n=0$, we define $\mathcal{P}_0 = \mathcal{Q}_0 =
\{0\}$, where $0$ is a sequence such that $p_i=0$ for all $i\in \N$.  The
unions of $\mathcal{P}_n$ and $\mathcal{Q}_n$ in $n\in\Z_+$ are denoted by
$\mathcal{P}$ and $\mathcal{Q}$, respectively. The height function of the
Young diagram corresponding to $p \in \mathcal{P}$ is defined by
\begin{equation*} \label{eq:2.1}
\psi_p(u)=\sum_{i\in\N} 1_{\{u < p_i\}}, \quad u\in \R_+,
\end{equation*}
and its scaled height function by
\begin{equation*} \label{eq:2.2}
\tilde{\psi}_p^N(u)= \frac1N \psi_p(Nu), \quad u\in\R_+,
\end{equation*}
for $N\ge 1$. Note that $\psi_q$ and $\tilde{\psi}_q^N$ are defined for
$q\in \mathcal{Q}$, since $\mathcal{Q} \subset \mathcal{P}$. For
$0<\e<1$, the dynamics $p_t := p_t^\e=(p_i(t)) _{i\in\N}$ on
$\mathcal{P}$ and $q_t := q_t^\e=(q_i(t))_{i\in\N}$ on $\mathcal{Q}$
are introduced as Markov processes on these spaces having generators
$L_{\e,U}$ and $L_{\e,R}$, respectively, defined as follows. The operator
$L_{\e,U}$ acts on functions $f:\mathcal{P} \to \R$ as
\begin{equation}\label{eq:2.3}
L_{\e,U}f(p) = \sum_{i \in \N} \bigl[ \e 1_{\{p_{i-1} > p_i\}} \{f(p^{i,
+})-f(p)\}+1_{\{p_i > p_{i+1}\}}\{f(p^{i,-})-f(p)\} \bigr],
\end{equation}
while the operator $L_{\e,R}$ acts on functions $f:\mathcal{Q} \to \R$ as
\begin{equation}  \label{eq:2.4}
L_{\e,R}f(q) = \sum_{i \in \N} \bigl[\e 1_{\{q_{i-1} > q_i + 1\}}
\{f(q^{i,+}) - f(q)\}+ 1_{\{q_i > q_{i+1} + 1 \text{ or } q_i = 1\}}
\{f(q^{i,-}) - f(q)\} \bigr],
\end{equation}
where $p^{i,\pm}=(p_j^{i,\pm})_{j\in \N} \in \mathcal{P}$ are defined for
$i\in \N$ and $p\in \mathcal{P}$ by
\[
p^{i,\pm}_j = \begin{cases}
    p_j  & \text{if $j \neq i$}, \\
    p_i \pm 1 & \text{if $j=i$},
\end{cases}
\]
and $q^{i,\pm}\in \mathcal{Q}$ similarly for $q\in \mathcal{Q}$. In
\eqref{eq:2.3} and \eqref{eq:2.4}, take $p_0 = \infty$ and $q_0 =
\infty$. These processes have the grandcanonical ensembles $\mu_U^\e$ on
$\mathcal{P}$ and $\mu_R^\e$ on $\mathcal{Q}$ as their invariant measures,
respectively, where $\mu_U^\e$ and $\mu_R^\e$ are probability measures on
these spaces defined by
\[
\mu^{\e}_U(p)=\frac{1}{Z_U(\e)}\e^{n(p)}, \quad p \in \mathcal{P} \quad
\text{and} \quad \mu^{\e}_R(q)=\frac{1}{Z_R(\e)}\e^{n(q)}, \quad q \in
\mathcal{Q},
\]
and $Z_U(\e)=\prod^{\infty}_{k=1} (1-\e^k)^{-1}$ and $Z_R(\e)=
\prod^{\infty}_{k=1} (1+\e^k)$ are the normalizing constants.

Choose $\e=\e(N) (= \e_U(N), \e_R(N))$ in such a way that in each case the
averaged size of the Young diagrams under $\mu_U^\e$ or $\mu_R^\e$ is equal
to $N^2$, i.\,e.
\[
E_{\mu^{\e(N)}_U}[n(p)]=N^2 \quad \text{and} \quad E_{\mu^{\e(N)}_R}
[n(q)]=N^2.
\]
The asymptotic behavior $\e_U(N) = 1-\a/N+O(\log N/N^2)$ and $\e_R(N) =
1-\b/N+ O(\log N/N^2)$ as $N\to\infty$, with $\a=\pi/\sqrt{6}$ and
$\b=\pi/\sqrt{12}$ was shown in Lemmas 3.1 and 3.2 in \cite{FS}. Define
the corresponding height functions diffusively scaled in space and time by
\[
\tilde{\psi}_U^N(t,u):= \tilde{\psi}_{p^\e_{N^2t}}^N(u) = \frac1N
\psi_{p^\e_{N^2t}}(Nu)
\quad\text{and}\quad
\tilde{\psi}_R^N(t,u):= \tilde{\psi}_{q^\e_{N^2t}}^N(u) = \frac1N
\psi_{q^\e_{N^2t}}(Nu),
\]
with $\e = \e(N)$. The following results are obtained in \cite{FS},
Theorems 2.1 and 2.2. Denote the partial derivative $\partial_u\psi$ by
$\psi'$.
\begin{enumerate}
\item If $\tilde{\psi}_U^N(0,u)$ converges to $\psi_0\in X_U$ (see below)
in probability as $N\to\infty$, then $\tilde{\psi}_U^N(t,u)$ converges to
$\psi_U(t,u)$ in probability.  Here $\psi(t,u) := \psi_U(t,u)$ is a
unique solution of the following nonlinear partial differential equation
(PDE):
\[
\partial_t\psi = \left(\frac{\psi'}{1-\psi'}\right)' + \a \frac{\psi'}
{1-\psi'},  \quad u\in \R_+^\circ,
\]
with the initial condition $\psi(0,\cdot) = \psi_{U,0}(\cdot)$, where $\R_+^\circ =(0, \infty)$.
\item If $\tilde{\psi}_R^N(0,u)$ converges to $\psi_0\in X_R$ (see below)
in probability as $N\to\infty$, then $\tilde{\psi}_R^N(t,u)$ converges to
$\psi_R(t,u)$ in probability. Here $\psi(t,u) := \psi_R(t,u)$ is a
unique solution of the following nonlinea PDE:
\[
\partial_t\psi = \psi'' + \b \, \psi' (1+\psi'), \quad u \in \R_+,
\]
with the initial condition $\psi(0,\cdot) =\psi_{R,0}(\cdot)$.
\end{enumerate}
Consider these PDEs in the function spaces $X_U$ and $X_R$, respectively,
and their solutions are unique in these classes:
\begin{align*}
X_U:=&  \{\psi:\R_+^\circ \to \R_+^\circ; \psi \in C^1,
\psi^{\prime}<0, \lim_{u \downarrow 0}\psi(u)=\infty, \lim_{u \uparrow
\infty}\psi(u)=0\}, \\
X_R:=& \{\psi:\R_+ \to \R_+; \psi \in C^1, -1 \le \psi^{\prime} \le 0,
\psi'(0)=-1/2, \lim_{u \uparrow \infty}\psi(u)=0\}.
\end{align*}
\begin{figure}[ht]
\begin{center}
\begin{minipage}{6.2cm}
\begin{tikzpicture}[baseline=0]
\draw[thick, ->] (-0.5,0) -- (5,0) node[below] {$u$};
\draw[thick, ->] (0,-0.5) -- (0,4) node[above] {$\tilde{\psi}^{1}_U(t,u)$};
\draw[dotted] (0, 0.5) -- (4,0.5);
\draw[dotted] (0, 1) -- (2.5,1);
\draw[dotted] (0, 1.5) -- (1,1.5);
\draw[dotted] (0, 2) -- (1,2);
\draw[color=blue, very thick] (0, 2.5) -- (1,2.5);
\draw[color=blue, very thick] (1, 1.5) -- (2.5,1.5);
\draw[color=blue, very thick] (2.5, 1) -- (4,1);
\draw[color=blue, very thick] (4, 0) -- (4.95,0);
\draw[dotted] (1, 2.5) -- (1,0) node[below] {$p_4, p_5$};
\draw[dotted] (2.5, 1.5) -- (2.5,0) node[below] {$p_3$};
\draw[dotted] (4, 1) -- (4,0) node[below] {$p_1, p_2$};
\draw[color=blue, fill=blue] (0,3.5) circle (0.05) node[right] {$(p_7,7)$};
\draw[color=blue, fill=blue] (0,3) circle (0.05) node[right] {$(p_6,6)$};
\draw[color=blue] (0,2.5) circle (0.05);
\draw[color=blue, fill=blue] (1,2.5) circle (0.05) node[right] {$(p_5,5)$};
\draw[color=blue, fill=blue] (1,2) circle (0.05) node[right] {$(p_4,4)$};
\draw[color=blue] (1,1.5) circle (0.05);
\draw[color=blue, fill=blue] (2.5,1.5) circle (0.05) node[right] {$(p_3,3)$};
\draw[color=blue, ] (2.5,1) circle (0.05);
\draw[color=blue, fill=blue] (4,1) circle (0.05) node[right] {$(p_2,2)$};
\draw[color=blue, fill=blue] (4,0.5) circle (0.05) node[right] {$(p_1,1)$};
\draw[color=blue] (4,0) circle (0.05);
\end{tikzpicture}
\end{minipage}
\begin{minipage}{6.2cm}
\begin{tikzpicture}[baseline=0]
\draw[thick, ->] (-0.5,0) -- (5,0) node[below] {$u$};
\draw[thick, ->] (0,-0.5) -- (0,4) node[above] {$\phantom{\tilde{\psi}}\psi_U(t,u)$};
\draw[color=blue, domain=0.03:2.7, smooth, scale=1.5] plot (\x,{-sqrt(6)/pi* ln(1-exp(-pi*\x/sqrt(6)))});
\end{tikzpicture}
\end{minipage}
\caption{A typical height function and its scaling limit}
\end{center}
\end{figure}
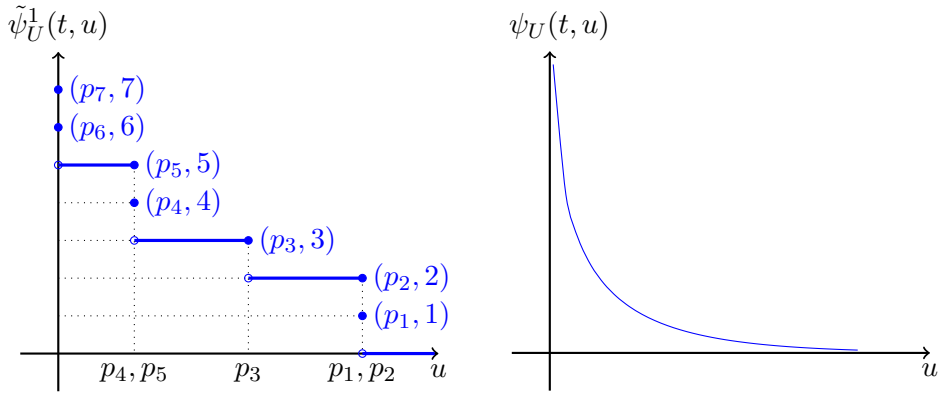

The aim of the present paper is to establish the corresponding fluctuation
limits.  Namely, we consider the fluctuations of $\tilde{\psi}_U^N(t,u)$
and $\tilde{\psi}_R^N(t,u)$ around their limits:
\[
\Psi_U^N(t,u) := \sqrt{N} \bigl(\tilde{\psi}_U^N(t,u) -\psi_U(t,u)\bigr)
\quad \text{and} \quad
\Psi_R^N(t,u) := \sqrt{N}
\bigl(\tilde{\psi}_R^N(t,u) -\psi_R(t,u)\bigr),
\]
which are elements of $D([0, T], D(\R_+^\circ))$ and  $D([0, T], D(\R_+))$,
respectively.

A natural idea in the U-case is to investigate the fluctuation of the curve
$\check{\psi}_U^N(t)$ around $\check{\psi}_U(t)$,
 which are obtained
by rotating the original curves $\tilde{\psi}_U^N(t)$ and $\psi_U(t) =
\{(u,y); y=\psi_U(t,u), u\in \R_+^\circ\}$ located in the first quadrant
of the $uy$-plane by 45 degrees  counterclockwise  around the origin
$O$, respectively, where $\tilde\psi_U^N(t)$ is a continuous indented curve
obtained from the graph $\{(u,y); y=\tilde{\psi}_U^N(t,u), u\in \R_+\}$ of
the original function $\tilde{\psi}_U^N(t,u)$ by filling all jumps by vertical
segments.
In particular, this contains a
part of $y$-axis: $\{(0,y); y \ge \tilde\psi_U^N(t,0)\}$.
\begin{figure}[ht]
\begin{center}
\begin{minipage}{7.3cm}
\begin{tikzpicture}
\draw[thick, ->] (-3.5,0) -- (3.5,0) node[below] {$v$};
\draw[thick, ->] (0,-0.5) -- (0,4) node[above] {$\check{\psi}^{1}_U(t,v)$};
\draw[color=blue, fill=blue] (-3.1815,0) circle (0.05);
\draw[color=blue, fill=blue] (-2.828,0) circle (0.05);
\draw[color=blue, fill=blue] (-2.4745,0) circle (0.05) node[below] {$\frac{\bar{q}_7}{\sqrt{2}}$};
\draw[color=blue, fill=blue] (-2.121,0) circle (0.05) node[below] {$\frac{\bar{q}_6}{\sqrt{2}}$};
\draw[color=blue] (-1.7675,0) circle (0.05);
\draw[color=blue] (-1.414,0) circle (0.05);
\draw[color=blue, fill=blue] (-1.0605,0) circle (0.05) node[below] {$\frac{\bar{q}_5}{\sqrt{2}}$};
\draw[color=blue, fill=blue] (-0.707,0) circle (0.05) node[below] {$\frac{\bar{q}_4}{\sqrt{2}}$};
\draw[color=blue] (-0.3535,0) circle (0.05);
\draw[color=blue] (0,0) circle (0.05);
\draw[color=blue] (0.3535,0) circle (0.05);
\draw[color=blue, fill=blue] (0.707,0) circle (0.05) node[below] {$\frac{\bar{q}_3}{\sqrt{2}}$};
\draw[color=blue] (1.0605,0) circle (0.05);
\draw[color=blue] (1.414,0) circle (0.05);
\draw[color=blue] (1.7675,0) circle (0.05);
\draw[color=blue, fill=blue] (2.121,0) circle (0.05) node[below] {$\frac{\bar{q}_2}{\sqrt{2}}$};
\draw[color=blue, fill=blue] (2.4745,0) circle (0.05) node[below] {$\frac{\bar{q}_1}{\sqrt{2}}$};
\draw[color=blue] (2.828,0) circle (0.05);
\draw[color=blue] (3.1815,0) circle (0.05);
\draw[color=blue, very thick, rotate=45] (0,4.5) -- (0,2.5) -- (1,2.5) -- (1,1.5) -- (2.5,1.5) -- (2.5,1) -- (4,1) -- (4,0) -- (4.5,0);
\draw[dotted] (-1.7675,1.7675) -- (-1.7675,0);
\draw[dotted] (-1.0605,2.4745) -- (-1.0605,0);
\draw[dotted] (-0.3535,1.7675) -- (-0.3535,0);
\draw[dotted] (0.707,2.828) -- (0.707,0);
\draw[dotted] (1.0605,2.4745) -- (1.0605,0);
\draw[dotted] (2.121,3.535) -- (2.121,0);
\draw[dotted] (2.828,2.828) -- (2.828,0);
\draw[dashed] (0,0) -- (3.5,3.5);
\draw[dashed] (0,0) -- (-3.5,3.5);
\end{tikzpicture}
\end{minipage}
\begin{minipage}{7.3cm}
\begin{tikzpicture}
\draw[thick, ->] (-3.5,0) -- (3.5,0) node[below] {$v$};
\draw[thick, ->] (0,-0.5) -- (0,4) node[above] {$\check{\psi}_U(t,v)$};
\draw[dashed] (0,0) -- (3.5,3.5);
\draw[dashed] (0,0) -- (-3.5,3.5);
\draw[color=blue, rotate=45, domain=0.03:2.7, smooth, scale=1.5] plot (\x,{-sqrt(6)/pi* ln(1-exp(-pi*\x/sqrt(6)))});
\draw[color=black] (2.4745,0) circle (0) node[below] {$\phantom{\frac{\bar{q}_1}{\sqrt{2}}}$};
\end{tikzpicture}
\end{minipage}
\end{center}
\caption{Rotating by $45^\circ$ yields functions on $\R$ and a particle
system on $\Z$.}
\end{figure}
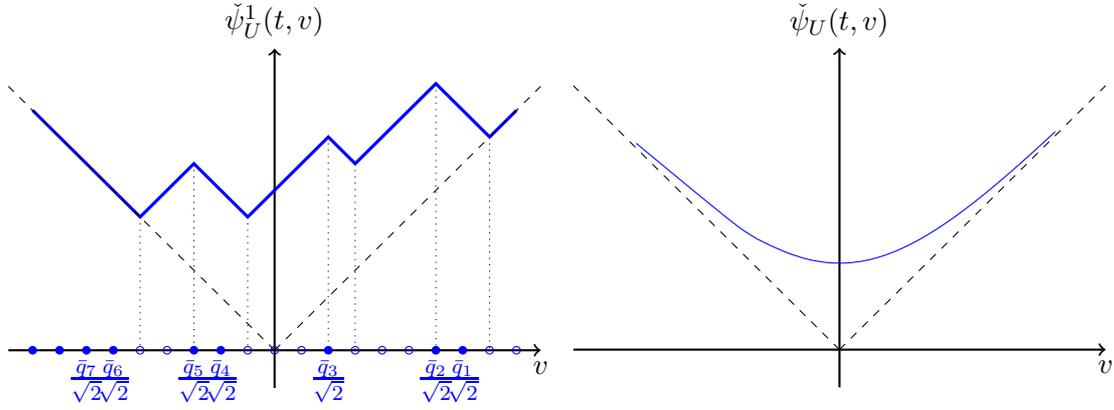

Then, we consider
\[
\check\Psi_U^N(t,v) := \sqrt{N} \bigl(\check\psi_U^N(t,v)
-\check\psi_U(t,v)\bigr),  \quad v\in \R,
\]
which is an element of $D([0, T], C(\R))$.  The fluctuation
$\check\Psi_U^N(t)$ defined as above is a natural object to study, since
the Young diagrams corresponding to the class $\mathcal{P}$ belong to the
same class under the reflection with respect to the line $\{y=u\}$, while
those corresponding to $\mathcal{Q}$ do not have such property in general.

We are now at the position to formulate our main theorems. In the U-case, we
first state the result for $\check\Psi_U^N(t)$ and then apply it to
$\Psi_U^N(t)$.  We assume the following three conditions on the initial
values $\{\check\Psi_U^N(0,v)\}_N$ and $\{\check\psi_U^N(0,0)\}_N$:
\begin{assum}
\begin{enumerate}
\item For every $\k \in \N$, the following holds:
\begin{enumerate}
\item $\sup_{N\in \N}E[\exp\{\k \check\psi_U^N(0,0)\}]<\infty$,
\item $\sup_{N \in \N} \sup_{v\in\R} E[|\check{\Psi}_U^N(0,v)|^{2\k}]
    <\infty$,
\item for  any $v_1, v_2 \in \Z/N$ $\sup_{N\in \N}E[|
\check{\Psi}_U^N(0,v_1) -\check{\Psi}_U^N(0,v_2)|^{2\k}] \leq C|v_1-v_2|^{\k}$ with $C>0$.
\end{enumerate}
\item $\{\check{\Psi}_U^N(0,v)\}_N$ are independent of the
noises determining the process $\{p_t^{\e(N)}; t\ge 0\}$.
\item $\check{\Psi}_U^N(0,v)$ converges weakly to $\check{\Psi}_{U,0}(v)$ in
$C(\R)$, and $E[|\check{\Psi}_{U,0}|_{L_r^2(\R)}^2]<\infty$ for all $r>0$.
\end{enumerate}
\end{assum}
For every initial value $\psi_U(0) \in X_U$, one can easily construct non-random
or random sequences $\{\tilde\psi_U^N(0)\}_N$ or
$\{\check\psi_U^N(0)\}_N$, which satisfy these three conditions.
\begin{thm}(U-case under rotation) \label{thm:2.0}
Under Assumption 1, $\check{\Psi}^N_U(t,v)$ converges weakly to
$\check{\Psi}_U(t,v)$ as $N \to \infty$  on the space $D([0,T], C(\R))$
for every $T>0$. The limit $\check{\Psi}_U(t,v)$ is in $C([0,T], C(\R))$
(a.\,s.) and characterized as a solution of the following SPDE:
\begin{equation}\label{eq:SPDE2.0}
\left\{
\begin{aligned}
\partial_t \check{\Psi}_U(t,v)  =&  \tfrac12 \check{\Psi}_U''(t,v) +
\tfrac{\a}{\sqrt{2}} (1-2\rho(t,\sqrt{2}v))  \check{\Psi}_U'(t,v) \\
& \quad + 2^{\frac34}\sqrt{\rho(t,\sqrt{2}v)(1-\rho(t,\sqrt{2}v))}
\dot{W}(t,v),\\
\check{\Psi}_U(0,v) =& \check{\Psi}_{U,0}(v),
\end{aligned}
\right.
\end{equation}
where $\rho(t, \cdot)$ is the solution of the PDE \eqref{eqn:burgers}
below, or equivalently $\rho(t,\cdot)= \Phi_U(\psi_U(t))(\cdot)$ with the map $\Phi_U:
X_U\to Y_U$ defined in Proposition 4.4 of \cite{FS}  (or given explicitly
in the proof of Lemma \ref{lem:3.2-b} below), and $\dot{W}(t,v)$ is the
space-time white noise on $[0,T]\times\R$.
\end{thm}
The solution of \eqref{eq:SPDE2.0} is defined in a weak sense:  We call
$\check{\Psi}_U(t,v)$ a solution of the SPDE \eqref{eq:SPDE2.0} if it is
adapted with respect to the increasing $\si$-fields generated by $\{\dot{W}
(s); s\le t\}$, satisfies $\check{\Psi}_U\in C([0,T],C(\R)) \cap
C([0,T],L_e^2(\R))$ (a.\,s.)  and for every $f \in C_0^{1,2}([0,T] \times
\R)$,
\begin{align*}
\lan \check{\Psi}_U(t),f(t)\ran  &= \lan \check{\Psi}_{U,0},f(0)\ran  +
\int_0^t \lan \check{\Psi}_U(s), \tfrac12 f''(s) - \tfrac{\a}{\sqrt{2}}
\bigl((1-2\rho(s, \sqrt{2}\ \cdot)) f(s)\bigr)'+ \partial_s f(s)\ran
ds\\
&\qquad + \int_0^t \int_{\R} f(s,v) 2^{\frac34}\sqrt{\rho(s,\sqrt{2}v)
(1-\rho(s,\sqrt{2}v))}W(dsdv) \text{ a.\,s.},
\end{align*}
where $\lan\check\Psi,f\ran = \int_\R \check\Psi(v)f(v) dv$. Similar to
the SPDE \eqref{eq:SPDE-2} (with the boundary condition) stated below, one
can show that the solution of \eqref{eq:SPDE2.0} is equivalent to its mild
form and unique in the above class.

Although the directions of the fluctuations are different in
$\check{\Psi}_U^N$ and $\Psi_U^N$, we still are able to deduce the next
theorem from Theorem \ref{thm:2.0}. As pointed out before, the
transformation is nonlinear, so it is important that the convergence in
Theorem \ref{thm:2.0} is shown in a function space $D([0,T],C(\R))$.
\begin{thm}(U-case) \label{thm:2.1}
Under Assumption 1, $\Psi_U^N(t,u)$ converges weakly to $\Psi_U(t,u)$ as
$N\to\infty$ on the space $D([0,T],D(\R_+^\circ))$ for every $T>0$.  The
limit $\Psi_U(t,u)$ is in $C([0,T],C(\R_+^\circ))$ (a.\,s.) and a solution
of the following SPDE:
\begin{equation}\label{eq:SPDE-1}
\left\{
\begin{aligned}
\partial_t\Psi_U(t,u) =& \left(\frac{\Psi_U'(t,u)}
{(1+\rho_U(t,u))^2}\right)' + \a \frac{\Psi_U'(t,u)}{(1+\rho_U(t,u))^2} +
\sqrt{\frac{2\rho_U(t,u)}{1+\rho_U(t,u)}} \dot{W}(t,u),\\
\Psi_U(0,u) =& \Psi_{U,0}(u),
\end{aligned}
\right.
\end{equation}
where $\rho_U(t,u) = - \psi_U'(t,u)$ and $\dot{W}(t,u)$ is the space-time
white noise on $[0,T]\times\R_+^\circ$.
\end{thm}
Let $\tilde{L}_r^2(\R_+^\circ), r>0$ be the weighted $L^2$-space of
functions on $\R_+^\circ$ equipped with the following norm: Take a positive
function $g_r \in C^\infty(\R_+^\circ)$ such that $g_r(u) = u^{1+2r/\a}$
for $u\in (0,1]$ and $g_r(u)=e^{-2ru}$ for $u\in [2,\infty)$, and define $|
\Psi|_{\tilde{L}_r^2(\R_+^\circ)} = \{\int_{\R_+^\circ}|\Psi(u)|^2
g_r(u)du\}^{1/2}$. Again, we set $\tilde{L}_e^2(\R_+^\circ) = \cap_{r>0}
\tilde{L}_r^2(\R_+^\circ)$. The reason to introduce these spaces is
explained in Remark \ref{rem:3.1} below.

The solution of the SPDE \eqref{eq:SPDE-1} is defined in a weak sense: We
call $\Psi_U(t,u)$ a solution of the SPDE \eqref{eq:SPDE-1} if it is
adapted, satisfies $\Psi_U\in C([0,T],C(\R_+^\circ)) \cap C([0,T],
\tilde{L}_e^2(\R_+^\circ)$ (a.\,s.) and for every $f \in C_0^{1,2}
([0,T]\times \R_+^\circ)$,
\begin{equation}\label{eq:3.13}
\begin{split}
\lan \Psi_U(t),f(t)\ran & = \lan \Psi_{U,0},f(0)\ran + \int_0^t \lan
\Psi_U(s),\Biggl(\frac{f'(s)-\a \, f(s)}{(1+\rho_U(s))^2}\Biggr)'
+\partial_s f(s)\ran ds\\
& \quad + \int_0^t \int_{\R_+^\circ} f(s,u) \sqrt{\frac{2\rho_U(s,u)}
{1+\rho_U(s,u)}} W(dsdu) \text{ a.\,s.},
\end{split}
\end{equation}
where $\lan\Psi_U,f\ran = \int_{\R_+^\circ}\Psi_U(u)f(u)du$. The solution
of the SPDE \eqref{eq:SPDE-1} is unique under condition \eqref{eq:LB}, see
Lemma \ref{lem:3.4} and Proposition \ref{porp:example} below.
\begin{rem}\label{rem:2.1}
The boundary condition for the SPDE \eqref{eq:SPDE-1} is unnecessary. Here
this is seen at least under the equilibrium situation: $\rho_U(t,u)
=\rho_U(u)$.  Consider the corresponding diffusion on $\R_+^\circ$ to the
linear differential operator appearing in \eqref{eq:SPDE-1} given by
\[
dX_t = b(X_t) dt + \si (X_t) d B_t
\]
with
\[
b(x) = \frac{\a}{(1+\rho_U(x))^2} - 2 \frac{\rho_U'(x)}{(1+\rho_U(x))^3},
\quad\si (x) = \frac{1}{1+\rho_U(x)},
\]
and $B_t$ a $1$-dimensional Brownian motion. Then we can show that the
corresponding scale function defined on $\R_+^\circ$ diverges to $-\infty$
as $u\downarrow 0$. This means that $0$ is a natural boundary for $X_t$,
see, e.g., Proposition 5.22 in \cite{KS}. Accordingly, we do not need any
boundary condition at $u=0$.
\end{rem}
For the RU-case, we assume the following three conditions on the initial
values $\{\Psi_R^N(0,u)\}_N$ and $\{\tilde{\psi}_R^N(0,0) \}_N$:
\begin{assum}
\begin{enumerate}
\item For any $\k \in \N$, the following holds:
\begin{enumerate}
\item $\sup_{N\in \N}E[\exp\{\k \tilde{\psi}_R^N(0,0) \}]<\infty$,
\item $\sup_{N \in  \N}\sup_{u\in\R_+} E[|\Psi_R^N(0,u)|^{2\k}]<\infty$,
\item  for any $u_1, u_2 \in \N/N$ $\sup_{N \in \N} E[|\Psi_R^N(0,u_1)
-\Psi_R^N(0,u_2)|^{2\k}]\leq C|u_1-u_2|^{\k}$ \\ with $C>0$.
\end{enumerate}
\item $\{\Psi_R^N(0,u)\}_N$ are independent of the noises determining the
process $\{q_t^{\e(N)}; t\ge 0\}$.
\item $\Psi_R^N(0,u)$ converges weakly to $\Psi_{R,0}(u)$ in
$D(\mathbb{R}_+)$. Moreover, we assume that $\Psi_{R,0} \in
C(\mathbb{R}_+)$ (a.\,s.) and $E[|\Psi_{R,0}|_{L_r^2(\R_+)}^2]< \infty$
for all $r>0$.
\end{enumerate}
\end{assum}
\begin{rem}
The scaled height $\tilde{\psi}_R^N(0,0)$ at $t=0$ and $u=0$ appearing in
Assumption 2-(1)(i) is equal to the initial particle number of the weakly
asymmetric simple exclusion process $\eta_t$ divided by $N$, see Section
\ref{section-4}.
\end{rem}
\begin{thm}(RU-case) \label{thm:2.2}
Under Assumption 2, $\Psi_R^N(t,u)$ converges weakly to $\Psi_R(t,u)$ as
$N\to\infty$ on the space $D([0,T],D(\R_+))$ for every $T>0$.  The limit
$\Psi_R(t,u)$ is in $C([0,T],C(\R_+))$ (a.\,s.) and characterized as a
solution of the following SPDE:
\begin{equation}\label{eq:SPDE-2}
\left\{
\begin{aligned}
\partial_t\Psi_R(t,u) & = \Psi_R''(t,u) + \b (1-2\rho_R(t,u))
\Psi_R'(t,u)\\
& \qquad + \sqrt{2\rho_R(t,u)(1-\rho_R(t,u))} \dot{W}(t,u), \\
\Psi_R'(t,0) & = 0, \\
\Psi_R(0, u) & =\Psi_{R,0}(u),
\end{aligned}
\right.
\end{equation}
where $\rho_R(t,u) = - \psi_R'(t,u)$ and  $\dot{W}(t,u)$ is the space-time
white noise on $[0,T]\times\R_+$.
\end{thm}
Again, we say $\Psi_R(t,u)$ is a solution of the SPDE \eqref{eq:SPDE-2} if
it is adapted, satisfies $\Psi_R\in C([0,T],C(\R_+)) \cap
C([0,T],L_e^2(\R_+))$ (a.\,s.) and for every $f \in C^{1,2}_0 ([0,T] \times
\R_+)$ satisfying $f'(t,0)=0$ the following holds:
\begin{equation}\label{eq:4.5}
\begin{split}
\lan\Psi_R(t),f(t)\ran =& \lan\Psi_{R,0},f(0)\ran + \int_0^t
\lan\Psi_R(s),f''(s)-\b ((1-2\rho_R(s))f(s))' + \partial_s f\ran ds\\
& + \int_0^t \int_{\R_+} f(s,u) \sqrt{2\rho_R(s,u)(1-\rho_R(s,u))}
W(dsdu) \text{ a.\,s.}
\end{split}
\end{equation}
Similarly as in \cite{GN}, one can show that the solution of
\eqref{eq:SPDE-2} is equivalent to its mild form, that is, $\Psi_R(t,u)$ is
an $L_e^2(\mathbb{R}_+)$-valued adapted process and the following holds:
\begin{align*}
\Psi_R(t,u) = &\int_{\mathbb{R}_+} \mathfrak{p}(t,u, v)\Psi_{R,0}(v)dv +
\int_0^t\int_{\R_+}  2\b \mathfrak{p}(t-s,u, v)\rho_R'(s,v) \Psi_R(s,
v)dvds\\
& - \int_0^t\int_{\R_+}   \frac{\partial}{\partial v} \mathfrak{p}(t-
s,u, v)\b (1-2\rho_R(s,v)) \Psi_R(s,v)dvds \\
&  + \int_0^t \int_{\R_+} \mathfrak{p}(t-s,u, v) \sqrt{2\rho_R(s,v)
(1-\rho_R(s,v))} W(dsdv) \text{ a.\,s.},
\end{align*}
where $\mathfrak{p}(t,u, v)$ is the fundamental solution to $\partial_t
\Psi(t,u) = \Psi''(t,u)$ with the homogeneous Neumann boundary condition at
$0$, that is
$
\mathfrak{p}(t,u, v)=\frac{1}{\sqrt{4\pi t}} \{e^{-\frac{(u-v)^2}
{4t}}+e^{-\frac{(u+v)^2}{4t}} \}, u, v \in \R_+.
$
The properties of $\rho_R(t,u)$ and basic estimates for $\mathfrak{p}(t,u,
v)$ imply the existence and uniqueness of the solution to \eqref{eq:SPDE-2}.
On the other hand, one can also show the continuity of the trajectory of
$\Psi_R(t, \cdot)$ as an $L_e^2(\mathbb{R}_+)$-valued process and the joint
continuity in $t$ and $u$. Since the arguments are standard, the details are
omitted.
\section{Proof of Theorems \ref{thm:2.0} and \ref{thm:2.1}}
\label{sec:uniformproof}
\subsection{Proof of Theorem \ref{thm:2.0}}
As already pointed out in Section 1 (or see Section 4 of \cite{FS} for more
details), the height difference $\xi_t =(\xi_t(x))_{x\in\N} \in (\Z_+)^\N$
of $\psi_{p_t}$ can be transformed into a weakly asymmetric simple exclusion
process $\bar{\eta}_t =(\bar{\eta}_t(x))_{x\in\Z} \in \{0,1\}^\Z$ on a
whole integer lattice $\Z$. For further use, we introduce  two functions
$\zeta^-_{\bar{\eta}_t}$ and $\zeta^+_{\bar{\eta}_t}$ on $\Z$ by
\[
\zeta^-_{\bar{\eta}_t}(x) := \sum_{z \leq x} (1 -\bar{\eta}_t(z))
\quad \text{and}\quad
 \zeta^+_{\bar{\eta}_t}(x) := \sum_{z \geq x+1} \bar{\eta}_t(z),
\]
which are the main parts of the transformation. The scaled empirical
measures of the  time accelerated process
$\bar{\eta}^N_t:=\bar{\eta}_{N^2t}$ of $\bar{\eta}_t$ given by
\[
\pi_t^N(dv) := \frac1N \sum_{x \in \Z} \bar{\eta}_t^N(x) \de _{\frac{x}
{N}}(dv), \quad v\in \R,
\]
converge, as $N \to \infty$, to the unique classical solution $\rho(t,v)$
of
\begin{equation} \label{eqn:burgers}
\left\{
\begin{aligned}
\partial_t \rho(t,v) &= \rho''(t,v) + \a  (\rho(t,v) (1-\rho(t,v)))',
\quad  t>0, v\in\R,\\
\rho(0,v) &= \rho_0(v).
\end{aligned}
\right.
\end{equation}
See Proposition 4.2 of \cite{FS} for the precise statement and
distinguish $\rho(t,v)$ from $\rho_U(t,u)$ in Theorem \ref{thm:2.1},
though we use similar notation. Furthermore, the continuous version
of the inverse transformation above leads to $\psi_U(t,u) =
\zeta^+_{\rho(t,\cdot)} \bigl(\trans (u)\bigr)$, with
\[
\zeta^-_{\rho(t,\cdot)}(v) := \int_{-\infty}^v (1 -\rho(t,w))dw
\quad \text{and}\quad
\zeta^+_{\rho(t,\cdot)}(v) := \int^\infty_v \rho(t,w)dw.
\]
This is indeed defined via the inverse map of $\Phi_U$, see Proposition 4.4
of \cite{FS}. To shorten notation, we will use $\trai (v) =
\zeta^-_{\rho(t, \cdot)}(v)$ and $\trai^N (x) = \zeta^-_{\bar{\eta}_t^N}
(x)$. The fluctuations for $\bar{\eta}_t^N$ around $\rho(t,\cdot)$
given by
\[
\bar{\xi}^N_U(t,dv) := \sqrt{N} \Bigl( \pi_t^N(dv) - \rho(t,v)dv \Bigr),
\]
are considered as  distribution-valued processes  in \cite{DG}. Due to
the fact that we want to deal with the height function $\tilde{\psi}_U^N(t,
v)$, we will look at an integrated version of $ \bar{\xi}^N_U(t,dv)$,
namely
\[
\bar{\Psi}^N_U(t,v) := \sqrt{N} \Bigl( \pi_t^N\bigl([v,\infty)\bigr)
- \int_v^\infty \rho(t,w)dw \Bigr).
\]
The asymptotic properties of $\rho(t,\cdot)$ and of the tails of
$\pi^N_t$ guarantee that the integrals are finite for all $v \in
\R$, therefore $\bar{\Psi}^N_U(t, v)$ is well-defined. There is an
immediate result on the fluctuations following from Theorem
\ref{thm:2.2} for the process with a stochastic reservoir at
$\{0\}$.

\begin{assum}
\begin{enumerate}
\item For every $\k \in \N$, the following holds:
\begin{enumerate}
\item $\sup_{N\in \N}E[\exp\{\k  \pi_0^N([0, \infty) \}]<\infty$,
\item $\sup_{N \in \N} \sup_{v\in\R} E[|\bar{\Psi}_U^N(0,v)|^{2\k}] <
\infty$,
\item for any $v_1, v_2 \in \Z/N$ $\sup_{N\in \N} E[|\bar{\Psi}_U^N
(0,v_1) -\bar{\Psi}_U^N(0,v_2)|^{2\k}] \leq C|v_1-v_2|^{\k}$ with $C>0$.
\end{enumerate}
\item $\{\bar{\Psi}_U^N(0,v)\}_N$ are independent of the noises
determining the process $\{\xi_t; t\ge 0\}$.
\item $\bar{\Psi}_U^N(0,v)$ converges weakly to $\bar{\Psi}_0(v)$ in
$D(\R)$, and $\bar{\Psi}_0 \in C(\mathbb{R})$ (a.\,s.) such that  for all
$r>0$ $E[|\bar{\Psi}_0|_{L_r^2(\R)}^2]<\infty$.
\end{enumerate}
\end{assum}
\begin{prop}\label{prop:ConvPsiBar}
Under Assumption 3, as  $N \to \infty$, $\bar{\Psi}^N_U(t,v)$ converges weakly on the space
$D([0,T], D(\R))$ to $\bar{\Psi}_U(t,v)$. The limit is characterized as the
unique (weak) solution of the following SPDE:
\begin{equation} \label{eq:SPDE-Prop}
\partial_t \bar{\Psi}_U(t,v)  = \bar{\Psi}_U''(t,v) + \a (1-2\rho(t,v))
 \bar{\Psi}_U'(t,v)
+ \sqrt{2\rho(t,v)(1-\rho(t,v))} \dot{W}(t,v),
\end{equation}
i.\,e. $\bar{\Psi}_U\in C([0,T],C(\R)) \cap C([0,T],L_e^2(\R))$ (a.\,s.)
and for every $f \in C_0^{1,2}([0,T] \times \R)$,
\begin{equation}\label{eqn:spdePsiBar}
\begin{split}
\lan \bar{\Psi}_U(t),f(t)\ran =& \lan
\bar{\Psi}_{U,0},f(0)\ran + \int_0^t \lan \bar{\Psi}_U(s),
f''(s) - \a \bigl((1-2\rho(s)) f(s)\bigr)'
+ \partial_s f(s)\ran ds\\
& + \int_0^t \int_{\R} f(s,v)
\sqrt{2\rho(s,v)(1-\rho(s,v))}W(dsdv) \text{ a.\,s.},
\end{split}
\end{equation}
with $\rho(t)$ being the solution to \eqref{eqn:burgers}, $\dot{W}(t,v)$
being the space-time white noise on $[0,T]\times\R$.
\end{prop}
\begin{rem}
{\rm (1)} In Assumption 3-(1)(i), $\pi_0^N([0, \infty))$ represents the initial
particle number divided by $N$ of $\bar{\eta}_t$ on the positive side.
Recalling the definition of the empirical measures of vacant sites of
$\bar{\eta}_t$: $\hat{\pi}_t^N(dv) =\frac1N \sum_{x \in \Z}
(1-\bar{\eta}_t^N(x)) \de _{x/N}(dv)$ given in Lemma 4.3 of \cite{FS},
this assumption implies a similar condition for the initial density
$\hat{\pi}_0^N((-\infty, 0])$ of vacant sites divided by $N$ on the
negative side by the symmetry in the state space $\mathcal{X}_U$ of
$\bar{\eta}_t$ given in Section 4.1 of \cite{FS}.
\\
{\rm (2)} The fluctuation limit for $\bar{\Psi}_U^N(t,v)$ on the positive side
can be studied similarly to Theorem \ref{thm:2.2}.  To study it on the
negative side, we note that $\bar{\Psi}_U^N(t,v)$ is equal to
\[
\hat{\bar{\Psi}}^N_U(t,v) := \sqrt{N} \Bigl(
\hat{\pi}_t^N\bigl((-\infty,v]
\bigr) - \int_{-\infty}^v (1-\rho(t,w))dw \Bigr)
\]
with an error less than $\sqrt{N}/N$.  The fluctuation limit for
$\hat{\bar{\Psi}}^N_U(t,v)$ (in particular, the tightness of the Hopf-Cole
transformed process) on the negative side is shown similarly to Theorem
\ref{thm:2.2} by looking at $1-\bar{\eta}_t^N(x)$ instead of
$\bar{\eta}_t^N(x)$.

\end{rem}
\begin{rem}
Dittrich and G\"{a}rtner \cite{DG} proved the fluctuation results for
$\bar{\xi}_U^N(t,dv)$ as a distribution-valued process. However, this is
not sufficient for our purpose. Indeed, since we will apply a nonlinear
transformation in the next stage, we need to establish the convergence in a
usual function space formulated as in Proposition \ref{prop:ConvPsiBar}.
This is essentially carried out in Section \ref{section-4}.
\end{rem}
We now prepare two lemmas to deduce Theorem \ref{thm:2.0} from Proposition
\ref{prop:ConvPsiBar}. Recall that  $p\in\mathcal{P}$ determines
$\psi_p$ and $\tilde\psi_p^N$ as well as $\bar\eta =
(\bar\eta(x))_{x\in\Z}$ and the empirical measures $\pi^N(dv)$ on $\R$. For simplicity, we will write $\pi(dv)$ for
$\pi^1(dv)$ in the following. The
next lemma concerns the indented curves $\check\psi_U^N$ and
$\check\psi_U^1$ (with $N=1$), obtained by rotating $\tilde\psi_p^N$
respectively $\tilde\psi_p^1$ as described before.
\begin{lem}\label{lem:3.2-a}
We number the set $\{x\in\Z; \bar\eta(x)=1\}$ from the right as
$\{\bar{q}_i\}_{i=1}^\infty$, that is, $\bar{q}_1 = \max\{x\in\Z;
\bar\eta(x)=1\},\  \bar{q}_2 = \max\{x<\bar{q}_1; \bar\eta(x)=1\}$ and so
on. Then, we have that
\begin{equation} \label{eq:3.2-a}
\check\psi_U^1(v) = \sqrt{2}\pi([\sqrt{2}v,\infty)) + v,
\end{equation}
for all $v\in \cup_{i=0}^\infty [(\bar{q}_{i+1}+1)/\sqrt{2}, \bar{q}_i/
\sqrt{2}]$, where $\bar{q}_0=\infty$. In particular for arbitrary $v\in\R$
\begin{align} \label{eq:3.2-b}
& |\check\psi_U^1(v) - \{\sqrt{2}\pi([\sqrt{2}v,\infty)) + v\}| \le
\sqrt{2},\\
& |\check\psi_U^N(v) - \{\sqrt{2}\pi^N([\sqrt{2}v,\infty)) + v\}| \le
\frac{\sqrt{2}}{N}.  \label{eq:3.2-c}
\end{align}
\end{lem}
\begin{proof}
Set $h(v) := 2\pi([v,\infty)) + v$.  Since $\pi([\bar{q}_i,\infty)) =
\sharp\{j;\bar{q}_j \ge \bar{q}_i\} = i$ for $i\in \N$, we have
$h(\bar{q}_i) = 2i+\bar{q}_i = 2i + (p_i-i) = p_i+i$, which is equal to the
height of the curve $\check\psi^1(v)$ at $v=\bar{q}_i/\sqrt{2}$ multiplied
by $\sqrt{2}$, i.\,e. $h(\bar{q}_i) = \sqrt{2}\check\psi^1
(\bar{q}_i/\sqrt{2})$. Therefore \eqref{eq:3.2-a} holds for
$v=\bar{q}_i/\sqrt{2}$. The functions on both sides of \eqref{eq:3.2-a}
have slope $1$ on the intervals $\bigl((\bar{q}_{i+1}+1)/\sqrt{2},
\bar{q}_i/\sqrt{2}\bigr)$ which yields the first assertion. The function
$h(v)/\sqrt{2}$ has a jump with size $\sqrt{2}$ at $v=\bar{q}_i$ and this
leads to \eqref{eq:3.2-b}. \eqref{eq:3.2-c} follows from \eqref{eq:3.2-b}
by scaling.
\end{proof}
The second lemma concerns the curve $\check\psi_U$ obtained by rotating
$\psi_U$.
\begin{lem} \label{lem:3.2-b}
It holds that
\begin{equation} \label{eq:3.2-d}
\check\psi_U(v) = \sqrt{2}\int_{\sqrt{2}v}^\infty \rho(w)dw + v,
\quad v\in \R,
\end{equation}
where $\rho(v) := \Phi_U(\psi_U)(v)$, see Theorem  \ref{thm:2.0} for the map $\Phi_U$.
\end{lem}
\begin{proof}
An explicit representation of the rotation via its rotation matrix yields
\[
\begin{pmatrix}
v \\ \check\psi(v)
\end{pmatrix}
= \frac1{\sqrt{2}}
\begin{pmatrix}
G_\psi(u) \\ u+ \psi(u)
\end{pmatrix},
\]
where $G_\psi(u) = u- \psi(u)$.  This implies that $\check\psi(v) =
\bigl\{ G_\psi^{-1}(\sqrt{2}v) + \psi(G_\psi^{-1}(\sqrt{2}v))
\bigr\}/\sqrt{2}$. Since $(G_\psi^{-1})'(v) = 1/\{1-\psi'(G_\psi^{-1}
(v))\}$, this implies
\[
\check\psi'(v) = \frac{1+\psi'(G_\psi^{-1}(\sqrt{2}v))}
{1-\psi'(G_\psi^{-1}(\sqrt{2}v))}.
\]
Together with $\rho(v) = \Phi_U(\psi)(v) = - \psi'(G_\psi^{-1}(v))
/\{1-\psi'(G_\psi^{-1}(v))\}$ or written equivalently as
$\psi'(G_\psi^{-1}(v)) = - \rho(v) / (1-\rho(v))$ we obtain
$\check\psi'(v) = 1-2\rho(\sqrt{2}v)$.

The derivative of the right hand side of \eqref{eq:3.2-d} is given by
$1-2\rho(\sqrt{2}v)$, which coincides with $\check\psi'(v)$. Since both
curves have $\{y=v\}$ as an asymptotic line for $v\to\infty$,
\eqref{eq:3.2-d} is proven for all $v\in \R$.
\end{proof}
There is an immediate corollary based on 
Lemmas \ref{lem:3.2-a} and \ref{lem:3.2-b}.
\begin{cor} \label{cor:3.2-c}
For all $t\ge 0$ and $v\in \R$ the relation
\[
\check\Psi_U^N(t,v) = \sqrt{2} \bar\Psi_U^N(t,\sqrt{2}v) +
\check{R}^N(t,v)
\]
holds with an error term satisfying $|\check{R}^N(t,v)|\le \sqrt{2/N}$.
\end{cor}
\begin{proof}[Proof of Theorem \ref{thm:2.0}]
By Corollary \ref{cor:3.2-c}, the limits $\check\Psi_U(t,v)$ and
$\bar\Psi_U(t,v)$ of $\check\Psi_U^N(t,v)$ and $\bar\Psi_U^N(t,v)$ are
related by $\check\Psi_U(t,v) = \sqrt{2}\bar\Psi_U(t,\sqrt{2}v)$.
Therefore we can derive the SPDE \eqref{eq:SPDE2.0} from the SPDE
\eqref{eq:SPDE-Prop} in the weak formulation by replacing the space-time
white noise properly. Corollary \ref{cor:3.2-c} also shows that Assumptions
1 and 3 are mutually equivalent.
\end{proof}
\subsection{Proof of Theorem \ref{thm:2.1}}
In order to derive the SPDE \eqref{eq:SPDE-1} for the limit
$\Psi_U(t,u)$ of $\Psi^N_U(t,u)$, we are not able to apply the same
transformation used to get $\psi_U(t, \cdot)$ from $\rho(t, \cdot)$ because the random
noise certainly makes it impossible to extend it to the spaces containing
$\check\Psi_U$ or equivalently $\bar{\Psi}_U$ and $\Psi_U$. Instead, we
exploit some of the calculations made in Section 4 of \cite{FS}. For every
$f \in C_0(\R_+^\circ)$, we set $F(u)= \int_0^u f(v)dv$, and then we have
that
\begin{align}
 \int_{\R_+^\circ} \tilde{\psi}_U^N(t,u) f(u) du =
&\frac1N \sum_{x \in \Z} F\bigl(\tfrac1N \trai^N(x) \bigr)
\bar{\eta}_t^N(x), \label{eqn:cor1} \\
\int_{\R_+^\circ} \psi_U(t,u) f(u) du = & \int_{\R} F\bigl(\trai(v)
\bigr)
\rho(t,v) dv. \label{eqn:cor2}
\end{align}
These are the key identities for our next proposition.
We will employ Proposition \ref{prop:ConvPsiBar} rather than Theorem
\ref{thm:2.0}, but which are actually equivalent as we observed
above.
\begin{prop}\label{prop:formulaPsi}
The weak limit $\Psi_U(t,u)$ of $\Psi^N_U(t,u)$ as $N \to \infty$ exists
and is given by the formula
\begin{equation}\label{eqn:formulaPsi}
\Psi_U(t,u) = \frac{\bar{\Psi}_U\bigl(t,\tra (u)\bigr)}
{1-\rho\bigl(t,\tra (u)\bigr)}.
\end{equation}
\end{prop}
\begin{proof}
Since the convergence in Proposition \ref{prop:ConvPsiBar} is only in a weak
sense, we start by using Skorohod's theorem and then assume that
$\bar{\Psi}^N_U(t,u)$ converges almost surely to $\bar{\Psi}_U(t,u)$ on
$D([0,T], D(\R))$ by choosing  a proper probability space. In order to
simplify the notation, we still use the same name in the following. Then, by
\eqref{eqn:cor1} and \eqref{eqn:cor2}, for each function $f \in
C^{2}_0(\R_+^\circ)$, we can compute
\begin{align*}
\int_{\R_+^\circ} \Psi_U^N(t,u) f(u) du = & \sqrt{N} \Bigl( \tfrac1N
\sum_{x \in \Z} F\bigl(\trai \bigl(\tfrac{x}{N}\bigr) \bigr)
\bar{\eta}_t^N(x) - \int_{\R} F\bigl(\trai (v) \bigr) \rho(t,v) dv
\Bigr)\\
& + \frac1{\sqrt{N}}   \sum_{x \in \Z}  \Bigl( F\bigl(\tfrac1N
\trai^N(x) \bigr) - F\bigl(\trai \bigl(\tfrac{x}{N}\bigr) \bigr) \Bigr)
\bar{\eta}_t^N(x) =: S_1^N +S_2^N.
\end{align*}
Integration by parts and summation by parts yield
\begin{align*}
\int_{\R} F\bigl(\trai (v) \bigr) \rho(t,v) dv &= \int_{\R}
f\bigl(\trai (v) \bigr) (1- \rho(t,v))  dv \int_v^\infty \rho(t,w)dw,\\
\frac1N \sum_{x \in \Z} F\bigl(\trai \bigl(\tfrac{x}{N}\bigr) \bigr)
\bar{\eta}_t^N(x) &= \int_{\R}f\bigl(\trai (v) \bigr)(1- \rho(t,v))
\pi_t^N([v, \infty))dv.
\end{align*}
Therefore $S_1^N$ can be written as an integral with respect to
$\bar{\Psi}_U^N$, and with the help of Proposition \ref{prop:ConvPsiBar}
and a simple substitution $u = \trai(v)$, we have that
\begin{equation}\label{eqn:convfirst}
\lim_{N \to \infty}S_1^N = \lim_{N \to \infty}\int_{\R} f\left(\trai
(v) \right) (1- \rho(t,v)) \bar{\Psi}_U^N(t,v) dv =\int_{\R_+^\circ}
f(u) \bar{\Psi}_U(t,\tra (u)) du.
\end{equation}
In the following, we are going to show
\begin{equation}\label{conve-s2}
\lim_{N \to \infty}S_2^N  = \int_{\R} f(\trai (v) ) \rho(t,v)
\bar{\Psi}_U(t,v) dv \text{ a.\,s.}
\end{equation}
By Taylor's formula, it holds
\[
F\bigl(\tfrac1N \trai^N(x) \bigr) - F\bigl(\trai\bigl(\tfrac{x}
{N}\bigr) \bigr) = f\bigl(\trai\bigl(\tfrac{x}{N}\bigr)\bigr) \Bigl(
\tfrac{\trai^N(x)}{N} - \trai\bigl(\tfrac{x}{N}\bigr)\Bigr) +
f'(\zeta_{N,x}) \Bigl( \tfrac{\trai^N(x)}{N} - \trai\bigl(\tfrac{x}
{N}\bigr)\Bigr)^2
\]
with $\zeta_{N,x} \in \bigl[ \min \{ \frac1N \trai^N (x), \trai
(\frac{x}{N})\}, \max \{ \frac1N \trai^N (x), \trai (\frac{x}
{N})\}\bigr]$.

The above appearing term $\frac1N \trai^N(x) - \trai\bigl(\frac{x}{N}
\bigr)$ is basically given by the process $\bar{\Psi}^N_U$. We show this
using the asymmetry property (see Subsection 4.1 in \cite{FS}) which leads
to
\[
\trai \bigl(\tfrac{x}{N}\bigr) = \left|\tfrac{x}{N}\right| +
\int_{\frac{x}{N}}^\infty \rho(t,v)dv \quad \text{and} \quad \tfrac1N
\trai^N (x) = \left|\tfrac{x}{N}\right| + \pi^N_t\bigl([\tfrac{x}{N},
\infty) \bigr) - \tfrac1N \bar{\eta}_t^N(x).
\]
Thus, using these relations, we see that
\begin{equation}\label{estimate-s2}
\begin{split}
\Bigl| S_2^N -\int_{\R} f&\left(\trai (v) \right) \rho(t,v)
\bar{\Psi}_U(t,v) dv \Bigr| \le \\
& \phantom{+s} |E_1| + |E_2| + \Bigl| \tfrac1N \sum_{x \in \Z}
f\bigl(\trai (\tfrac{x}{N}) \bigr)\bigl(\bar{\Psi}_U^N \bigl(t,
\tfrac{x}{N} \bigr) - \bar{\Psi}_U \bigl(t, \tfrac{x}{N} \bigr)\bigr)
\bar{\eta}_t^N(x)\Bigr| \\
&+ \Bigl|\tfrac1N \sum_{x \in \Z} f \bigl(\trai \bigl(\tfrac{x}
{N}\bigr) \bigr) \bar{\Psi}_U (t,  \tfrac{x}{N} ) \bar{\eta}_t^N(x) -
\int_{\R} f(\trai (v) ) \rho(t,v) \bar{\Psi}_U(t,v) dv \Bigr|,
\end{split}
\end{equation}
where
\[
E_1=  \frac{-1}{N^{3/2}} \sum_{x \in \Z} f\bigl(\trai \bigl(\tfrac{x}
{N}\bigr) \bigr) \bar{\eta}_t^N(x) \quad \text{and}\quad E_2=\frac{1}
{N^{3/2}}\sum_{x \in \Z} \bar{\eta}_t^N(x) f'(\zeta_{N,x})  \Bigl(
\bar{\Psi}^N_U(t,u) - \tfrac{\bar{\eta}_t^N(x)}{\sqrt{N}}\Bigr)^2.
\]
Clearly, $E_1 \to  0$ a.\,s. because of the extra $\sqrt{N}$ in the
denominator. On the other hand, from Proposition \ref{prop:ConvPsiBar}, in
particular, the fact that the  limit $\bar{\Psi}_U(t,u)$ of
$\bar{\Psi}^N_U(t,u)$ is in $C([0,T],C(\R))$ a.\,s., we know that
\[
\sup_{t\in [0, T],N \in \N, v\in [-K, K]} |\bar{\Psi}^N_U(t,v)|< \infty
\text{ a.\,s.},
\]
which implies  $E_2 \to 0$ a.\,s. by recalling that $f \in
C_0^2(\R_+^\circ)$.

To conclude the proof of \eqref{conve-s2}, let us now reformulate a
result which follows from Proposition 4.2 of \cite{FS}. Under our
assumptions, for any function $g \in C_0(\R_+^\circ)$, as $N
\to \infty$,
$\frac1N \sum_{x \in \Z} g\bigl(\trai\bigl(\frac{x}{N}\bigr)\bigr)
\bar{\eta}_t^N(x)$ converges to  $\int_{\R} g \bigl(\trai
(v)\bigr)\rho(t,v) dv$ in probability. Applying this result for
$g(\cdot)= f(\cdot) \bar{\Psi}_U(t,\tra(\cdot))$ and using the
compactness of the support of $f$, we have that the last term on the
right hand side of \eqref{estimate-s2} converges to $0$ a.s. In the
end, using Proposition \ref{prop:ConvPsiBar} again and recalling
that we apply Skorohod's theorem, we have
\[
\lim_{N \to \infty}\sup_{t\in [0, T], v\in [-K, K]}
|\bar{\Psi}^N_U(t,v)-\bar{\Psi}_U(t,v) |=0 \text{ a.\,s.},
\]
and thus applying the above result for $g(\cdot)=f(\cdot)$, we also see
that the third  term on the right hand side of \eqref{estimate-s2} converges
to $0$ a.\,s.  So, the proof of \eqref{conve-s2} is
completed.

Finally, we substitute with $u = \trai(v)$ and therefore the limit for
$S_2^N$ is given by
\[
\int_{\R} f(\trai (v)) \rho(t,v) \bar{\Psi}_U(t,v) dv =
\int_{\R_+^\circ} f(u) \frac{\rho(t, \tra (u))}{1-\rho(t,\tra (u))}
\bar{\Psi}_U(t,\tra (u))du,
\]
which completes the proof \eqref{eqn:formulaPsi} with the help of
\eqref{eqn:convfirst}.
\end{proof}
Now that we have a formula for the limit process the next step is to
identify the corresponding SPDE. A direct computation with $\psi_U(t,u) =
\zeta^+_{\rho(t,\cdot)}\bigl(\trans (u)\bigr)$ leads to the following
lemma, recall that $\rho_U(t,u) = - \psi_U'(t,u)$ is defined in Theorem
\ref{thm:2.1}.
\begin{lem}\label{lem:Rho}
We have that
\[
\rho_U(t,u)=\frac{\rho(t, \tra (u))}{1-\rho(t,\tra (u))} \quad
\text{and}\quad 1+\rho_U(t,u) = \frac{1}{1-\rho(t,\tra (u))}.
\]
\end{lem}
We are at the position to give the proof of Theorem \ref{thm:2.1}. We prove
that the limit $\Psi_U(t,u)$ of $\Psi_U^N(t,u)$ obtained in Proposition
\ref{prop:formulaPsi} satisfies the SPDE \eqref{eq:SPDE-1}.
\begin{proof}[Proof of Theorem \ref{thm:2.1}]
Fix a function $f \in C^{1,2}_0([0,T] \times \R_+^\circ)$ and consider the
process $\Psi_U(t,u)$ tested with $f$. Then, by the representation formula
\eqref{eqn:formulaPsi} combined with Lemma \ref{lem:Rho} and the
substitution $v= \tra (u)$, we get
\begin{align*}
\int_{\R_+^\circ} f(t,u) \Psi_U(t,u) du
& =\int_{\R_+^\circ} f(t,u) \bigl(1+\rho_U(t,u)\bigr) \bar{\Psi}_U\bigl(
t,\tra (u) \bigr) du \\
& =\int_{\R} f\bigl(t,\trai (v) \bigr) \bar{\Psi}_U(t,v) dv = \lan
\bar{\Psi}_U(t), f(t)
\circ \trai \ran.
\end{align*}
Since $f(t) \circ \trai  \in C^{1,2}_0([0,T] \times \R)$,
\eqref{eqn:spdePsiBar} rewrites the right hand side as
\begin{align*}
\lan \bar{\Psi}_{U,0},f(0) \circ \zeta_0 \ran + &\int_0^t \lan
\bar{\Psi}_U(s), (f(s)\circ \zeta_s)'' - \a \bigl((1-2\rho(s)) f(s)
\circ \zeta_s \bigr)' + \partial_s \bigl(f(s) \circ \zeta_s \bigr)
\ran ds\\
+& \int_0^t\int_{\R} f(s)\circ\zeta_s(v) \sqrt{2\rho(s, v)(1-\rho(s,v))}
W(dsdv).
\end{align*}
Thus, for the initial condition, we get analogue to the above
\[
\lan \bar{\Psi}_{U,0},f(0) \circ \zeta_0 \ran = \int_{\R_+^\circ} f(0,u)
\Psi_{U,0}(u) du.
\]
Let us consider the drift term.  The relation $\zeta_s'(v) = 1-\rho(s,v)$ implies
\begin{gather*}
\bigl(f(s,\zeta_s(v))\bigr)'' - \a \bigl((1-2\rho(s,v))
f(s,\zeta_s(v))\bigr)' = (1-\rho(s,v))^2 f''(s,\zeta_s(v))\qquad\qquad\\
\qquad - \Bigl(\rho'(s,v) + \a (1-2\rho(s,v)) (1-\rho(s,v)) \Bigr)
f'(s,\zeta_s(v)) + 2\a \rho'(s,v) f(s,\zeta_s(v)),
\end{gather*}
and by \eqref{eqn:burgers},
\[
\partial_s \bigl(f(s, \zeta_s(v) \bigr) = \partial_s f(s) \circ
\zeta_s(v)  -f'(s,  \zeta_s(v))  \bigl( \rho'(s,v) + \a  \rho(s,v)
(1-\rho(s,v))\bigr).
\]
These yield that the drift term is equal to
\begin{align*}
\int_0^t \big\lan \bar{\Psi}_U(s),   (1-\rho(s))^2 f''(s)\circ \zeta_s
&- \bigl(2 \rho'(s)  + \a (1-\rho(s))^2\bigr) f'(s)\circ \zeta_s \\
&\quad+ 2 \a \rho'(s) f(s) \circ \zeta_s + \partial_sf(s)\circ\zeta_s
\big\ran ds.
\end{align*}
With the substitution $u=\trai(v)$ and \eqref{eqn:formulaPsi} combined with
Lemma \ref{lem:Rho}, we come back to an expression in $\Psi_U(t)$:
\[
\int_0^t \int_{\R_+^\circ} \Psi_U(s,u) \Biggl(\Biggl(\frac{f'(s,u)-\a
f(s,u)} {(1+\rho_U(s,u))^2}\Biggr)' + \partial_s f(s,u)\Biggr)du ds.\]
The last task is to check the noise term. We consider the quadratic
variation of the in the above appearing stochastic integral, which is given
by
\begin{align*}
\int_0^t \int_{\R} f^2(s,\zeta_s(v)) 2\rho(s,v) (1- \rho(s,v)) dv ds
= &\int_0^t \int_{\R_+^\circ} f^2(s,u) 2\rho(s, \zeta_s^{-1}(u)) du ds \\
=& \int_0^t \int_{\R_+^\circ} f^2(s,u) \frac{2\rho_U(s,u)}{1+\rho_U(s,u)}
du ds.
\end{align*}
This proves \eqref{eq:3.13} with a suitably taken space-time white noise
$\dot{W}(t,u)$ on $[0,T]\times\R_+^\circ$ (which is different from that in
Proposition \ref{prop:ConvPsiBar}) as in Lemma \ref{lem:4.17} below.
\end{proof}
For the proof of Theorem \ref{thm:2.1}, the uniqueness of the solution to
the SPDE \eqref{eq:SPDE-1} in the limit was unnecessary.  Nevertheless, we
show that uniqueness holds under the condition \eqref{eq:LB} stated below.
\begin{lem} \label{lem:3.4}
The relation \eqref{eqn:formulaPsi} for $\bar{\Psi}_U(t)$ and $\Psi_U(t)$
translates to
\begin{equation}\label{equiv-norm-1}
\|\bar{\Psi}_U(t)\|_{L^2_r(\R)}^2 =\int_{\R_+^\circ} \Psi_U^2(t,u)
\frac{e^{-2r|u-\psi_U(t, u)|}}{1+\rho_U(t, u)} du.
\end{equation}
If in addition $\rho_U(t, u)$ satisfies the condition
\begin{equation} \label{eq:LB}
c := \inf_{t\in [0,T], u\in (0,1]} u \rho_U(t, u) >0,
\end{equation}
then for every $r>0$, there exists $C_r>0$ such that
\begin{equation} \label{eq:UB}
\|\bar{\Psi}_U(t)\|_{L^2_r(\R)} \le C_r \|\Psi_U(t)\|_{\tilde{L}^2_{r(c\a\wedge 1)}(\R_+^\circ)}.
\end{equation}
In particular, $\Psi_U(t)\in
\tilde{L}_e^2(\R_+^\circ)$ implies $\bar{\Psi}_U(t) \in L_e^2(\R)$ and then
the solution $\Psi_U$ of the SPDE \eqref{eq:SPDE-1} is unique in the class
$C([0,T],C(\R_+^\circ)) \cap C([0,T],\tilde{L}_e^2(\R_+^\circ))$.
\end{lem}
\begin{proof}
By a change of variables, the left hand side of \eqref{equiv-norm-1} can be
rewritten as
\[
\int_{\R_+^\circ} \bar{\Psi}_U^2\bigl(t, \zeta_t^{-1}(u) \bigr)e^{-2r|
\zeta_t^{-1}(u)|} \bigl(\zeta_t^{-1}(u) \bigr)' du.
\]
It is easily seen that this integral is equal to the right hand side of
\eqref{equiv-norm-1} by applying Proposition \ref{prop:formulaPsi}, Lemma
\ref{lem:Rho} and recalling that $\trai^{-1} (u)=u -\psi_U(t, u)$ and
$(\trai^{-1} (u))'= 1+ \rho_U(t, u)$. This proves \eqref{equiv-norm-1}. To
show \eqref{eq:UB}, first note that $\psi_U(t) \in X_U$ behaves like
\[
 \frac{e^{-2r|u- \psi_U(t,u)|}}{1-\psi'_U(t,u)} \asymp e^{-2ru},
\]
for large enough $u$ uniformly in $t\in [0,T]$. On the other hand, condition
\eqref{eq:LB} implies $\rho_U(t, u) \geq \bar{\rho}_U(u) := cu^{-1}$ on
$(0, 1]$ and therefore
\[
u-\psi_U(t, u) \leq u-\bar{\psi}_U(u) \, (<0),
\]
near $0$, where $\bar{\psi}_U(u) = -c \log u$. This results in a behavior
like
\[
\frac{e^{-2r|u-\psi_U(t, u)|}}{1+\rho_U(t, u) } \leq
\frac{e^{-2r|u-\bar{\psi}_U(u)|}}{1+\bar{\rho}_U(u) }
=\frac{u^{2rc}e^{2ru}}{1+cu^{-1}} \asymp \frac1c u^{2rc+1},
\]
near $0$. Applying these estimates to \eqref{equiv-norm-1} yields
\eqref{eq:UB}. Finally, transform the solution $\Psi_U(t)$ of the SPDE
\eqref{eq:SPDE-1} in the class $C([0,T],C(\R_+^\circ)) \cap C([0,T],
\tilde{L}_e^2(\R_+^\circ))$ into $\bar{\Psi}_U(t)$ by
\eqref{eqn:formulaPsi}. By \eqref{eq:UB} $\bar{\Psi}_U(t)$ is a solution
of the SPDE \eqref{eq:SPDE-Prop} in the class $C([0,T],C(\R)) \cap
C([0,T],L_e^2(\R))$. Since $\bar{\Psi}_U(t)$ is uniquely determined in this
class, uniqueness for $\Psi_U(t)$ follows.
\end{proof}
In the final part of this section, we give an example of a class of initial
values $\rho_U(0,u)$ for which the condition \eqref{eq:LB} is satisfied
along the time evolution $\rho_U(t,u)$. We first prepare a comparison
theorem for solutions of PDE \eqref{eqn:burgers}.

\begin{lem}  \label{lem:comparisonWASEP}
If two initial values of \eqref{eqn:burgers} satisfy $\rho^{(1)}(0,v) \le
\rho^{(2)}(0,v), v\in \R$, then the corresponding solutions satisfy
$\rho^{(1)}(t,v) \le \rho^{(2)}(t,v)$ for every $t>0$ and $v \in \R$.
\end{lem}
\begin{proof}
This is immediate by applying the Hopf-Cole transformation. Or since the
underlying microscopic system, the weakly asymmetric simple exclusion
process on $\Z$, is attractive, by passing to the hydrodynamic limit we see
the conclusion for the limit equation \eqref{eqn:burgers}.
\end{proof}
Let $\rho_U^\infty(u)$ and $\rho^\infty(v;C)$ be stationary solutions of
the PDEs \eqref{eq:3.14} and \eqref{eqn:burgers}, respectively, with
explicit formulas
\begin{equation}\label{eq:3.rho_infty}
\rho_U^\infty(u) := \frac1{e^{\a u}-1}
\quad \text{and} \quad
\rho^\infty(v;C) := \frac{C}{e^{\a v}+C},\  \text{for every}\ C>0.
\end{equation}
Note that $\rho^\infty(v; C)$ are shifts of $\rho^\infty(v;1)$ and further recall that $\rho(v)=\Phi_U(\psi_U)(v)$ is
defined in Lemma \ref{lem:3.2-b}.
\begin{lem}\label{lem:example1}
Assume that the derivative $\rho_U(u)=-\psi'_U(u)$ of
$\psi_U \in X_U$ satisfies
\begin{equation}\label{exm:ass1}
C_2\rho_U^\infty(u) \le \rho_U(u) \le C_1\rho_U^\infty(u),
\end{equation}
for some $C_1 \ge C_2>0$ and
\begin{equation}\label{exm:ass2}
\limsup_{u \downarrow 0} |\rho_U(u) - \rho_U^\infty(u)| < \infty.
\end{equation}
Then, there exist $\tilde{C}_1 \ge \tilde{C}_2 > 0$ such that
\begin{equation}\label{exm:ass3}
\rho^\infty(v;\tilde{C}_2) \le \rho(v) \le \rho^\infty(v;\tilde{C}_1).
\end{equation}
\end{lem}
\begin{proof}
Recall the definitions of $\Phi_U$ and $G_{\psi}$ given in the proof of
Lemma \ref{lem:3.2-b} and note that what we only need to prove
\[
C_2' e^{-\a v} \le \rho_U((G_{\psi})^{-1}(v)) \le C_1' e^{-\a v},
\]
for some $C_1' \ge C_2' > 0$. Under condition (\ref{exm:ass1}),
we can  reduce this  to show that there exist $D_1 \ge D_2>0$ such that
\[
D_2 e^{\a v} \le e^ { \a (G_{\psi})^{-1}(v)}-1 \le D_1 e^{\a v}.
\]
The condition (\ref{exm:ass2}) implies $A=\sup_{0<u \le 1}|\rho_U(u) -
\rho_U^\infty(u)|<\infty$ and therefore
\[
\psi_U(u)= \int^1_u \rho_U(u')du' +\psi_U(1) \le -\frac{1}{\a} \log(1-
e^{-\a u}) + A+\psi_U(1), \ \text{for any}\  0 < u \le 1.
\]
Similarly, for any $0< u \le 1$, $\psi_U(u) \ge -\frac{1}{\a} \log(1-
e^{-\a u}) - A +\psi_U(1)$.
Denote $A+\psi_U(1)= \tilde{C}_1$ and $-A+\psi_U(1)=\tilde{C}_2$. Then
\[
{\a} u+  \log(1-e^{-\a u}) -{\a} \tilde{C}_1 \le {\a} G_{\psi}(u)
 \le {\a} u +  \log(1-e^{-\a u}) - {\a}\tilde{C}_2
\]
and by taking the exponential, we obtain that for any $v$ satisfying
$(G_{\psi})^{-1} (v) \le 1$,
\[
\bigl( e^{\a (G_{\psi})^{-1} (v) }- 1 \bigr) e^{- \a \tilde{C}_1} \le e^
{\a v} \le \bigl( e^{\a (G_{\psi})^{-1} (v)} - 1 \bigr) e^{- \a
\tilde{C}_2}.
\]
For $v$ satisfying $(G_{\psi})^{-1} (v) \ge 1$, we apply the same argument
as above with $\psi_U(u)= \int^{\infty}_u \rho_U(u')du'$ to obtain the
inequality
\[
e^{\a (G_{\psi})^{-1} (v)} \bigl( 1- e^{- \a (G_{\psi})^{-1} (v) }\bigr)^C
\le e^ {\a v} \le e^{\a (G_{\psi})^{-1} (v) }\bigl( 1- e^{- \a
(G_{\psi})^{-1} (v)} \bigr)^{1/C},
\]
where $C= \max\{C_1, 1/C_2\} \ge 1$.
Then, it is obvious that $\bigl(1- e^{- \a (G_{\psi})^{-1} (v)
}\bigr)^{C-1} \ge (1- e^{- \a})^{C-1}$ and $\bigl( 1- e^{- \a
(G_{\psi})^{-1} (v)} \bigr)^{1/C -1 } \le (1- e^{- \a})^{1/C -1 }$ for
$v$ satisfying $(G_{\psi})^{-1} (v) \ge 1$. In the end, for any $v \in \R$,
we have
\[
\min \bigl\{e^{- \a \tilde{C}_1}, (1- e^{- \a})^{C-1} \bigr\}  \le
{e^ {\a v}}\bigl( e^{\a (G_{\psi})^{-1} (v) }- 1  \bigr)^{-1} \le
\max \bigl\{e^{- \a \tilde{C}_2}, (1- e^{- \a})^{1/C-1} \bigr\},
\]
which concludes the proof.
\end{proof}
\begin{lem}\label{lem:example2}
Assume that $\rho(\cdot) \in Y_U$ satisfies the condition \eqref{exm:ass3}
for some $C_1 \ge C_2 > 0$ in place of $\tilde{C}_1 \ge \tilde{C}_2>0$.
Then, $\rho_U(u):=- (\Psi_U(\rho(\cdot))(u))'$ satisfies
\[
\frac{C_2}{C_1} \rho_U^\infty(u)\le \rho_U(u) \le \frac{C_1}{C_2}
\rho_U^\infty(u),
\]
where $\Psi_U:Y_U\to X_U$ is the inverse map of $\Phi_U$, see Proposition
4.4 of \cite{FS}.
\end{lem}
\begin{proof}
By definition $\rho_U(u)= \bigl(1-\rho((\zeta^{-}_{\rho})^{-1}
(u))\bigr)^{-1} -1$ and
\[
\bigl(C_1 e^{-\a (\zeta^{-}_{\rho})^{-1}(u)}+1 \bigr)^{-1} \le 1-
\rho((\zeta^{-}_{\rho})^{-1}(u)) \le \bigl(C_2 e^{-\a
(\zeta^{-}_{\rho})^{-1}(u)}+1 \bigr)^{-1}
\]
holds by assumption. Then, it is easy to see that
\[
C_2 e^{-\a (\zeta^{-}_{\rho})^{-1}(u)} \le \rho_U(u) \le C_1 e^{-\a
(\zeta^{-}_{\rho})^{-1}(u)}.
\]
On the other hand, since
$\zeta^{-}_{\rho}(v)=\int^v_{\infty}(1-\rho(w))dw$
and  $\int^v_{\infty}\frac{1}{C e^{-\a w}+1} dw=\frac{1}{\a} \log
(\frac{C+e^ {\a v}}{C})$,
\[
\frac{1}{\a} \log \left(\frac{C_1+e^ {\a (\zeta^{-}_{\rho})^{-1}(u)}}
{C_1}\right) \le u  \le \frac{1}{\a} \log \left(\frac{C_2+e^ {\a
(\zeta^{-}_{\rho})^{-1}(u)}}{C_2}\right)
\]
holds. Thus, we obtain
\[
{C_1}^{-1}\rho_U^\infty(u) \le e^{-\a (\zeta^{-}_{\rho})^{-1}(u)} \le
{C_2}^{-1}\rho_U^\infty(u),
\]
which concludes the proof.
\end{proof}
\begin{prop}\label{porp:example}
Assume that the derivative $\rho_U(0,u)=-\psi_U'(0,u)$ of
$\psi_U(0,\cdot) \in X_U$ satisfies two conditions \eqref{exm:ass1} and
\eqref{exm:ass2} in Lemma \ref{lem:example1} with $\rho_U(u)$ replaced by
$\rho_U(0,u)$.  Then, there exist constants $\tilde{C}_1>\tilde{C}_2 >0$
such that for any $t>0$, the solution $\rho_U(t,u)$ of the PDE
\eqref{eq:3.14} below satisfies
\begin{equation}\label{eq:3.19}
\tilde{C}_1 \rho_U^\infty(u) \le \rho_U(t,u) \le \tilde{C}_2
\rho_U^\infty(u).
\end{equation}
In particular the lower bound in \eqref{eq:3.19} implies the condition
\eqref{eq:LB}.
\end{prop}
\begin{proof}
First, note that the function $\rho^\infty(v;C)$ is a stationary solution
of
the PDE \eqref{eqn:burgers} for any $C>0$. Then, with Lemma
\ref{lem:comparisonWASEP}, the conclusion follows from Lemmas
\ref{lem:example1}
and \ref{lem:example2}.
\end{proof}
\begin{rem} \label{rem:3.1}
Under the equilibrium situation, that is, for
$\rho_U^\infty(u) = \lim_{t\to\infty} \rho_U(t,u)$,
$\zeta^{-1}(u) = \lim_{t\to\infty} \zeta^{-1}(t,u)$, $u\in \R_+^\circ$
and
$\rho^\infty(v;1) = \lim_{t\to\infty} \rho(t,v)$,
$\zeta(v) = \lim_{t\to\infty} \zeta(t,v)$, $v\in \R$,
we have explicit formulas:
\[
\zeta^{-1}(u) = \frac1{\a} \log( e^{\a u} -1) \quad \text{and} \quad
\zeta(v) = \frac1{\a} \log( e^{\a v} +1).
\]
From this, we see that the norm $|\bar{\Psi}|_{L_r^2(\R)}$ is equivalent to
$|\Psi|_{\tilde{L}_r^2(\R_+^\circ)}$, if $\bar{\Psi}$ and $\Psi$ are
related
with each other by the relation stated in Proposition \ref{prop:formulaPsi}:
$\Psi(u) = \bar{\Psi}(\zeta^{-1}(u))/(1-\rho(\zeta^{-1}(u)))$.  This explains
the reason for considering the norm $|\Psi|_{\tilde{L}_r^2(\R_+^\circ)}$.
\end{rem}
\begin{rem}
Similarly to Lemma \ref{lem:comparisonWASEP}, the attractiveness of the
underlying weakly asymmetric zero-range process with stochastic reservoir
leads to a comparison theorem for $\rho_U(t,u)$.  More precisely, the
function $\rho_U(t,u) =
- \psi_U'(t,u)$, defined from a solution $\psi_U(t,u)$ of the PDE in
the statement (1) of Section 2, solves the nonlinear PDE:
\begin{equation} \label{eq:3.14}
\partial_t\rho_U = \Bigl(\frac{\rho_U}{1+\rho_U}\Bigr)''
 + \a \Bigl(\frac{\rho_U}{1+\rho_U}\Bigr)',  \quad u\in \R_+^\circ.
\end{equation}

If two initial values of \eqref{eq:3.14} satisfy
$0<\rho_U^{(1)}(0,u) \le \rho_U^{(2)}(0,u), u\in \R_+^\circ$, then the
corresponding solutions satisfy
$0<\rho_U^{(1)}(t,u) \le \rho_U^{(2)}(t,u)$ for every $t>0$ and
$u\in \R_+^\circ$.
\end{rem}


\section{Proof of Theorem \ref{thm:2.2}}\label{section-4}

Let $q_t := q_t^\e=(q_i(t))_{i\in\N}$ be the Markov process on
$\mathcal{Q}$ introduced in Section 2 and
let $\eta_t = (\eta_t(x))_{x\in\N}\in \{0,1\}^\N$ be the height
differences of the height function $\psi_{q_t}$
determined from $q_t$.  The process $\eta_t$ is also defined by $\eta_t(x)
= \sharp\{i; q_i(t)=x\}$,
and set $\eta_t(0)=\infty$ for convenience.  As shown in Section 5.1 of
\cite{FS},
the process $\eta_t$ is a weakly asymmetric simple exclusion process on $\N$
with a weakly
asymmetric stochastic reservoir at $\{0\}$ and its generator is given at p.\
353 in \cite{FS}.
 Here again, we apply the Hopf-Cole transformation for $\eta_t$ at the
microscopic level.

Section \ref{subsection4.1} essentially reduces the proof of Theorem
\ref{thm:2.2} to a fluctuation
result for a process  on the whole lattice $\Z$, which is related to
the Hopf-Cole transformed process $\zeta_t^N$ and is introduced mainly to
avoid the boundary condition at $\{0\}$ by
a simple  transformation, see Proposition \ref{prop:4.4}. The proof of
Theorem \ref{thm:2.2} is formulated in
Section \ref{subsection4.1} based on  Proposition \ref{prop:4.4}, whose
proof is given  in Section \ref{subsection4.3}.
\subsection{Fluctuations for the Hopf-Cole Transformed Process}\label{subsection4.1}
Let $\eta_t^N = (\eta_t^N(x))_{x\in\N} := (\eta_{N^2t}(x))_{x\in\N}$ be
the weakly asymmetric
simple exclusion process speeded up by the factor $N^2$ in time with a
stochastic reservoir at $\{0\}$
and consider its microscopic Hopf-Cole transformation $\zeta_t^N =
(\zeta_t^N(x))_{x\in\N}$ defined by
\[
\zeta_t^N(x):=\exp \Bigl\{- (\log \e) \sum_{y=x}^{\infty}\eta_t^N(y)
\Bigr\}, \quad \e=\e_R(N).
\]
Its interpolation $\tilde{\zeta}^N(t,u), u\in \R_+$ with the proper
scaling in space is given by
\begin{equation} \label{eq:4.inter}
\tilde{\zeta}^N(t,u):=\exp \Bigl\{- (\log \e)
\Bigl(\sum_{y=[Nu]+1}^{\infty}\eta_t^N(y)+ 1_{\{u\ge 1/N\}}
([Nu]+1-Nu)\eta_t^N([Nu])\Bigr) \Bigr\}.
\end{equation}
It is clear that for each $t\geq 0$, $\tilde{\zeta}^N(t,
\cdot)$ is a $C(\R_+)$-valued process.
Theorem 5.2 of \cite{FS} states that, if the scaled empirical measure
$\pi_0^N$ of $\eta_0^N$ converges
to $\rho_0(v)dv$ in probability as $N\to\infty$ with $\rho_0(v)$
satisfying $\rho_0\in C(\R_+,[0,1])$
and $\int_0^\infty \rho_0(v)dv<\infty$, then $\tilde{\zeta}^N(t,u)$
converges to $\om(t,u)$
in probability, which is a unique bounded classical solution of the
following linear diffusion equation:
\begin{equation} \label{eq:3.1}
\left\{
\begin{aligned}
\partial_t\om & = \om'' + \b \, \om',  \quad u\in \R_+,  \\
\om(0,u) & = \exp \{ \b \int_{u}^{\infty} \rho_0(v)dv \},  \quad u\in
\R_+,  \\
2 \om'(t,0) & + \b \om(t,0) =0, \ \mbox{and}\quad  \om(t,\infty) =1,
\quad t>0.
\end{aligned}
\right.
\end{equation}
Instead of immediately considering the fluctuations of
$\tilde{\zeta}^N(t,u)$ around its limit,
the goal is to avoid the mixed boundary condition above.
Therefore the next paragraph reduces the problem to another asymptotic
problem on $\Z$, formulated in Proposition \ref{prop:4.4} below.

At first recall from Section 5.3.3 of \cite{FS} that $\zeta_t^N =
(\zeta_t^N(x))_{x\in\N}$ satisfies the stochastic differential equation
(SDE):
\[
d\zeta_t^N(x) = N^2 \bigl( \e \zeta_t^N(x-1)-(\e+1)\zeta_t^N(x) +
\zeta_t^N(x+1) \bigr)dt + dM_t^N(x),\ x \in \N,
\]
where $\zeta_t^N(0) := \e^{-1}\zeta_t^N(2)$ and $(M_t^N(x))_{x\in\N}$ are
martingales  with quadratic variations and covariations given as follows:
\begin{equation} \label{eq:4.SDE-1-1}
\begin{split}
& \frac{d}{dt} \lan M^N(x)\ran_t = \zeta_t^N(x)^2 \left\{ a_N c_+(x-1,
\eta_t^N) + b_N c_-(x-1, \eta_t^N)\right\},\   x \ge 2,  \\
& \frac{d}{dt} \lan M^N(1)\ran_t = \zeta_t^N(1)^2 \left\{ a_N
1_{\{\eta_t^N(1)=0\}} + b_N 1_{\{\eta_t^N(1)=1\}} \right\},\\
& \lan M^N(x), M^N(y)\ran_t =0, \quad 1 \le x \not= y.
\end{split}
\end{equation}
Here $ c_+(x, \eta)=1_{\{\eta(x)=1, \eta(x+1)=0\}}$, $ c_-(x,
\eta)=1_{\{\eta(x)=0, \eta(x+1)=1\}}$,
$a_N=N^2(1-\e)^2/\e$ and $b_N=N^2(1-\e)^2$.  Note that
$\lim_{N\to\infty}a_N = \lim_{N\to\infty}b_N = \b^2$
and, in Lemma 5.6 of \cite{FS}, $c_{\pm}(x-1, \eta_t^N)$ are reversed.

Instead of dealing with the boundary condition $\zeta_t^N(0) =
\e^{-1}\zeta_t^N(2)$ for $x=0$, a simple transformation for $\zeta_t^N$ and
its extension to $\Z$ makes the analysis easier.
\begin{lem} \label{lem:4.3}
Let us consider the process $\bar{\zeta}_t^N
=(\bar{\zeta}_t^N(x))_{x\in\Z}$ defined by
\[
\bar{\zeta}_t^N(x) = \exp\left\{ -(\log\e)x/2 \right\} \zeta_t^N(x)
\]
for $x\ge 1$ and $\bar{\zeta}_t^N(x) = \bar{\zeta}_t^N(2-x)$ for $x\le
0$. Then, $\bar{\zeta}_t^N(x)$ satisfies the SDE:
\begin{equation} \label{eq:SDE-4.2}
d \bar{\zeta}_t^N(x) = N^2 \e^{1/2} \De \bar{\zeta}_t^N(x) dt + N^2
\bigl(2\e^{1/2}- (\e+1)\bigr)\bar{\zeta}_t^N(x) dt + d\bar{M}_t^N(x),
\end{equation}
on the whole lattice space $\Z$, where $\bar{M}_t^N(x)
=e^{-(\log\e)x/2}M_t^N(x)$ for $x\ge 1$, $\bar{M}_t^N(x) =\bar{M}_t^N(2-
x)$ for $x\le 0$ and $\De\zeta(x) = \zeta(x-1) -2\zeta(x) +\zeta(x+1)$
for $x\in\Z$.
\end{lem}
The proof of this lemma is straightforward and omitted.
The above transformation motivates the corresponding one for $\om(t,u)$, the
solution of  \eqref{eq:3.1}. In view of the scaling in
$\e $, it is natural to set $\bar{\om}(t, u):= e^{\b |u|/2}\om(t,|u|)$ and
then, parallel to  \eqref{eq:SDE-4.2},
to introduce its discretized  equations with initial values $\bar{\om}_0^N(x)= e^{-(\log \e) |x|/2} \om \bigl(0, \left|\tfrac{x}
{N}\right|\bigr)$:
\begin{equation}\label{eq:4.24}
d\bar{\om}_t^N(x)=N^2\e^{1/2}\De\bar{\om}_t^N(x)dt+
N^2\bigl(2\e^{1/2}-(\e
+1)\bigr)\bar{\om}_t^N(x) dt,\quad x\in \mathbb{Z}.\\
\end{equation}
It is known that the linear interpolation $\bar{\om}^N(t,u)$ of
$\bigl(\bar{\om}_t^N(x)\bigr)_{x\in \Z}$ converges to $\bar{\om}(t, u)$, see p.\ 214  in \cite{CY} or \cite{RM}.
More precisely, we have that
\begin{equation}\label{Pro4.3-Proof-2}
\lim_{N \to \infty}\sup_{t\in [0, T]}\sup_{u\in
[-K,K]}\sqrt{N}|\bar{\om}^N(t,u) -\bar{\om}(t,u)|=0.
\end{equation}
\begin{lem}\label{lem-barphi}
The process $\bar{\Phi}_t^N(x)$ defined by
\begin{equation}\label{eq:4.19}
\bar{\Phi}_t^N(x):=\sqrt{N}\left(\bar{\zeta}_t^N(x)-\bar{\om}_t^N(x) \right), \ x\in \mathbb{Z},
\end{equation}
satisfies the following SDE:  
\begin{equation}\label{barphi-1}
d \bar{\Phi}_t^N(x) = N^2 \e^{1/2} \De \bar{\Phi}_t^N(x) dt
+ N^2 \left(2\e^{1/2}- (\e+1)\right)\bar{\Phi}_t^N(x) dt +
\sqrt{N} d\bar{M}_t^N(x),
\end{equation}
which can  be represented in its mild form:
\begin{equation}\label{barphi-2}
\bar{\Phi}_t^N(x) =\sum_{y \in \mathbb{Z}}
p^N(t,x,y)e^{c_Nt}\bar{\Phi}_0^N(y)+
\int_0^t\sum_{y \in \mathbb{Z}} p^N(t-s,x,y)e^{c_N(t-s)}\sqrt{N}
d\bar{M}_s^N(y).
\end{equation}
Here $p^N(t,x,y)= p(N^2\e^{1/2}t, x-y)$ and
$p(t, x)$ is the (fundamental) solution of 
\begin{equation}\label{heatker}
\partial_tp(t,x)=\De p(t,x),\quad x\in \mathbb{Z},\quad \text{with}
\quad p(0,x)=\de _0(x),
\end{equation}
 and $c_N := N^2(2\e^{1/2} - (\e+1)) = -N^2(\e^{1/2}-1)^2$ behaves like
$c_N \sim -\b^2/4$.
\end{lem}
The SDE \eqref{barphi-1} is an immediate consequence from
\eqref{eq:SDE-4.2} and \eqref{eq:4.24}.
It is also  easy to obtain \eqref{barphi-2}. In fact, it is enough to apply
integration by parts to the process
$\{p^N(t-s, x,y)e^{c_N(t-s)}\bar{\Phi}_t(y)\}_{s\in [0, t]}$ for each
$t> 0$ and then integrate both sides from $0$ to $t$.

Now let us consider the linear interpolation of $\bar{\Phi}^N_t(x)$. More
precisely, we deal with the following process with values in
$D([0, T],C(\mathbb{R}))$:
\begin{equation}\label{DefLinInterPhi}
\bar{\Phi}^N(t, u) := ([Nu] +1 -Nu) \bar{\Phi}^N_t([Nu]) +(Nu- [Nu])
\bar{\Phi}^N_t([Nu] +1).
\end{equation}
The next subsection is devoted to prove the following proposition.

\begin{prop} \label{prop:4.4}
Suppose Assumption 2 is satisfied. Then, as $N \to \infty$, the transformed process ${\bar \Phi}^N(t,u)$
converges weakly to ${\bar{\Phi}}(t,u)$ on the space $D([0,T],C(\R))$.
Moreover, the limit ${\bar {\Phi}}(t,u)$ is in $C([0,T],C(\R))$
(a.\,s.) and it is a solution of the following SPDE:
\begin{equation} \label{eq:SPDE-4.3}
\left\{
\begin{aligned}
\partial_t{\bar \Phi}(t,u) & = {\bar{\Phi}}''(t,u)
- \tfrac{\b^2}4 {\bar \Phi}(t,u) \\
& \qquad  + e^{\b |u|/2} \b \,
\om(t,|u|)\sqrt{2\rho_R(t,|u|)(1-\rho_R(t,|u|))} \dot{\bar{W}}(t,u),
\quad u \in \R, \\
{\bar \Phi}(0,u) & = {\bar \Phi}_0(u),
\end{aligned}
\right.
\end{equation}
where $\bar{W}$ is a $Q$-cylindrical Brownian motion on $L^2(\mathbb{R})$
with the following covariance: for any test functions $\phi$ and $\psi$ on
$\R$,
\begin{equation}\label{Q-Wiener}
E[{\bar{W}}(t, \phi){\bar{W}}(t, \psi)] =s\wedge t \
\lan \phi, Q\psi \ran
\end{equation}
with $Q \psi(u)=\psi(u) +\psi(-u)$, and $\lan \cdot , \cdot \ran$ denotes
the inner product of $L^2(\R)$.
Furthermore, if  ${\bar \Phi}_0 \in\cap_{r>\b/2}L_r^2(\mathbb{R})$ then
there exists a unique weak solution ${{\bar{\Phi}}(t,\cdot)}$ in $C([0,T],
\cap_{r>\frac{\b}{2}}L_r^2(\mathbb{R}))$.
\end{prop}
\begin{rem}
The $Q$-cylindrical Brownian motion $\bar{W}$ can be easily constructed
based on a Brownian sheet on
$[0, \infty)\times \mathbb{R}_+$.
\end{rem}
The weak  and mild solutions of \eqref{eq:SPDE-4.3} are defined in similar
ways to \eqref{eq:SPDE-2}:
$\bar{\Phi}(t,u)$ is said to be  a  weak solution of the SPDE
\eqref{eq:SPDE-4.3} with  initial value
$\bar{\Phi}_0 \in \cap_{r>\b/2}L_r^2(\mathbb{R})$ if  $\bar{\Phi}\in
C([0,T],C(\R))
\cap C([0, T], \cap_{r>\b/2}L_r^2(\mathbb{R}))$ (a.\,s.) and for every
function $f \in C^{1, 2}_0([0,T]\times\R)$,
\begin{equation}\label{eq:4.10}
\begin{split}
&  \lan \bar{\Phi}(t),f(t)\ran = \lan\bar{\Phi}_0,f(0)\ran + \int_0^t
\lan \bar{\Phi}(s),f''(s)- \tfrac{\b^2}4 f(s) +\partial_s f\ran ds\\
& \quad + \int_0^t \int_{\R} f(s,u) e^{\b |u|/2} \b \,
\om(s,|u|)\sqrt{2\rho_R(s,|u|)(1-\rho_R(s,|u|))} \bar{W}(dsdu)
\text{ a.\,s.}
\end{split}
\end{equation}
In particular, from its mild form
\begin{align*}
&\bar{\Phi}(t,u)=  \int_\R \frac1{\sqrt{4\pi t}} e^{\bigl\{
-\frac{\b^2}4 t -\frac{(u-v)^2}{4t} \bigr\}} \bar{\Phi}_0(v)dv \\
&\quad+ \int_0^t \int_{\R}\frac1{\sqrt{4\pi t}} e^{
\bigl\{-\frac{\b^2}4( t-s) -\frac{(u-v)^2}{4(t-s)}\bigr\}}
e^{\frac{\b |v|}{2}}
\b \, \om(s,|v|)\sqrt{2\rho_R(s,|v|)(1-\rho_R(s,|v|))} \bar{W}(dsdv),
\end{align*}
we can easily show the existence and uniqueness in  $C([0,T],C(\R))
\cap C([0, T], \cap_{r>\b/2}L_r^2(\mathbb{R})).$

%

From Proposition \ref{prop:4.4},  we can obtain a result for fluctuation of
$\tilde{\zeta}^N(t,u)$ around its
limit $\om(t, u)$, which is used to show Theorem \ref{thm:2.2} in the last
part of this subsection.
\begin{cor} \label{prop:3.1}
Under Assumption 2,
$\Phi^N(t,u):= \sqrt{N} (\tilde{\zeta}^N(t,u) -
\om(t,u))$ converges weakly to $\Phi(t,u)$ on the space
$D([0,T],C(\R_+))$ as $N\to\infty$.
Moreover the limit $\Phi(t,u)$ is in $C([0,T],C(\R_+))$\ (a.s.) and
characterized as a solution of the  SPDE:
\begin{equation}  \label{eq:3.2}
\left\{
\begin{aligned}
\partial_t\Phi(t,u) & = \Phi''(t,u) + \b \Phi'(t,u)\\
&\quad\quad+ \b\, \om(t,u)\sqrt{2\rho_R(t,u)(1-\rho_R(t,u))}
\dot{W}(t,u),\ u\in \R_+,\\
2\Phi'(t,0) &+ \b \Phi(t,0) = 0, \\
\Phi(0, u)& =\Phi_0(u),
\end{aligned}
\right.
\end{equation}
which has a unique weak solution in  $C([0,T],L_e^2(\R_+))$ for each
$\Phi_0 \in L_e^2(\R_+)$.
\end{cor}
\begin{proof}
Assume Proposition \ref{prop:4.4} is proved. Consider the even functions
$e^{\b |u|/2} \Phi^N(t,|u|)$ on $D([0, T], C(\R))$. We first show that  for
each
$K>0$
\begin{equation}\label{Pro4.3-Proof-2-1}
\lim_{N \to \infty}E\Bigl[\sup_{t\in [0, T]}\sup_{u\in
[-K,K]}|e^{\b |u|/2} \Phi^N(t,|u|)- \bar{\Phi}^N(t,u)|^{2\k }
\Bigr]=0.
\end{equation}
The monotonicity of $\zeta_t^N(x)$ in $x \in \mathbb{N}$ yields
\[
\label{Pro4.3-Proof-3}
\sqrt{N}\bigl| ([Nu] +1 -Nu) \bar{\zeta}_t^N([Nu])+
(Nu-[Nu])\bar{\zeta}_t^N([Nu] +1)) -e^{\b
|u|/2}\tilde{\zeta}^N(t,|u|) \bigr| \leq
CN^{-\frac12}\zeta_t^N(1).
\]
Now Lemma \ref{Lem4.8} below with \eqref{Pro4.3-Proof-2} completes the
proof of \eqref{Pro4.3-Proof-2-1}.
Therefore, as  $N \to  \infty$, $e^{\b |u|/2} \Phi^N(t,|u|)$ converges
weakly on $D([0,T],C(\R))$ to the same limit
${\bar {\Phi}}(t,u)$  as that of  ${\bar \Phi}^N(t,u)$ and it immediately
follows that $\Phi^N(t,u)$ converges weakly on $D([0,T],C(\R_+))$ to
$\Phi(t,u) (=e^{-\b u/2}{\bar
\Phi}(t,u)),\ u \in \mathbb{R}_+$
and the limit $\Phi(t,u)$ is in $C([0,T],C(\R_+))$ (a.\,s.).

To see that $\Phi(t,u)$ is a solution of the SPDE \eqref{eq:3.2}, for a
given $g\in C_0^{1, 2}([0,T]\times \R_+)$ satisfying
$2g'(t,0)-\b\,g(t,0)=0$, set $f(t,u)= e^{-\b |u|/2}g(t,|u|)$. Then $f\in
C_0^{1, 2}([0,T]\times \R)$. Taking such $f$ in \eqref{eq:4.10}, a simple
computation yields
\begin{align*}
\lan\Phi(t),g(t)\ran
=& \lan\Phi_0,g(0)\ran + \int_0^t \lan\Phi(s),g''(s)-\b g'(s)+\partial_s
g(s)\ran ds\\
& + \b \int_0^t \int_{\R_+} g(s,u)
\om(s,u)\sqrt{2\rho_R(s,u)(1-\rho_R(s,u))} W(dsdu)  \quad\text{a.\,s.},
\end{align*}
which completes the proof.
\end{proof}
\begin{rem}
A microscopic interpretation of the mixed boundary condition at
$u=0$ in \eqref{eq:3.2} is found in Lemma 5.8 of \cite{FS}.
\end{rem}
From now on, we formulate the proof of Theorem \ref{thm:2.2} based on
Proposition \ref{prop:4.4}, or more precisely Corollary \ref{prop:3.1},
which is divided into two lemmas.
First note that Assumption 2 can be rewritten into conditions on
$\bar{\Phi}^N_0$. This is mostly used later on but we state it here since
the assertion (3) is needed.
\begin{lem}\label{WASEP-Ini}
Under Assumption 2,  the following holds:
\begin{enumerate}
\item For any  $\k \in \N$, the following estimates hold:
\begin{enumerate}
\item $\sup_NE[\bar{\zeta}_0^N(1)^{2\k }]<\infty$,
\item $E\left[ \left|\bar{\Phi}_0^N(x) \right|^{2\k} \right]\leq  C
e^\frac{\k ' \b  |x|}{N}, x\in \Z$,
\item $E\left[ \left|\bar{\Phi}_0^N(x)-\bar{\Phi}_0^N(y) \right|^{2\k}
\right]\leq
C (e^{\frac{\k' \b |x|}{N}} + e^{\frac{\k' \b |y|}{N}}) \left(\left|
\tfrac{x-y}{N} \right|^{2\k \a } +
\left|\tfrac{x-y}{N}\right|^{2\k}\right)$, for $x, y \in \mathbb{Z}$ and
some $\k' > \k$ and any $\a  \in (0, 1/2)$.
\end{enumerate}
\item $\{\bar{\Phi}_0^N(x)\}_N$ are independent of the noises determining
the process $\{\eta_t^N; t\ge 0\}$,
\item $\bar{\Phi}^N(0,u)$ converges weakly to $\bar{\Phi}_0(u)=
e^{\b |u|/2} \b \om(0,|u|)\Psi_{R,0}(|u|)$ in $C(\R)$ as $N \to \infty$.
In addition, for  all $r>\frac{\b}{2}$,
$E\bigl[|\bar{\Phi}_{0}|_{L_r^2(\R)}^2 \bigr]<\infty$.
\end{enumerate}
\end{lem}
\begin{proof}
Conditions (1)-(3) in Assumption 2 are referred to as (A2-1)-(A2-3),
respectively. (2) is obviously implied by (A2-2).
The next step is the  condition (1). The properties of the transformation
from $q \in \mathcal{Q}$ to
$\eta \in \mathcal{X}_R$, see Lemma 5.1 of \cite{FS}, yield for any
$x\in \N$
\[
\bar{\zeta}_0^N(x)= e^{-(\log \e) |x|/2} \exp \bigl\{-(\log \e) N
\tilde{\psi}_R^N \bigl(0,\left|\tfrac{x}{N}\right|\bigr)\bigr\}.
\]
Therefore, (i) follows directly from (A2-1)(i).

By the definition of $\bar{\Phi}_0^N(x)$ it is enough to prove (ii) for $x
\in \N$. Since 
$\om (0, u) = \exp\{\b \psi_{R,0}(u)\}, u \geq 0$, we deduce that
\begin{equation}\label{WASEP-Ini-0-New}
\bar{\Phi}_0^N(x)=\sqrt{N} e^{-(\log \e) x/2} \Bigl[\exp\bigl\{-(\log
\e)N \tilde{\psi}_R^N
\bigl(0,\tfrac{x}{N}\bigr) \bigr\} - \exp \bigl\{\b
\psi_{R,0}\bigl(\tfrac{x}{N}\bigr) \bigr\}\Bigr].
\end{equation}
The mean value theorem implies the existence of a random variable
$\th ^N(x)$ with values between
$\b \psi_{R,0}\bigl(\tfrac{x}{N}\bigr)$ and $-(\log \e)N
\tilde{\psi}_R^N\bigl(0,\tfrac{x}{N}\bigr)$ such that
\[
\bar{\Phi}_0^N(x)= e^{-(\log \e) x/2 + \th ^N(x)} \bigl[\b\Psi_R^N
\bigl(0,\tfrac{x}{N}\bigr) - \sqrt{N} \bigl(\b +(\log \e)N \bigr)
\tilde{\psi}_R^N\bigl(0,\tfrac{x}{N}\bigr) \bigr].
\]
Note that  $\sqrt{N}(\b+(\log \e) N ) \to  0$ with order
$O\bigl(\tfrac{\log N}{\sqrt{N}}\bigr)$ and combine the monotonicity  of
$\tilde{\psi}_R^N (0, u)$ and $\psi_{R,0}(u)$ with (A2-1)(i) leads to the
estimate
\begin{equation}\label{WASEP-Ini-0}
\sup_{x,N }E\bigl[e^{ 4\k \th ^N(x)}\bigr] \leq C \quad \text{and} \quad
E\Bigl[\bigl|\sqrt{N} \bigl(\b +(\log \e)N \bigr) \tilde{\psi}_R^N(0,
\tfrac{x}{N})
\bigr|^{4 \k} \Bigr]\leq C\Biggl(\frac{\log N}{\sqrt{N}}\Biggr)^{4\k}.
\end{equation}
After applying Schwarz's inequality together with (A2-1)(ii), we arrive at
\[
E \bigl[ \bigl|\bar{\Phi}_0^N(x) \bigr|^{2\k} \bigr] \leq
C e^{-2 \k(\log \e) x} \leq C e^{ \frac{\k' \b x}{N}}.
\]
As explained above, we assume $x, y \in \N$ in the proof of (iii).
It follows  from \eqref{WASEP-Ini-0-New} that
\begin{equation}\label{WASEP-Ini-0-1}
\begin{split}
\bar{\Phi}_0^N(x) -\bar{\Phi}_0^N(y)=& A_1 + \sqrt{N}
e^{-(\log \e) x/2}\left(\exp \{\b \tilde{\psi}_R^N
\bigl(0,\tfrac{x}{N}\bigr)\} -
\exp\{ \b \psi_{R,0}\bigl(\tfrac{x}{N}\bigr)\} \right) \\
& -\sqrt{N} e^{-(\log \e) y/2}\left(\exp \{\b \tilde{\psi}_R^N\bigl(0,
\tfrac{y}{N}\bigr) \} - \exp\{\b \psi_{R,0}\bigl(
\tfrac{y}{N}\bigr) \}\right),
\end{split}
\end{equation}
where the term $A_1:= A_1^N(x,y)$ is estimated as above by
\[
E[|A_1|^{2\k}] \leq C \Bigl(e^{\frac{\k' \b x}{N}} +
e^{\frac{\k' \b y}{N}}\Bigr)N^{-2\k \a}\ \text{ for every } \a <1/2.
\]
Rewriting the two other summands on the right hand side of
\eqref{WASEP-Ini-0-1} yields
\[
\left| \bar{\Phi}_0^N(x) -\bar{\Phi}_0^N(y) \right|^{2\k} \leq C
|A_1|^{2\k} + \Bigl(e^{\frac{\k' \b x}{N}} +
e^{\frac{\k' \b y}{N}}\Bigr) |A_2 + A_3|^{2\k},
\]
with
\begin{align*}
 A_2& = N^{\frac12}\exp\{\b\psi_{R,0}(\tfrac{x}{N}) \}\bigl( \exp \{\b
\bigl(\tilde{\psi}_R^N(0, \tfrac{x}{N})
 -\psi_{R,0}(\tfrac{x}{N}) \bigr) \} - \exp \{\b
\bigl(\tilde{\psi}_R^N(0, \tfrac{y}{N}) -\psi_{R,0}(
\tfrac{y}{N})\bigr)\}\bigr) \\
A_3 &= N^{\frac12}\bigl( \exp\{\b\psi_{R,0}(
\tfrac{x}{N}) \}-\exp\{\b\psi_{R,0}(\tfrac{y}{N})\} \bigr)
\bigl(\exp \{\b \bigl( \tilde{\psi}_R^N(0, \tfrac{y}{N}) -\psi_{R,0}
(\tfrac{y}{N}) \bigr) \} -1\bigr).
\end{align*}
However, with a similar approach to \eqref{WASEP-Ini-0} one can derive
\[
E[|A_2|^{2\k}] \leq C \left|\tfrac{x-y}{N} \right|^{\k} \quad \text{and}
\quad  E[|A_3|^{2\k}] \leq C \left|\tfrac{x-y}{N} \right|^{2\k}.
\]
This concludes the proof of (iii).

The final task is assertion (3). The proof of Corollary \ref{prop:3.1}
suggests that it is enough to prove  (3) for $e^{\b |u|} \Phi^N(0, |u|)$
instead of $\bar{\Phi}^N(t,u)$.
It is easy to see that
\begin{equation}\label{WASEP-Ini-3}
\begin{split}
& \Phi^N(0, u)=\sqrt{N} \left[\exp \{ \b \tilde{\psi}_R^N(0,u)\} -
\exp \{ \b \psi_{R,0}(u) \} \right]   \\
 & + \sqrt{N} \Bigl[\exp\{ -(N\log\e) \tilde{\psi}_R^N \bigl(0,
\tfrac{[Nu]}{N} \bigr) - (\log\e) r^N(0,u)\} -\exp \{\b
\tilde{\psi}_R^N(0,u) \} \Bigr]
\end{split}
\end{equation}
with $r^N(0,u) = 1_{\{u\ge 1/N\}} ([Nu]+1-Nu)\eta_0^N([Nu]) \in [0,1]$.
By Taylor's theorem, the first term on the right hand side of
\eqref{WASEP-Ini-3} is equal to
\begin{equation}\label{WASEP-Ini-4}
\b \exp \{\b \psi_{R,0}(u) \} \Psi_R^N(0, u) +\tfrac12 \b^2 \exp\{ \b
\th ^N(u)\} N^{-1/2} \bigl(\Psi_R^N(0,u) \bigr)^2,
\end{equation}
where $\th ^N(u)$ is a random variable with values between
$\tilde{\psi}_{R}^N(0, u)$ and $\psi_{R, 0}(u)$.
From (A2-3), it  follows that the first part converges weakly to $\b
\om(0,u)\Psi_{R,0}(u)$ in $D(\R_+)$.
In a similar way to the  proof of  (1)(ii), we see that
\[
E\bigl[ \bigl| \exp\{ \b \th ^N(u) \}N^{-1/2} (\Psi_R^N(0, u))^2
 \bigr|^{\k} \bigr] \leq CN^{-\k/2} E\bigl[\bigl(
\Psi_R^N(0, u) \bigr)^{4\k} \bigr]^{1/2},
\]
which, combined with the right continuity of $\Psi_R^N(0, u)$, implies that
the second term of \eqref{WASEP-Ini-4} converges to $0$ in probability in
$D(\R_+)$.
In addition, it is easy to check that the second term in \eqref{WASEP-Ini-3}
also converges to $0$
in probability in $D(\R_+)$. Because $e^{\b |u|/2} \Phi^N(0, |u|)$ is even,
we see  that
$e^{\b |u|/2} \Phi^N(0, |u|)$ converges weakly to
$e^{\b |u|/2} \b \om(0,|u|)\Psi_{R,0}(|u|)$ on $D(\R)$. On the other hand, since the Skorohod topology
relativized to $C(\R_+)$ coincides with its locally uniform topology, the
continuity of $\Phi^N(0, u)$ in $u$ completes the proof.
\end{proof}
\begin{lem}\label{R-rela-lem4.4}
Assume Proposition \ref{prop:4.4} is shown. Then as $N\to\infty$,
$\Psi_R^N(t,u)$ converges weakly on the space $D([0,T], D(\R_+))$ to
\[
\Psi_R(t,u) =\frac{\Phi(t,u)}{\b \om(t,u)}.
\]
Moreover, the limit $\Psi_R(t,u)$ is in
$C([0,T],C(\R_+))$ (a.\,s.) and satisfies the SPDE \eqref{eq:SPDE-2}.
\end{lem}
\begin{proof}
From Lemma \ref{WASEP-Ini}-(3) the relation $\Psi_{R,0}(u)=\Phi_0(u)/(\b
\om(0, u))$ is known. Due to Skorohod's representation theorem we may assume
that $\Phi^N(t,u)$ converges to $\Phi(t,u)$ uniformly on $[0,T]\times
[0,K]$ for every $K>0$ (a.\,s.) on a properly changed probability space. The
definitions of $\tilde{\zeta}^N(t,u)$ and $\tilde{\psi}_R^N(t,u)$
correspond to
\[
\log \tilde{\zeta}^N(t,u) = -(N\log\e) \tilde{\psi}_R^N \bigl(t,
\tfrac{[Nu]}{N} \bigr)
- (\log\e) r^N(t,u),
\]
where $r^N(t,u) = 1_{\{u\ge \frac1{N}\}} ([Nu]+1-Nu)\eta_t^N([Nu]) \in [0,1]$.
Since $\e = 1-\frac{\b}{N} + O(\frac{\log N}{N^2})$, we have that
\[
\tilde{\psi}_R^N \bigl(t,\tfrac{[Nu]}{N} \bigr) = \tfrac{1}{\b} \{1+O
\bigl( \tfrac{\log N}{N} \bigr)\}
\log \tilde{\zeta}^N(t,u) + O\bigl(\tfrac1N \bigr),
\]
with an error $O(\frac1N)$ which is uniform in $(t,u)$. On the other hand,
we know
\[
\psi_R(t,u) = \tfrac{1}{\b} \log \om(t,u)
\]
and $\inf_{t, u}\om(t,u)\ge 1$, see p. 354 in  \cite{FS}. Thus, we estimate the difference between $\Phi^N(t,u)$ and
$\Phi^N(t,[Nu]/N)$ and due to the uniform convergence of $\Phi^N(t, \cdot)$ to $\Phi(t,\cdot)$, arrive at
\[
\Psi_R^N(t,u) = \sqrt{N} \Biggl[ \frac1{\b} \Bigl(1
  +O\bigl(\tfrac{\log N}N\bigr)\Bigr)
\log \Bigl( \om(t,u)+ \tfrac{\Phi^N(t,u)}{\sqrt{N}} \Bigr)
+ o\Bigl(\tfrac1{\sqrt{N}}\Bigr)
- \frac1{\b} \log \om(t,u) \Biggr],
\]
which concludes the proof of the first part.

To complete the proof, it is enough to check the weak form \eqref{eq:4.5}
of the SPDE \eqref{eq:SPDE-2} for $f \in
C^{1,2}_0([0,T]\times\R_+)$ satisfying $f'(t,0)=0$.  For such $f$,
set $g(t,u) =f(t,u)/(\b \om(t,u))$. Then, we easily see that $g$ satisfies
the condition: $2g'(t,0)-\b\,g(t,0)=0$ and, if we consider a weak solution
of \eqref{eq:3.2} for such $g$, a simple computation
leads to \eqref{eq:4.5} for $f$.
\end{proof}

\subsection{Proof of Proposition \ref{prop:4.4}}\label{subsection4.3}
Subsection \ref{subsec-4.3.1} concerns an important uniform estimate on
${\zeta}_t^N$. 
We then formulate some lemmas for the proof of the
tightness of  $\{\bar{\Phi}^N(t,u)\}_N$ in Subsection  \ref{subsec-4.3.2}
and finally give the proof of
Proposition \ref{prop:4.4} in Subsection \ref{subsec-4.3.3}. To show the
tightness of $\{\bar{\Phi}^N(t,u)\}_N$
on the space $D([0,T],C(\R))$, we mainly mimic the approaches used in
\cite{BG} and \cite{DG}.
\subsubsection{A Uniform Estimate}\label{subsec-4.3.1}
The following lemma is crucial and will be frequently used in the sequel.
\begin{lem}\label{Lem4.8}
Let $\k \in \N$ as above. Under  Assumption 2-(1)(i), for any $T>0$, we have
\[
\sup_{N\in \N}E \Bigl[ \sup_{s\in[0,T]} \zeta_s^N(1)^{2\k } \Bigr] <
\infty.
\]
\end{lem}
\begin{proof}
One can modify the proof of  Proposition 5.4 in \cite{FS}.
Let $\fa  \in C_b^2(\R_+)$ such that
$\fa ' \geq 0,\  \fa (u)=0 $ for $u\in (0,  1]$ and
$\fa (u)=1$ for $u\geq 2$, and for each $\k$ set
\[\label{Lem4.8-1}
\Pi_t^N(\fa ):=\exp \Bigl\{-\k  (\log \e ) \sum_{x\in
\mathbb{N}} \eta_t^N(x)\fa  \bigl(\tfrac{x}{N} \bigr) \Bigr\}.
\]
Since
$\sup_{N\in \mathbb{N}, s\in [0, T]}{\zeta_s^N(1)^\k}/{\Pi_s^N(\fa )} < \infty,$
it is enough to show that
\begin{equation}\label{Lem4.8-3}
\sup_{N\in \mathbb{N}}E\Bigl[\sup_{s\in [0, T]}\Pi_s^N(\fa )^2
\Bigr]< \infty.
\end{equation}
Consider the martingale
$M_t^N(\fa )$ given by
\begin{equation}\label{Lem4.8-4}
M_t^N(\fa )=\Pi_t^N(\fa ) -\Pi_0^N(\fa ) -\int_0^t
\bar{L}^N\Pi_s^N(\fa )ds,
\end{equation}
where $\bar{L}^N=N^2\bar{L}_{\e (N), R}$. Here $\bar{L}_{\e , R}$ is the
generator described at p. 353 of \cite{FS}, so that
\begin{align*}\label{Lem4.8-5}
\bar{L}^N\Pi_s^N(\fa ) =& N^2 \Pi_s^N(\fa )\sum_{x\in
\mathbb{N}} \Bigl(\e  c_+(x, \eta_s^N)+
c_-(x,\eta_s^N) \Bigr) \notag\\
\times&\Bigl[\exp \Bigl\{-\k ( \log \e )
\Bigl( \fa  \bigl(\tfrac{x+1}{N} \bigr)-\fa  \bigl(\tfrac{x}{N} \bigr)
\Bigr)(\eta_s^N(x)-\eta_s^N(x+1))
\Bigr\}-1 \Bigr].
\end{align*}
By simple calculations the quadratic
variation  of $M_t^N(\fa )$ is
given by the following relation:
\begin{align*}
\frac{d}{dt}\lan M^N(\fa ) \ran_t &= N^2
\Pi_s^N(\fa )^2\sum_{x\in \mathbb{N}} \Bigl(\e  c_+(x,
\eta_s^N)+
c_-(x, \eta_s^N) \Bigr) \notag\\
&\times\Bigl[\exp \Bigl\{-2\k (\log \e )
\Bigl(\fa  \bigl(\tfrac{x+1}{N} \bigr)-\fa  \bigl(\tfrac{x}{N} \bigr)
\Bigr)\bigl(\eta_s^N(x)-\eta_s^N(x+1)\bigr)
\Bigr\}\Big. \notag \\
\Big. & -2\exp\Bigl\{-\k (\log \e )
\Bigl(\fa  \bigl(\tfrac{x+1}{N} \bigr)-\fa  \bigl(\tfrac{x}{N}
\bigr)\Bigr)
\bigl(\eta_s^N(x)-\eta_s^N(x+1) \bigr) \Bigr\} +1 \Bigr].
\end{align*}
We first claim that  there exists
$C_1=C_1(\k, \|\fa ' \|_{\infty},\| \fa ''
\|_{\infty})>0$ such that
\begin{equation}\label{Lem4.8-7}
\bar{L}^N\Pi_s^N(\fa )\leq C_1\Pi_s^N(\fa ).
\end{equation}
In fact, note that $c_-(x, \eta)-c_+(x, \eta)=\eta(x+1)
-\eta(x) $ and after rearranging the sum, we can rewrite
$\bar{L}^N\Pi_s^N(\fa )$ as follows:
\begin{align*} 
&\bar{L}^N  \Pi_s^N(\fa )= N^2
\Pi_s^N(\fa )   \notag\\
\times& \Bigl[\sum_{x\in \mathbb{N}} \eta_s^N(x)\Bigl(\exp
\Bigl\{-\k (\log
\e ) \bigl(\fa  \bigl(\tfrac{x+1}{N} \bigr)-\fa  \bigl(\tfrac{x}{N}
\bigr) \bigr)\Bigr\}-\exp \Bigl\{-\k (
\log
\e )\bigl(\fa  \bigl(\tfrac{x}{N}\bigr)-\fa  \bigl(\tfrac{x-1}{N}
\bigr)\bigr)\Bigr\}\Bigr)\Bigr.
\notag\\
+& \sum_{x \in \mathbb{N}} c_-(x, \eta_s^N)\Bigl(\exp \Bigl\{-\k ( \log
\e ) \bigl(\fa  \bigl(\tfrac{x+1}{N} \bigr)-\fa  \bigl(\tfrac{x}{N}
\bigr)\bigr)\Bigr\}+\exp \Bigl\{-\k
(\log \e ) \bigl(\fa  \bigl(\tfrac{x+1}{N} \bigr)-\fa  \bigl(
\tfrac{x}{N}\bigr) \bigr) \Bigr\}\Bigr)
\notag\\
-&   \Bigl.2\sum_{x \in \mathbb{N}} c_-(x, \eta_s^N) - \sum_{x\in
\mathbb{N}} (1-\e ) c_+(x,
\eta_s^N)\Bigl(\exp \Bigl\{-\k (\log
\e ) \bigl(\fa  \bigl(\tfrac{x+1}{N} \bigr)-\fa  \bigl(\tfrac{x}{N}
\bigr)\bigr) \Bigr\} -1 \Bigr)
\Bigr].
\end{align*}
Thus, by Taylor's formula and Lemma 3.2 of \cite{FS}, we can
show \eqref{Lem4.8-7}. As a consequence of  \eqref{Lem4.8-4},
\eqref{Lem4.8-7} and Gronwall's inequality, we have that
\[
\Pi_t^N(\fa ) \leq e^{C_1t}\Bigl(\Pi_0^N(\fa )+\sup_{s\in [0,
t]}|M_s^N(\fa )|\Bigr), \ t\leq T,
\]
 which implies
\begin{equation}\label{Lem4.8-10}
E\Bigl[\sup_{s\in [0, t]}\Pi_s^N(\fa )^2 \Bigr] \leq 2e^{2C_1
t}\Bigl(E\Bigl[\Pi_0^N(\fa )^2 \Bigr] + E\Bigl[\sup_{s\in [0,
t]} M_s^N(\fa )^2 \Bigr] \Bigr), \ t\leq T.
\end{equation}
A similar approach to  \eqref{Lem4.8-7} yields that  there exists  
$C_2=C_2(\k, \|\fa ' \|_{\infty})>0$ such that
\[
\lan M^N(\fa ) \ran_t \leq \frac{C_2}{N}\int_0^t
\Pi_s^N(\fa )^2ds 
\]
and therefore Doob's inequality implies
\[
E\Bigl[\sup_{s\in [0, t]}M_s^N(\fa )^2\Bigr] \leq
\frac{4C_2}{N}E \Bigl[ \int_0^t \Pi_s^N(\fa )^2ds
\Bigr], \ t\leq T.
\]
In the end,  Gronwall's inequality, \eqref{Lem4.8-10} and  Lemma
\ref{WASEP-Ini} conclude the proof of \eqref{Lem4.8-3}.
\end{proof}
\subsubsection{Tightness of $\bar{\Phi}^N(t,u)$}\label{subsec-4.3.2}
A criterion for the tightness of on the space $D([0,T],C(\R))$ is given by
the theorem due to Aldous and Kurtz (see \cite{EK}, \cite[Theorem 2.7]{K}
or \cite[Proposition 4.9]{BG}). It states that it is sufficient to show the
following estimates:
\begin{enumerate}
\item For every $t, K>0$, there exist $ \k \ge 1$, $C$ and
$\a>0$ such that
\begin{align*}
& \sup_N E \bigl[|\bar{\Phi}^N(t,0)|^\k \bigr] < \infty, \\
& \sup_N E \bigl[|\bar{\Phi}^N(t,u_1) - \bar{\Phi}^N(t,u_2)|^\k \bigr]
\leq C|u_1-u_2|^{1+\a}, \quad |u_1|, |u_2| \leq K.
\end{align*}
\item There exists a process $A^N(\de), \de >0$ such that
\begin{align*}
&
E \bigl[d(\bar{\Phi}^N(t+\de,\cdot),\bar{\Phi}^N(t,\cdot)) \big|
\mathcal{F}_t \bigr] \leq E \bigl[A^N(\de) \big|\mathcal{F}_t\bigr],\
t\in [0, T], \\
& \lim_{\de\downarrow 0}\limsup_{N\to\infty} E \bigl[A^N(\de)\bigr]=0.
\end{align*}
\end{enumerate}
Here $\mathcal{F}_t= \si\{\bar{\Phi}^N(s, \cdot); 0\leq s\leq
t\}$ and  $d(\cdot, \cdot)$ denotes a metric on $C(\R)$ which determines
the topology of the uniform convergence on each compact subset of
$\R$.

Before we go to our main topic of this subsection, let us state Burkholder's
inequality according to Theorem 7.11 of \cite{WAL}.
\begin{lem}\label{BurkIneq}
For any $L^{2\k }$-integrable  real valued martingale $M_t$ and
fixed $t>0$, there exists a constant $C=C(\k, t)>0$ such that
\[
E\Bigl[\sup_{s\in [0, t]}|M_s|^{2\k } \Bigr]\leq CE \bigl[\lan M
\ran_t^\k \bigr] +CE \Bigl[\sup_{s\in [0, t]}|M_s-M_{s-} |^{2\k }
\Bigr],
\]
where $\lan M \ran_t$ denotes the  quadratic
variational process of $M_t$.
\end{lem}

As further preparations, we formulate some estimates for
$(\bar{\Phi}_t^N(x))_{x\in \Z}$ defined by \eqref{eq:4.19}.
Let us first summarize some properties of $p(t,x)$ given by
\eqref{heatker}, see \cite{BG,G} for their proofs.
\begin{lem}\label{Ker-est}
There exists a constant $C>0$ such that the following
holds:
\\
$({\rm i})$ For any $t>0$,  $\sup_{x\in \Z}p(t,x) \leq C t^{-\frac12},$ \\
$({\rm ii})$ For any $\alpha \leq 1/2, \quad |p(t,x)-p(t,y)| \leq C
|x-y|^{2\alpha}t^{-\frac12-\alpha},\ x, y \in \Z, $
$$ |p(t+h,x)-p(t,x)| \leq  C h^{2\alpha}t^{-\frac12-2\alpha},\ x \in \Z,\ h>0.$$
\end{lem}
\begin{lem}\label{Est-mon}
Under Assumption 2,  there exist  $C$
and $\tilde {\k}>0$ such that
\[
E \bigl[|\bar{\Phi}_t^N(x)|^{2\k } \bigr]\leq C e^{\frac{\tilde{\k}\b
|x|}{N}} \Bigl(1+ t^{\frac{\k}{2}}+ \bigl(\sqrt{N}(\e^{-1}
-1)\bigr)^{2\k } \Bigr),\ t\in [0, T].
\]
\end{lem}
\begin{proof}
We denote the first and second terms in the right hand side of
\eqref{barphi-2} by $I^{(N,1)}(t,x)$ and $I^{(N,2)}(t,x)$,
respectively. Then, Lemma \ref{WASEP-Ini} (1)(ii),
H\"{o}lder's inequality and the property $\sum_{y \in
\Z}p^N(t,x,y)=1$ result in
\[\label{mom-1}
E\bigl[|I^{(N,1)}(t,x)|^{2\k }\bigr] \leq  \sum_{y \in
\mathbb{Z}} p^N(t,x,y)e^{2\k  c_Nt}
E\bigl[|\bar{\Phi}_0^N(y)|^{2\k }\bigr]
\leq \sum_{y \in \mathbb{Z}} p^N(t,x,y)e^{2\k
c_Nt}Ce^{\frac{\k' \b  |y|}{N}}.
\]
On the other hand, since for each $a \in \R$, the
 function $e^{ax}$ is the eigenfunction of the
operator $ \De$ corresponding to  the eigenvalue $e^{a} +e^{-a}
-2$
\begin{equation}\label{mom-5}
\sum_{y \in \mathbb{Z}} p^N(t,x,y)e^{ay} =e^{ax}\exp\Bigl\{
\bigl(e^{a} +e^{-a} -2 \bigr)N^2\e^{1/2} t \Bigr\}
\end{equation}
holds. Note that $e^{\frac{\k' \b  |y|}{ N}} \leq e^{\frac{\k'
\b  y}{ N}} +e^{-\frac{\k' \b  y}{ N}}$  
and it follows that
\begin{equation}\label{mom-1-1}
E\bigl[|I^{(N,1)}(t,x)|^{2\k }\bigr] \leq Ce^{\frac{\k' \b
|x|}{N}}, \ t\in [0,T],
\end{equation}
by  the behavior of $c_N$ and the convergence of
$\bigl(e^{\frac{\k' \b }{ N}} +e^{-\frac{\k' \b }{ N}}
-2\bigr)N^2$ as $N \to \infty$.

In the following, we are going to  estimate
$E\bigl[|I^{(N,2)}(t,x)|^{2\k }\bigr]$ by using  Burkholder's
inequality stated in Lemma \ref{BurkIneq}. However, as a
stochastic process, it is well-known that $I^{(N,2)}(t,x)$
is not a martingale. Instead of the direct disposal of
$I^{(N,2)}(t,x)$, we will fix $t>0$ and consider the
process
\[
I_{t}^{(N,2)}(r,x)=\int_0^r\sum_{y \in \mathbb{Z}}
p^N(t-s,x,y)e^{c_N(t-s)}\sqrt{N}d\bar{M}_s^N(y), \ r<t,
\]
which is a real valued  martingale in $r$ for each $x \in
\mathbb{Z}$ with quadratic variation
\[
\lan I_{t}^{(N,2)}(\cdot,x) \ran_r = \int_0^r\sum_{y
\in \mathbb{Z}} \bigl(p^N(t-s,x,y)e^{c_N(t-s)}\bigr)^2
Nd\lan \bar{M}^N(y) \ran_s, \ r<t.
\]
Since $\zeta_t^N(x) \le \zeta_t^N(1)$ for
any $x\geq 2$,  $a_N$ and $b_N$ are both
bounded in $N$, \eqref{eq:4.SDE-1-1} yields  that
\begin{equation}\label{mom-4-1}
 d \lan \bar{M}^N(y)\ran_s
\le C e^{-(\log \e)|y|} \zeta_s^N(1)^2 ds,
\end{equation}
 which implies
\[
\big\lan I_{t}^{(N,2)}(\cdot,x) \big\ran_r\leq C
\int_0^r\sum_{y \in \mathbb{Z}}
\bigl(p^N(t-s,x,y)e^{c_N(t-s)}\bigr)^2 Ne^{-(\log \e) |y|}
\zeta_s^N(1)^2ds.
\]
Thus, by the above estimate and  Lemma \ref{Lem4.8}, we have
\begin{equation}\label{mom-8}
 E\bigl[ \lan I_{t}^{(N,2)}(\cdot,x) \ran_r^\k
\bigr]  \leq  C\Bigl(\int_0^r \sup_{y\in\mathbb{Z}}|p^N(t-s,x,y)
N| e^{-(\log \e)|x|}  ds \Bigr)^\k
 \leq C e^{-\k (\log \e)|x|}t^{\k/2},
\end{equation}
where $Np^N(t,x,y) \leq C t^{-\frac{1}{2}}$ has been used,  see Lemma
\ref{Ker-est} (i).

Finally, let us consider the jump size of $I_t^{(N,2)}(r,x)$. By the
definition, the jump size is determined by $M_s^N(y)$, which
inherits from $\zeta_s^N(x)$.  More precisely, we have that
\begin{align}\label{mom-9}
 \qquad \sup_{r\in [0,t ]} \big|I_t^{(N,2)}(r,x)-I_t^{(N,2)}(r-,x) \big|
\leq& C\sup_{r\in [0,t ]} \sum_{y\in \mathbb{Z}}
p^N(t-r,x,y) \sqrt{N}|\bar{M}_r^N(y)
-\bar{M}_{r-}^N(y)| \notag \\
=& C\sup_{r\in [0,t ]}  \sum_{y\in \mathbb{Z}}
p^N(t-r,x,y) \sqrt{N}|\bar{\zeta}_r^N(y)
-\bar{\zeta}_{r-}^N(y)|.
\end{align}
Since $\eta_t^N(y)$ does not jump  at same time for different
$y$, we see that
\begin{equation}\label{mom-10}
|\bar{\zeta}_r^N(y) -\bar{\zeta}_{r-}^N(y)| \leq Ce^{-(\log
\e)|y|/2} |\zeta_{r}^N(1)- \zeta_{r-}^N(1)| \leq  Ce^{-(\log
\e)|y|/2}(\e^{-1} -1)\zeta_{r}^N(1).
\end{equation}
Consequently, Lemma
\ref{Lem4.8} and \eqref{mom-5} imply again that
\[
E\Bigl[\sup_{r\in [0,t]}|I_t^{(N,2)}(r,x)-I_t^{(N,2)}(r-,x)|^{2\k } \Bigr]
\leq Ce^{-\k(\log \e)|x|}\bigl(\sqrt{N}(\e^{-1} -1)\bigr)^{2\k }.
\]
Note that $I_{t}^{(N,2)}(r,x)$ converges to $I^{(N,2)}(t,x)$ in
$L^2(\Omega)$ as $r \uparrow t$ and we can conclude  the proof
by Lemma \ref{BurkIneq}, \eqref{mom-1-1} and \eqref{mom-8}.
\end{proof}
\begin{lem}\label{Lemma-barphi-1}
Under Assumption 2, the following estimates hold:
\begin{enumerate}
\item For each $\a <1/2$,  there exist   $C$ and $\tilde{\k}>0$
such that for any $t\leq T$ and $x, y \in \mathbb{Z}$,
\begin{equation}\label{barphi-3}
\begin{split}
E \bigl[|\bar{\Phi}_t^N(x) -\bar{\Phi}_t^N(y)|^{2\k} \bigr]\leq& C
\Bigl(e^{\frac{\tilde{\k} \b  |x|}{N}} +e^{\frac{\tilde{\k}
\b    |y|}{N}} \Bigr)  \\
\times & \Bigl(\frac{|x-y|^{2\k \a }}{N^{2\k\a }}
+\frac{|x-y|^{2\k }}{N^{2\k}} +\bigl(\sqrt{N}(\e^{-1} -1)
\bigr)^{2\k } \Bigr).
\end{split}
\end{equation}
\item For each $\a <1/4$, there exist $C$ and
$\tilde{\k}>0$ such that for any $t_1, t_2 \leq T$ and $x\in
\mathbb{Z}$,
\begin{equation}\label{barphi-4}
E\bigl[|\bar{\Phi}_{t_1}^N(x) -\bar{\Phi}_{t_2}^N(x)|^{2\k }\bigr] \leq
Ce^{\frac{\tilde{\k} \b  |x|}{N}}\Bigl(|t_1-t_2|^{2\k\a } +
\bigl(\sqrt{N}(\e^{-1} -1) \bigr)^{2\k }\Bigr).
\end{equation}
\end{enumerate}
\end{lem}
\begin{proof}
The main idea to prove this lemma is similar to that for
Lemma \ref{Est-mon}. We will only give some necessary
explanations by using same notations. We begin with the proof of \eqref{barphi-3}.  The representation  of $I^{(N, 1)}(t,x)$,
 change of variables and  Lemma \ref{WASEP-Ini} yield that
\begin{equation}\label{mom-differ-1-0}
 E \bigl[(I^{(N, 1)}(t,x)-I^{(N,
1)}(t,y))^{2\k } \bigr]\leq C \Bigl(e^{\frac{\k' \b  |x|}{N}}
+e^{\frac{\k' \b
|y|}{N}} \Bigr)  
\times \Bigl(\frac{|x-y|^{2\k \a }}{N^{2\k\a }}
+\frac{|x-y|^{2\k }}{N^{2\k}} \Bigr).
\end{equation}
For $I_t^{(N,2)}$, owing to Lemma
\ref{BurkIneq}, it is sufficient to deal with $\lan
I_t^{(N,2)}(\cdot,x) - I_t^{(N,2)}(\cdot,y) \ran_r $ and
the jump of $ I_t^{(N,2)}(r,x) -I_t^{(N,2)}(r,y)$
respectively.\\
By  Lemma \ref{Ker-est} (ii),
$
N|p^N(s,x,z) -p^N(s,y,z)| \leq CN^{-2\a }s^{-1/2 -\a }|x-y|^{2\a },\ \a  <1/2.
$\\
Now let us take $r=t$ and we deduce
that
\begin{align}\label{mom-differ-1}
 & E\bigl[ \big\lan I_t^{(N,2)}(\cdot,x) - I_t^{(N,2)}(\cdot,y)
 \big\ran_t^{\k} \bigr] \\
 &\leq CE\Bigl[\sup_{s\in [0, t]}|\zeta_s^N(1)|^{2\k } \Bigr]
\Bigl(\int_0^t\sum_{z \in \mathbb{Z}} \bigl(p^N(s,x,z)
-p^N(s,y,z)\bigr)^2e^{2c_Ns} Ne^{-(\log \e) |z|} ds
\Bigr)^\k  \notag \\
&\leq C \Bigl(\frac{|x-y|^{2\a }}{N^{2\a }} \int_0^t s^{-1/2
-\a }e^{2c_Ns}\sum_{z\in \mathbb{Z}}(p^N(s,x,z)+
p^N(s,y,z))e^{-(\log \e) |z|}ds \Bigr)^\k \notag \\
&\leq C\Bigl(e^{-\k(\log \e) |x|}+e^{-\k(\log \e) |y|}
\Bigr)N^{-2\k\a }|x-y|^{2\k \a }t^{ \k(1/2
-\a )}, \notag
\end{align}
where Lemma \ref{Lem4.8} has been applied for the second inequality, and $\a
<1/2$ as well as \eqref{mom-5} have been used for the last inequality. Using
a similar approach to \eqref{mom-9} and recalling
\eqref{mom-5} and \eqref{mom-10}, we have
\begin{align*} 
& E\Bigl[\sup_{r\in [0,t ]} \Bigl|\bigl(I_t^{(N,2)}(r,x) -
I_t^{(N,2)}(r,y) \bigr)
-\bigl(I_t^{(N,2)}(r-,x)-I_t^{(N,2)}(r-,y) \bigr) \Bigr|^{2\k } \Bigr] \\
&\leq  C\Bigl( \sup_{r\in [0,t ]} \sqrt{N}\sum_{z \in \mathbb{Z}}
|p^N(t-r,x,z)- p^N(t-r,y,z)| e^{c_N(t-r)}e^{-(\log \e)|z|/2}
(\e^{-1} -1) \Bigr)^{2\k } \notag\\
&\leq C \bigl(\sqrt{N}(\e^{-1} -1) \bigr)^{2\k }
\bigl(e^{-\k(\log \e) |x|}+e^{-\k(\log \e) |y|}
\bigr),\notag
\end{align*}
which yields \eqref{barphi-3} together with \eqref{mom-differ-1-0} and
\eqref{mom-differ-1}.

Now we show the second estimate \eqref{barphi-4} but only for $I^{(N,2)}
(t_2, x) -I^{(N,2)}(t_1,x)$,
since that for $I^{(N,1)}(t_2, x) -I^{(N,1)}(t_1,x)$ is easier.  By
a similar approach as above, for any $0\leq t_1 < t_2 \leq T$, we
can easily obtain
\begin{align*}
& E\Bigl[ \Bigl|\int_{t_1}^{t_2}\sum_{y \in \mathbb{Z}}
p^N(t_2-s,x,y)e^{c_N(t_2-s)}\sqrt{N}d\bar{M}_s^N(y) \Bigr|^{2\k }
\Bigr]\\
&\leq Ce^{-\k(\log \e) |x|}
\Bigl(\int_{t_1}^{t_2}\sup_{y} p^N(t_2-s,x,y)N
e^{2c_N(t_2-s)}e^{(\e +\e^{-1}
-2)N^2\e^{1/2} (t_2-s)}ds\Bigr)^{\k} \notag \\
 &\quad+ CE\Bigl[\Bigl(\sup_{r\in [t_1, t_2]} \Bigl|\int_{r-}^r
\sum_{y\in \mathbb{Z}} p^N(t_2 -r,
x,y)e^{c_N(t_2-s)}\sqrt{N}d\bar{M}_s^N(y) \Bigr|
 \Bigr)^{2\k } \Bigr] \\
&\leq  C e^{-\k(\log \e) |x|}|t_1 -t_2|^{\frac{\k}{2}} +
Ce^{-\k(\log \e) |x|} \bigl(\sqrt{N}(\e^{-1} -1)
\bigr)^{2\k }. \notag
\end{align*}
Hereafter, for simplicity, to deal with the jump part, we write it
in its integral form and consider the  process directly. In fact,
the following calculations are just formal, see Lemma \ref{Est-mon}
for  concrete explanations.

By  Lemma \ref{BurkIneq}  and \eqref{mom-4-1} and imitating
the procedure used in the proof of \eqref{barphi-3}, we can
deduce that
\begin{equation*}\label{mom-differ-2-1}
\begin{split}
&E\Bigl[ \Bigl|\int_0^{t_1} \sum_{y \in \mathbb{Z}} (
p^N(t_1-s, x, y)e^{c_N(t_1-s)} -p^N(t_2-s, x,
y)e^{c_N(t_2-s)}
)\sqrt{N}\bar{M}_s^N(y) \Bigr|^{2\k } \Bigr] \\
 &\leq   C
\Bigl(\int_0^{t_1} \sum_{y \in \mathbb{Z}} \bigl(
p^N(t_1-s, x, y)e^{c_N(t_1-s)} -p^N(t_2-s, x,
y)e^{c_N(t_2-s)} \bigr)^2Ne^{-(\log
\e)|y|}ds \Bigr)^\k\\
  &+ C E\Bigl[\sup_{r\in [0, t_1]}\Bigl|\int_{r-}^{r} \sum_{y \in
\mathbb{Z}} \bigl( p^N(t_1-s, x, y)e^{c_N(t_1-s)} -p^N(t_2-s, x,
y)e^{c_N(t_2-s)} \bigr)\sqrt{N}\bar{M}_s^N(y) \Bigr|^{2\k }
\Bigr]
\end{split}
\end{equation*}
holds, where the second term on the right hand side denotes the
corresponding jump part.
 With the behavior of $c_N$, the property
\[
N| p^N(t_1-s, x, y) -p^N(t_2-s, x, y) | \leq C (t_2-t_1)^{2\a }
(t_1-s)^{-1/2 -2\a }, 
\]
and \eqref{mom-5}, we obtain that the first term on the right
hand side is bounded from above by
\[
 Ce^{-\k(\log \e) |x|}(t_2 -t_1)^{2\k \a },
\]
where the restriction of $\a  <1/4$ has been used.

Relation \eqref{mom-10} is used to bound the second term from above by
\[
2e^{-\k(\log \e)|x|}\bigl(\sqrt{N}(\e^{-1} -1)
\bigr)^{2\k }\sup_{s\in [0, t_1]}\Bigl(e^{(\e
+\e^{-1} -2)N^2\e^{1/2} (t_1-s)} +e^{(\e
+\e^{-1} -2)N^2\e^{1/2} (t_2-s)}\Bigr)^\k.
\]
All the above estimates applied yield the upper bound
\begin{align*}
&E\Bigl[ \Bigl|\int_0^{t_1} \sum_{y \in \mathbb{Z}} \bigl( p^N(t_1-s, x,
y)e^{c_N(t_1-s)} -p^N(t_2-s, x, y)e^{c_N(t_2-s)} \bigr)
\sqrt{N}\bar{M}_s^N(y) \Bigr|^{2\k } \Bigr] \\
&\leq  Ce^{-\k(\log \e)|x|} |t_2-t_1|^{2\k\a } +
Ce^{-\k(\log \e) |x|} \bigl(\sqrt{N}(\e^{-1} -1)
\bigr)^{2\k }.
\end{align*}
Now we can easily conclude the proof of \eqref{barphi-4} by
choosing a proper $\tilde{\k}$.
\end{proof}
To show the tightness of  $\bar{\Phi}^N(t, u)$ and  that its limit is in
$C([0, T], C(\R))$, we prepare two lemmas.
We first establish the local H\"{o}lder estimates in the space variable.
\begin{lem}\label{Kol-Lem-1}
For any  $T>0$ and each $\a  < 1/2$, there exists a
constant $C>0$ such that for any $u_1, u_2 \in [-K, K]$ and
$t\leq T$
\begin{equation}\label{Kol-Lem-1-1}
E\bigl[|\bar{\Phi}^N(t, u_1) -\bar{\Phi}^N(t, u_2)|^{2\k }\bigr] \leq C
|u_1 -u_2|^{2\k\a },
\end{equation}
and moreover for   $\a ' \in [0, \frac{2\k\a  -1}{2\k })$
\begin{equation}\label{Kol-Lem-1-2}
\sup_{t\in [0, T]}E\Big[\sup_{u_1 \neq
u_2}\Bigl(\frac{\bar{\Phi}^N(t, u_1) -\bar{\Phi}^N(t, u_2)|}{|u_1
-u_2|^{\a '}} \Bigr)^{2\k } \Big] <\infty.
\end{equation}
\end{lem}
\begin{proof}
Let us first assume that $[Nu_1] =[Nu_2]$. Then by
\eqref{barphi-3} 
we deduce  that
\begin{align*}
E \bigl[|\bar{\Phi}^N(t, u_1) -\bar{\Phi}^N(t, u_2)|^{2\k }\bigr]
\leq& CN^{2\k }|u_1-u_2|^{2\k }
E \bigl[|\bar{\Phi}^N_t([Nu_1]+1)-\bar{\Phi}^N_t([Nu_1])|^{2\k } \bigr]
\notag\\
\leq & C'|u_1-u_2|^{2\k\a }.
\end{align*}
Next we assume, without loss of generality, that $Nu_1 < [Nu_1]+1 \leq [Nu_2] \leq
Nu_2$. From the above estimate and again \eqref{barphi-3}, we obtain
\begin{align*}
E \bigl[|\bar{\Phi}^N(t, u_1) -\bar{\Phi}^N(t, u_2)|^{2\k } \bigr]
\leq & C\Bigl(\bigl(\tfrac{[Nu_1]+1}{N}- u_1\bigr)^{2\k\a  } +
\bigl(u_2 -\tfrac{[Nu_2]}{N}\bigr)^{2\k\a  }\Bigr)\\
  + C E \bigl[ \bigl|\bar{\Phi}^N & \bigl(t,\tfrac{[Nu_1]+1}{N}\bigr)
-\bar{\Phi}^N \bigl(t, \tfrac{[Nu_2]}{N} \bigr)\bigr|^{2\k } \bigr]
\leq C'|u_1-u_2|^{2\k\a },
\end{align*}
which implies \eqref{Kol-Lem-1-1}. On the other hand, the
second assertion \eqref{Kol-Lem-1-2} is a direct
consequence of \eqref{Kol-Lem-1-1} and Kolmogorov's
theorem, for example, see Proposition 4.4 of \cite{BG}.
\end{proof}
In fact, it is clear that $\bar{\Phi}^N(t, u)$ is not continuous in
$t$, so Kolmogorov's continuity theorem cannot be applied directly.
To overcome this difficulty, we introduce the process
$\bar{\mathbf{\Phi}}^N(t,u)$, namely consider the linear interpolation in
time $t$ defined as
\[
\bar{\mathbf{\Phi}}^N(t, u) :=([N^2t] +1 -N^2t)\bar{\Phi}^N \bigl(
\tfrac{[N^2 t]}{N^2}, u \bigr) - (N^2t- [N^2t])\bar{\Phi}^N \bigl(
\tfrac{[N^2t] +1}{N^2}, u \bigr).
\]
\begin{lem}\label{Kol-Lem-2}
Let $\a  < 1/4$. For any $t_1, t_2 \leq T$ and $u_1, u_2 \in [-K,
K]$, there exists a constant $C>0$ such that
\begin{equation}\label{Kol-Lem-2-1}
 E \bigl[|\bar{\mathbf{\Phi}}^N(t_1, u_1) -\bar{\mathbf{\Phi}}^N
(t_2, u_2)
|^{2\k } \bigr]   \leq  C \Bigl( \bigl(|t_2-t_1|^{2\k\a } +|u_2
-u_1|^{2\k\a } \bigr) + \bigl(\sqrt{N}(\e^{-1} -1)
\bigr)^{2\k } \Bigr)
\end{equation}
and moreover for  $\a ' \in [0, \frac{2\k\a -1}{2\k })$
\[
E\Big[\sup_{t_1\neq t_2}\sup_{u_1 \neq u_2 }
\Bigl(\frac{\bar{\mathbf{\Phi}}^N(t_1, u_1)
-\bar{\mathbf{\Phi}}^N(t_2, u_2)}{(|t_2-t_1|+ |u_2
-u_1|)^{\a '}}\Bigr)^{2\k } \Big] < \infty.
\]
\end{lem}
\begin{proof}
Note that
\begin{align*}
 & E\bigl[|\bar{\mathbf{\Phi}}^N(t_1, u_1) -\bar{\mathbf{\Phi}}^N(t_1,
u_2) |^{2\k } \bigr]\\
& \leq C \bigl([N^2t_1]+1 -N^2t_1
\bigr)^{2\k } E\Bigl[\bigl|\bar{\Phi}^N \bigl(\tfrac{[N^2t_1]}{N^2}, u_1
\bigr)
-\bar{\Phi}^N \bigl(\tfrac{[N^2t_1]}{N^2}, u_2 \bigr)\bigr|^{2\k }
\Bigr] \notag\\
 &+C\bigl(N^2t_1
-[N^2t_2] \bigr)^{2\k }E\Bigl[\bigl|\bar{\Phi}^N \bigl(
\tfrac{[N^2t_1]+1}{N^2},u_1 \bigr) -\bar{\Phi}^N \bigl(
\tfrac{[N^2t_1] +1}{N^2}, u_2 \bigr) \bigr|^{2\k } \Bigr]
\end{align*}
and use Lemma \ref{Kol-Lem-1} to obtain
\begin{equation}\label{Kol-Lem-2-4}
E \bigl[|\bar{\mathbf{\Phi}}^N(t_1, u_1) -\bar{\mathbf{\Phi}}^N
(t_1, u_2)|^{2\k } \bigr]
\leq C|u_1-u_2|^{2\k\a }.
\end{equation}
Let us now deal with the term
$E \bigl[|\bar{\mathbf{\Phi}}^N(t_1, u_2) -\bar{\mathbf{\Phi}}^N
(t_2, u_2)|^{2\k } \bigr].$
We mainly use a similar method to the proof of Lemma \ref{Kol-Lem-1} and
first assume that $[N^2t_1]=[N^2t_2]$. Then  by \eqref{barphi-4} and
$\a  <1/4$, we have that
\begin{equation}\label{Kol-Lem-2-6}
\begin{split}
& E \bigl[|\bar{\mathbf{\Phi}}^N(t_1, u_2) -\bar{\mathbf{\Phi}}^N(t_2,
u_2)|^{2\k } \bigr] \\
&\leq
C\bigl(N^2(t_2-t_1)\bigr)^{2\k }E\Bigl[\bigl|\bar{\Phi}^N \bigl(
\tfrac{[N^2t_1]}{N^2},u_2 \bigr)
-\bar{\Phi}^N \bigl(\tfrac{[N^2t_1]+1}{N^2}, u_2 \bigr)\bigr|^{2\k }
\Bigr] \\
  &\leq C'\bigl(N^2(t_2-t_1)\bigr)^{2\k }\Bigl[N^{-4\k\a } +
\bigl(\sqrt{N}(\e^{-1} -1)\bigr)^{2\k } \Bigr] \leq C''
|t_2 -t_1|^{2\k\a }.
\end{split}
\end{equation}
For general $t_1$ and $t_2$, without loss of generality,  we
may assume that $N^2t_1 <[N^2t_1]+1 \leq [N^2]t_2 \leq
N^2t_2.$ Use \eqref{barphi-4} and
\eqref{Kol-Lem-2-6}, to derive the estimate
\begin{align*}
& E\bigl[|\bar{\mathbf{\Phi}}^N(t_1, u_2) -\bar{\mathbf{\Phi}}^N(t_2,
u_2) |^{2\k } \bigr]\\
&\leq C E \Bigl[\bigl|\bar{\mathbf{\Phi}}^N(t_1, u_2) -
\bar{\mathbf{\Phi}}^N \bigl(\tfrac{[N^2t_1]+1}{N^2}, u_2 \bigr)
\bigr|^{2\k } \Bigr] +
C E\Bigl[\bigl|\bar{\mathbf{\Phi}}^N \bigl(\tfrac{[N^2t_2]}{N^2}, u_2
\bigr)
-\bar{\mathbf{\Phi}}^N(t_2, u_2) \bigr|^{2\k } \Bigr] \notag \\
 & + C E\Bigl[
\bigl|\bar{\mathbf{\Phi}}^N \bigl(\tfrac{[N^2t_1]+1}{N^2}, u_2 \bigr)
-\bar{\mathbf{\Phi}}^N \bigl(\tfrac{[N^2t_2]}{N^2},
u_2 \bigr) \bigr|^{2\k } \Bigr] \leq C'|t_1- t_2|^{2\k\a }. \notag
\end{align*}
Now we obtain \eqref{Kol-Lem-2-1} by \eqref{Kol-Lem-2-4}
and the above estimate. The last part of this lemma is
trivial by Kolmogorov's theorem.
\end{proof}
\begin{prop}\label{tightPhi-1}
The process ${\bar{\Phi}}^N(t, u)$ in Proposition \ref{prop:4.4} is
tight in $D([0, T], C(\mathbb{R}))$.
\end{prop}
\begin{proof}
Recall that $\bar{\Phi}^N(t, 0)= \bar{\Phi}^N_t(0)$ and the Lemmas
\ref{Est-mon} and \ref{Kol-Lem-1} let us easily observe that for each
$t\leq T$, $\bar{\Phi}^N(t, \cdot)$ satisfies the estimates in $(1)$ of
Aldous-Kurtz's conditions. 

In the second step we have to show that condition (2) is also satisfied by
$\bar{\Phi}^N(t, u)$.
To formulate our proof, we will consider the  following  metric on $C(\R)$:
\begin{eqnarray*}
d(w_1, w_2):= \sum_{n\in \mathbb{N}} 2^{-n} \Big(1\wedge \sup_{u\in
[-n, n]} |w_1(u) -w_2(u)| \Big), \ w_1, \ w_2 \in C(\R).
\end{eqnarray*}
It is clear that $C(\R)$ equipped with the metric $d(\cdot,
\cdot)$ is complete and separable.
For each $\de >0$ define
\[
A^N(\de )  :=\sup_{t\in [0, T]}d \bigl(\bar{\Phi}^N(t+\de , \cdot),
\bar{\Phi}^N(t, \cdot) \bigr).
\]
It is clear that
$E \bigl[d(\bar{\Phi}^N(t+\de,\cdot),\bar{\Phi}^N(t,\cdot)) \big|
\mathcal{F}_t \bigr]
\le E \bigl[A^N(\de) \big|\mathcal{F}_t \bigr].$
Thus it is enough to show
\begin{equation}\label{tightPhi-1-3}
 \lim_{\de\downarrow 0}\limsup_{N\to\infty} E \bigl[A^N(\de) \bigr]=0
\end{equation}
in order to complete the proof. For any $\de '>0$, we have that
\begin{align*}
& E\Bigl[\sup_{t\in [0, T]} \bigl\{1\wedge \sup_{u\in [-K,K]}
|\bar{\Phi}^N(t+\de , u) - \bar{\Phi}^N(t, u)| \bigr\} \Bigr] \\
&\leq P(B_K^N(\de ')) + E\Bigl[\sup_{t\in [0, T]} \sup_{u\in
[-K,K]} |\bar{\Phi}^N(t+\de , u) -
\bar{\mathbf{\Phi}}^N(t+\de , u)|, B_K^N(\de ')^c \Bigr] \notag \\
 &\quad+ E\Bigl[\sup_{t\in [0, T]} \sup_{u\in [-K,K]}
|\bar{\Phi}^N(t, u) - \bar{\mathbf{\Phi}}^N(t, u)|, B_K^N(\de ')^c
\Bigr] \notag \\
 &\quad+ E\Bigl[\sup_{t\in [0, T]} \sup_{u\in [-K,K]}
|\bar{\mathbf{\Phi}}^N(t+\de , u) - \bar{\mathbf{\Phi}}^N(t, u)|,
B_K^N(\de ')^c \Bigr],
\end{align*}
where $B_K^N(\de ')$ is defined in  Lemma
\ref{errorphi} below. Then, by Lemma \ref{Kol-Lem-2}, we
see that
\[
E\Bigl[\sup_{t\in [0, T]}\bigl\{1\wedge \sup_{u\in [-K,K]}
\bigl|\bar{\Phi}^N(t+\de , u) - \bar{\Phi}^N(t, u) \bigr| \bigr\}
\Bigr]
\leq P(B_K^N(\de )) + \tilde{\de } + C\sqrt{N}(\e^{-1}-1)
\]
with $\tilde{\de } =2\de ' + C \de ^{\a }$. This implies
\eqref{tightPhi-1-3} because $\de '$ and $K$
are arbitrary and $P(B_K^N(\de )) \to 0$ by Lemma
\ref{errorphi} (see below).
\end{proof}
As a last step let us formulate the lemma, needed in the proof of
Proposition \ref{tightPhi-1} above. This
lemma tells us that the processes $\bar{\Phi}^N(t, u)$ and
$\bar{\mathbf{\Phi}}^N(t, u)$ are uniformly close.
\begin{lem}\label{errorphi}
For any $\de '>0$ and $K \in \mathbb{N}$,  consider the
following event:
\[
B_K^N(\de ') := \Bigl\{\sup_{t\in [0, T]}\sup_{u\in [-K,
K]}|\bar{\Phi}^N(t, u) - \bar{\mathbf{\Phi}}^N(t, u)| \geq \de '
\Bigr\}.
\]
Then we have that $\lim_{N \to \infty} P(B_K^N(\de '))=0$.
\end{lem}
\begin{proof}
Set
\[
I=\bigl\{(k, x): k=0,1,2, \cdots, [N^2T],\ x\in \mathbb{Z}\
\text{ s.\,t. } \min_{u \in [-K, K]} |Nu -x|\leq 1 \bigr\}.
\]
It is easy to
see that the number of the elements in $I$ is bounded from above by
$CN^{3}$, that is, $\#I \leq CN^{3}$. Based on this  observation, let us
first show that for
any $(k, x) \in I$
\begin{equation}\label{errorphi-1}
E \Bigl[\sup_{N^2t \in [k, k+1]}\sup_{Nu\in [x, x+1]}
|\bar{\Phi}^N(t, u)-\bar{\mathbf{\Phi}}^N(t, u)|^{2\k } \Bigr] \leq
CN^{-4\k\a }, \quad  \a  <1/4,
\end{equation}
where $C$ is a generic constant and is independent of $N$, $k$ and $x$. From
the definitions of $\bar{\Phi}^N(t, u)$ and $\bar{\mathbf{\Phi}}^N(t,
u)$, we easily see that
\begin{align*}\label{errorphi-2}
 |\bar{\Phi}^N(t, u)-& \bar{\mathbf{\Phi}}^N(t, u)| 
\leq  |\bar{\Phi}^N \bigl( \tfrac{k}{N^2}, x+1 \bigr)-\bar{\Phi}^N(t,
x+1)| +
|\bar{\Phi}^N \bigl(\tfrac{k+1}{N^2}, x \bigr)-\bar{\Phi}^N(t, x)|
\notag \\
& + |\bar{\Phi}^N \bigl(\tfrac{k}{N^2}, x \bigr)-\bar{\Phi}^N
\bigl(\tfrac{k +1}{N^2},x \bigr)| + |\bar{\Phi}^N \bigl(\tfrac{k}{N^2},
x+1 \bigr)-\bar{\Phi}^N \bigl(\tfrac{k+1}{N^2}, x+1 \bigr)|.
\end{align*}
By the definition of $\bar{\Phi}^N(t, u)$ and \eqref{barphi-4}, for
some $\a  <1/4$, we observe that
\begin{equation}\label{errorphi-2-0}
E\Bigl[\bigl|\bar{\Phi}^N \bigl(\tfrac{k}{N^2}, x\bigr) - \bar{\Phi}^N
\bigl(\tfrac{k+1}{N^2}, x \bigr) \bigr|^{2\k } +
\bigl|\bar{\Phi}^N \bigl(\tfrac{k}{N^2}, x+1 \bigr) - \bar{\Phi}^N
\bigl(\tfrac{k+1}{N^2}, x+1 \bigr) \bigr|^{2\k } \Bigr] \leq \frac{C}
{N^{4\k\a }}.
\end{equation}
In the following, to  conclude the proof of
\eqref{errorphi-1}, we first show that
\begin{equation}\label{errorphi-2-1}
E \Bigl[\sup_{t\in I_N(k)}
\bigl|\bar{\Phi}^N \bigl(\tfrac{k}{N^2}, x+1 \bigr)-\bar{\Phi}^N(t, x+1)
\bigr|^{2\k } \Bigr]\leq  CN^{-\k},
\end{equation}
where $I_N(k) = [\frac{k}{N^2}, \frac{k+1}{N^2}]$. By the definition of
$\bar{\Phi}_t^N(x)$, it follows that the left side of \eqref{errorphi-2-1}
is bounded from above by
\[
CE\Bigl[\sup_{t\in I_N(k)}\Bigl(\sqrt{N}\bigl(\zeta_{\frac{k}
{N^2}}^N(x+1)-
\zeta_t^N(x+1) \bigr)\Bigr)^{2\k } \Bigr] + C\sup_{t\in
I_N(k)}\Bigl(\sqrt{N}
\bigl(\bar{\om}_{\frac{k}{N^2}}^N(x+1)-
\bar{\om}_t^N(x+1) \bigr)\Bigr)^{2\k }.
\]
It is known  that
\[
\sup_{t\in I_N(k)}\bigl|\bar{\om}_{\frac{k}{N^2}}^N(x+1)-
\bar{\om}_t^N(x+1)\bigr|\leq N^{-1}.
\]
Hence, to show \eqref{errorphi-2-1}, it is enough to prove that there exists
a constant $C$ such that
\begin{equation}\label{errorphi-2-2}
E\Bigl[\sup_{t\in I_N(k)}\bigl(\zeta_{\frac{k}{N^2}}^N(x+1)-
\zeta_t^N(x+1)
\bigr)^{2\k } \Bigr]\leq CN^{-2\k}.
\end{equation}
To show this, we will use the martingale approach. For each $x \in
\mathbb{N}$, we have that
\[
\zeta_t^N(x)=\zeta_0^N(x) +\int_0^t \bar{L}^N
\zeta_s^N(x)ds +M_t^N(x),
\]
where  
\begin{align*}
\bar{L}^N
\zeta_s^N(x) = & N^2 \bigl(\e c_+(x-1, \eta_s^N)+
c_-(x-1,\eta_s^N) \bigr)\zeta_s^N(x) \Bigl( e^{-\log \e
(\eta_s^N(x-1) -\eta_s^N(x))} -1 \Bigr),   x\geq 2, \\
\bar{L}^N
\zeta_s^N(1) = & N^2 \bigl(\e 1_{ \{\eta_s^N(1)=0\}}
+1_{\{\eta_s^N(1)=1\}} \bigr) \zeta_s^N(1)\Bigl(e^{-\log \e
(1 -2\eta_s^N(1))} -1 \Bigr).
\end{align*}
From the expression of $\bar{L}^N
\zeta_s^N(x)$, it is easy to deduce that there exists a constant $C$ such
that for any $x \in \mathbb{N}$ $\bar{L}^N \zeta_s^N(x) \leq C
N\zeta_s^N(x)$ and thus, it follows that
\[
 E\Bigl[\sup_{t\in I_N(k)}   \Bigl|\int_{\frac{k}{N^2}}^t \bar{L}^N
\zeta_s^N(x+1)ds \Bigr|^{2\k} \Bigr] \leq    C
N^{-2\k}E\Bigl[\sup_{t\in I_N(k)}
\zeta_t^N(1)^{2\k} \Bigr].
\]
Since  $a_N$ and $b_N$ converge to $\b ^2$ as $N\to \infty$,
\eqref{eq:4.SDE-1-1} yields that for any $x\in \mathbb{N}$
\[
d\lan M^N(x) \ran_s \leq C \zeta_s^N(x)^2 ds,
\]
and then, by Lemma \ref{BurkIneq} and \eqref{mom-10},
\begin{align*}
 & E\Bigl[\sup_{t\in I_N(k)
}|M_{\frac{k}{N^2}}^N(x+1) -M_{t}^N(x+1)|^{2\k } \Bigr] 
\leq CE\Bigl[\sup_{t\in I_N(k) }|
M^N_{t}(x+1) - M^N_{t-}(x+1)|^{2\k } \Bigr]\\
  &   + CE\Bigl[\sup_{t\in I_N(k)
}\Bigl(\lan M^N(x+1)\ran_{t} -\lan
M^N(x+1)\ran_{\frac{k}{N^2}}\Bigr)^{\k } \Bigr] \leq
\frac{C}{N^{2\k}}\Bigl( E \Bigl[\sup_{t\in I_N(k)} \zeta_t^N(1)^{2\k}
\Bigr] + 1 \Bigr).
\end{align*}
Therefore, by Lemma \ref{Lem4.8}, we can show \eqref{errorphi-2-2}.
A similar argument yields
\begin{equation}\label{errorphi-2-2-2}
 E\Bigl[\sup_{t\in I_N(k)} \bigl|
 \bar{\Phi}^N \bigl(\tfrac{k+1}{N^2}, x \bigr)-\bar{\Phi}^N(t, x) \bigr
|^{2\k} \Bigr]
 \leq C N^{-\k}.
\end{equation}
So, by \eqref{errorphi-2-0}-\eqref{errorphi-2-2-2}, we  can complete the
proof of \eqref{errorphi-1}.

Finally, the proof can be  concluded by \eqref{errorphi-1} and Chebyshev's
inequality. In fact,
\begin{align*}
& P(B_K^N(\de ')) \leq (\de ')^{-2\k }E\Bigl[\sup_{t\in [0,
T]}\sup_{u\in [-K, K]}|\bar{\Phi}(t, u) - \bar{\mathbf{\Phi}}(t,u)|^{2\k } \Bigr]\\
\leq& \sum_{(k, x) \in I}(\de ')^{-2\k }E \Bigl[\sup_{N^2t \in
[k, k+1]}\sup_{Nu\in [x, x+1]} |\bar{\Phi}^N(t, u)-
\bar{\mathbf{\Phi}}^N(t, u)|^{2\k } \Bigr] \leq CN^{3-4\k\a },
\end{align*}
which implies the result by taking  $\k> \frac{3}{4\a }$ and then letting
$N\to \infty$.
\end{proof}
\subsubsection{Derivation of the SPDE \eqref{eq:SPDE-4.3}}\label{subsec-4.3.3}
Taking a test function $g \in C^2_0(\R)$ and by
\eqref{barphi-1} and the definition of $\bar{\Phi}^N(t,u)$,  we arrive at
\begin{equation} \label{eq:3.2.1b}
 \lan\bar{\Phi}^N(t,\cdot), g\ran = \frac1N \sum_{x\in\Z}
\bar{\Phi}_0^N(x)g \bigl(\tfrac{x}{N} \bigr)
  + \int^t_0 b^N(\bar{\Phi}_s^N,g)ds + \sqrt{N} \bar{M}_t^N(g) +
  R_t^N,
\end{equation}
where  $\lan\bar{\Phi}^N(t,\cdot), g\ran= \int_{\R}
\bar{\Phi}(t, u)g(u) du$,
\begin{align*}
b^N(\bar{\Phi}_s^N,g) = & \frac1N \sum_{x\in\Z} N^2\e^{1/2}\De
g \bigl(\tfrac{x}{N} \bigr) \bar{\Phi}_s^N(x)
 + \frac{c_N}{N}\sum_{x\in\Z} g \bigl(\tfrac{x}{N}
\bigr)\bar{\Phi}_s^N(x),  \\
\sqrt{N}\bar{M}_t^N(g)=& \frac{1}{\sqrt{N}} \sum_{x\in\Z}
\bar{M}_t^N(x) g \bigl(\tfrac{x}{N}\bigr),
\end{align*}
and $R_t^N$ is an error term, i.\,e. $R_t^N= \lan\bar{\Phi}^N(t,\cdot),
g\ran -\frac1N \sum_{x\in \mathbb{Z}} \bar{\Phi}_t^N(x)g
\bigl(\tfrac{x}{N} \bigr)$.

We first deal with the error term $R_t^N$. It is easy to show that $R_t^N$
is bounded from above by
\[
\int_{\R} \bigl|(\bar{\Phi}_t^N([Nu])
-\bar{\Phi}_t^N([Nu]+1))g(u) \bigr|du +
N^{-1}\|g'\|_{\infty}\int_{{\rm supp}(g)}
|\bar{\Phi}_t^N([Nu])|du.
\]
Thus, we easily see that
$R_t^N$ converges to $0$ in $L^{2\k }(\Omega)$. In fact, by Lemma
\ref{Lemma-barphi-1}, we have
\[
E\Bigl[\Bigl(\int_{\R} \bigl|\bigl(\bar{\Phi}_t^N([Nu])
-\bar{\Phi}_t^N([Nu]+1)\bigr)g(u)\bigr|du \Bigr)^{2\k }
\Bigr] \leq
CN^{-2\k\a }\bigl(\|g\|_{\infty}|{\rm supp}(g)|\bigr)^{2\k},
\]
which goes to $0$ as $N\to \infty$,  and for the second term, we can
apply Lemma \ref{Est-mon}.

Use \eqref{eq:4.SDE-1-1} for the martingale term $\sqrt{N}\bar{M}_t^N(g)$
and observe that
\begin{align*}
\frac{d}{dt} \lan & \sqrt{N}  \bar{M}^N (g)\ran_t  = \frac{1}N
\zeta_t^N(1)^2 e^{-(\log\e)} \Bigl( a_N
1_{\{\eta_t^N(1)=0\}} + b_N1_{\{\eta_t^N(1)=1\}}\Bigr)g \bigl( \tfrac1N
\bigr)^2 \\
 + & \frac{1}N \sum_{x=2}^\infty \zeta_t^N(x)^2 e^{-(\log\e)x} \Bigl(
a_N
c_+(x-1,\eta_t^N ) + b_N  c_-(x-1,\eta_t^N ) \Bigr)  \times \Bigl(g
\bigl( \tfrac{x}{N} \bigr) +g \bigl(-\tfrac{x}{N} \bigr) \Bigr)^2,
\end{align*}
which converges as $N\to\infty$ to
\[
\b^2 \int_{\R_+}  e^{\b u} \om(t,u)^2 \cdot
2\rho_R(t,u)(1-\rho_R(t,u)) \bigl(g(u)+g(-u)\bigr)^2 du,
\]
by the hydrodynamic limit, i.e., by Corollary 5.3 of \cite{FS}.

Now let us state the following lemma, which concludes the proof of
Proposition \ref{prop:4.4}.

\begin{lem}\label{lem:4.17}
There exists a $Q$-cylindrical Brownian motion $\bar{W}$ with the covariance
determined by \eqref{Q-Wiener} such that the weak limit of $\sqrt{N}
\bar{M}_t^N(g)$ as $N\to\infty$ has the same law as that of the process
\[
\b  \int_0^t\int_{\mathbb{R}}
 e^{\b |u|/2} \om(s,|u|)\sqrt{2\rho_R(s,|u|)(1-\rho_R(s,|u|))} g(u)\bar{W}
(dsdu).
\]
Therefore, the limit of $\bar{\Phi}^N(t,u)$ is characterized by the SPDE
\eqref{eq:SPDE-4.3}.
\end{lem}
\begin{proof}
Let us consider
\begin{align*}
\mathbf{M}_t^N(g)=& \frac1N \sum_{x\in \mathbb{Z}}
\bar{\Phi}_t^N(x)g \bigl(\tfrac{x}{N} \bigr)-\frac1N \sum_{x\in\Z}
\bar{\Phi}_0^N(x) g \bigl(\tfrac{x}{N} \bigr)-\int^t_0
b^N(\bar{\Phi}_s^N,g)ds, \\
\bar{\mathbf{M}}_t^N(g)= &\bigl(\mathbf{M}_t^N(g)\bigr)^2- \lan \sqrt{N}
\bar{M}^N(g) \ran_t.
\end{align*}
Here, $\mathbf{M}_t^N(g)$ is nothing but $\sqrt{N} \bar{M}_t^N(g)$
appeared in \eqref{eq:3.2.1b}. However, to make the explanation of the proof
clear, we introduce this notation. From the definition of
$\bar{\Phi}_t^N(x)$, we know that both of the above processes are
martingales. Let $\mathcal{P}$ be a limit point of the sequence ${P}^N$,
the distribution of $\bar{\Phi}^N(t, \cdot)$ on $D([0,T], C(\mathbb{R}))$.
Then, it is clear that $\mathcal{P}$ is concentrated on $C([0, T],
C(\mathbb{R}))$ from Lemma \ref{Kol-Lem-2}. In the following, with some
abuse of notations, we will use $\Phi(t)$ to denote the canonical coordinate
process on $C([0, T], C(\mathbb{R}))$. Assume $F$ denotes an arbitrary
$D([0, s], C(\mathbb{R}))$-measurable function defined on $D([0, T],
C(\mathbb{R}))$ with continuous and bounded restriction on $C([0, T],
C(\mathbb{R}))$. From the explanations at the beginning of this subsection,
for $0\leq s< t\leq T$, letting $N \to \infty$, we can show
\begin{align*}
E^{\mathcal{P}}\bigl[(\mathbf{M}_t(g)-\mathbf{M}_s(g))F \bigr] = &
\lim_{N \to
\infty}
E^{{P}^N} \bigl[(\mathbf{M}_t^N(g)-\mathbf{M}_s^N(g))F \bigr]=0, \\
E^{\mathcal{P}} \bigl[(\bar{\mathbf{M}}_t(g)-\bar{\mathbf{M}}_t(g))F
\bigr] = &
\lim_{N \to \infty}
E^{P^N} \bigl[(\bar{\mathbf{M}}_t^N(g)-\bar{\mathbf{M}}_t^N(g))F
\bigr]=0,
\end{align*}
where $E^{\mathcal{P}}$ denotes the expectation with respect to
$\mathcal{P}$,
\begin{align}\label{margsol-1}
\mathbf{M}_t(g):=&\lan \Phi(t),g \ran - \lan \Phi(0), g
\ran - \int_0^t \lan \Phi(s), g'' -\tfrac{\b ^2}{4}g\ran
ds,\\
\bar{\mathbf{M}}_t(g) :=& \bigl(\mathbf{M}_t(g)\bigr)^2 -\int_0^t
\int_\mathbb{R} \psi^2(s,u)g(u)\bigl(g(u)+g(-u)\bigr)du ds,
\label{margsol-2}
\end{align}
and $\psi(t, u)=\b\om(t,|u|) e^{\b|u|/2 } \cdot \sqrt{2\rho_R(t,|u|)
(1-\rho_R(t,|u|))} $ in this part. Therefore, we deduce that both of the
processes $\mathbf{M}_t(g)$ and $\bar{\mathbf{M}}_t(g)$ defined by
\eqref{margsol-1} and \eqref{margsol-2}, respectively,  are $\mathcal{P}$-
martingales.

Using a similar way to \cite{BG}, we call that a probability measure
$\mathcal{P}$ on $C([0, T], C(\mathbb{R}))$ is a martingale solution of
\eqref{eq:SPDE-4.3} if the law of $\Phi(0)$ under $\mathcal{P}$ coincides
with the law of $\mathbf{\Phi}_0$ under $P$ and for any test function $g$,
$\mathbf{M}_t(g)$ and $\bar{\mathbf{M}}_t(g)$ are $\mathcal{P}$-local
martingales. We refer to \cite{F} for another approach to study martingale
problems for SPDEs.

In the following, we will show that the martingale solution of
\eqref{eq:SPDE-4.3} is equivalent to its weak solution. To show this, we
associate a martingale measure $\mathbf{M}(t, A)$ on $[0,T] \times
\mathbb{R}$ to $\mathbf{M}_t(g)$. In other words, we will assume that
$\mathbf{M}(t, A)$ is a continuous worthy martingale measure, see
\cite{WAL},  with quadratic variational process
\[
\lan\mathbf{M}\ran(dtdu) = \psi^2(t,u)dt\nu(dv),
\]
where $\nu(A) =|A|+|-A|$ for any Borel subset $A$ of $\mathbb{R}$ and  $-A
:=\{-x:\  x\in A\}$. Let us consider a $Q$-cylindrical Brownian motion
$\bar{\mathbf{W}}$ with covariance defined by \eqref{Q-Wiener} such that
it is independent of $\mathcal{P}$. We remark that this can be realized by
extending the probability space and the corresponding filtration. However,
for the brevity of notation, we will still use $\mathcal{P}$ to denote the
extended  probability measure. Now set
\begin{equation}\label{margsol-3}
\mathbf{W}_t(g)= \int_0^t \int_\mathbb{R}
\frac{g(u)}{\psi(s,u)}1_{\{\psi(s,u)\neq 0 \}}\mathbf{M}(dsdu) +
\int_0^t\int_\mathbb{R}1_{\{\psi(s,u)= 0 \}}g(u)\bar{\mathbf{W}}(dsdu).
\end{equation}
From the symmetry of $\psi(t,u)$ in $u$ and the independence of
$\mathbf{M}$ and $\bar{\mathbf{W}}$, we see that
\[
E^{\mathcal{P}} \bigl[\mathbf{W}_t^2(g) \bigr]= \int_0^t
\int_\mathbb{R}
g(u) \bigl(g(u)+g(-u)\bigr) dsdu.
\]

Therefore, by L\'evy's martingale characterization theorem, we  know that
$\mathbf{W}_t$ is a $Q$-cylindrical Brownian motion with covariance
characterized  by \eqref{Q-Wiener} and
\[
\mathbf{M}_t(g)= \int_0^t \int_\mathbb{R} \psi(s,u)g(u)\mathbf{W}(dsdu).
\]
In fact, by the definition of $\mathbf{W}_t$, see \eqref{margsol-3}, we
have that
\begin{align*}
& \int_0^t \int_\mathbb{R} \psi(s,u)g(u)\mathbf{W}(dsdu) = \int_0^t
\int_\mathbb{R} \psi(s,u)\frac{g(u)}{\psi(s,u)}1_{\{\psi(s,u)\neq 0
\}}\mathbf{M}(dsdu) \\
&  +\int_0^t\int_\mathbb{R}1_{\{\psi(s,u)= 0
\}}\psi(s,u)g(u)\bar{\mathbf{W}}(dsdu)= \int_0^t
\int_\mathbb{R}g(u)\mathbf{M}(dsdu).
\end{align*}
Combining this with \eqref{margsol-1}, we obtain that
\[
\lan \Phi(t),g \ran = \lan \Phi(0), g \ran + \int_0^t \lan \Phi(s), g''
-\frac{\b ^2}{4}g\ran ds + \int_0^t \int_\mathbb{R}
\psi(s,u)g(u)\mathbf{W}(dsdu),
\]
which means that the martingale solution satisfies \eqref{eq:SPDE-4.3} in
its weak sense with the $Q$-cylindrical Wiener process $\mathbf{W}(t)$
constructed by \eqref{margsol-3} by the arbitrariness of $g$. In the end, we
remark that the martingale problem is well-posed, that is, the uniqueness
holds, which is clear from the uniqueness of the weak solution.
\end{proof}
\section{Invariant Measures of the SPDEs}
To compare our dynamic fluctuation results with the static fluctuations
formulated in Proposition \ref{pro-CLT-5.1} below, we explicitly compute the
invariant measures of the SPDEs \eqref{eq:SPDE-1} and \eqref{eq:SPDE-2}.
\subsection{Static Fluctuations}
First, we state a result for the fluctuations under grandcanonical ensembles
$\mu_U^{\e(N)}$ and $\mu_R^{\e(N)}$, which is in fact simpler than those
under canonical ensembles, see \cite{P}, \cite{FVY}, \cite{Y}. Let
$\psi_U$ and $\psi_R$ be the height functions of the Vershik curves:
\begin{align*}
\psi_U(u)& =-\frac{1}{\a} \log\big(1-e^{-\a u}\big),
\quad u \in \R_+^\circ,\\
\psi_R(u)& =\frac{1}{\b} \log\big(1+e^{-\b u}\big),
\quad u \in \R_+.
\end{align*}
Then, for the static fluctuations $\Psi_U^N(u)$ and $\Psi_R^N(u)$ defined
by
\begin{align*}
& \Psi_U^N(u) := \sqrt{N} \big(\tilde{\psi}^N(u) -\psi_U(u)\big),
\quad u \in \R_+^\circ, \\
& \Psi_R^N(u) := \sqrt{N} \big(\tilde{\psi}^N(u) -\psi_R(u)\big),
\quad u \in \R_+,
\end{align*}
we have the following proposition.
\begin{prop}\label{pro-CLT-5.1}
The fluctuation fields $\Psi_U^N(u)$ and $\Psi_R^N(u)$ weakly converge to
$\Psi_U(u)$ and $\Psi_R(u)$ under $\mu_U^{\e(N)}$ and $\mu_R^{\e(N)}$,
respectively, as $N\to\infty$, where $\Psi_U, \Psi_R$ are mean $0$
Gaussian processes with covariance structures
\begin{align*}
& C_U(u,v) = \frac1\a \rho_U(u \vee v), \quad u, v \in \R_+^\circ\\
& C_R(u,v) = \frac1\b \rho_R(u \vee v), \quad u, v \in\R_+,
\end{align*}
and $\rho_U = - \psi_U' (= \rho_U^\infty$ in \eqref{eq:3.rho_infty}),
$\rho_R = - \psi_R'$ are slopes of the Vershik curves, respectively, with
$u \vee v=\max\{u, v\}$.
\end{prop}
\begin{proof}
The proof is not difficult by noting the following facts. Under $\mu_U^\e$,
the height differences $\xi(x)(=\psi(x-1)-\psi(x)\ \text{or} \ \#\{i; p_i
=x\}), x\in \N$, are independent random variables, which are geometrically
distributed: $\mu_U^\e(\xi(x)=k)=a^k/(1-a)$ for $k \in \Z_+$ with
$a=\e^x$. On the other hand, under $\mu_R^\e$, the height differences
$\eta(x), x \in \N$, are independent and distributed as
$\mu_R^\e(\eta(x)=k)=a^k/(1+a)$ for $k=0,1$ with $a=\e^x$.
\end{proof}
\begin{rem}
{\rm (1)} As shown in {\rm \cite{Y}, \cite{FVY}}, the CLT under canonical
ensembles can be reduced from that under grandcanonical ensembles by removing
the effect of fluctuations of area.  \newline
{\rm (2)} The Gaussian process $\Psi_R$ satisfies $\Psi_R\in L_r^2(\R_+)$
a.s.\ for every $r>-\b/2$ ($L_r^2$ is defined also for $r<0$), since
$$
E[|\Psi_R|_{L_r^2(\R_+)}^2] = \int_0^\infty E[\Psi_R(u)^2] e^{-2ru}du
= \frac1{\b} \int_0^\infty \rho_R(u)e^{-2ru}du
$$
is finite if and only if $2r+\b>0$.
\end{rem}
\subsection{Uniform Case}
Let $Q_U$ be the differential operator
\[
Q_U= - \frac{\partial}{\partial u} \Bigl\{ \frac1{\rho_U(u)
(1+\rho_U(u))} \frac{\partial}{\partial u} \Bigr\}
\]
defined on $L^2(\R_+^\circ,du)$. Note that this operator does not require any boundary condition, see Remark \ref{rem:2.1}.
\begin{thm}\label{thm:Uinvm}
The Gaussian measure $N(0,Q_U^{-1})$ is the unique invariant measure of the
SPDE \eqref{eq:SPDE-1}, which appeared in Theorem \ref{thm:2.1}.
\end{thm}

\begin{proof}
Since $\rho(t,u)$ in the SPDE \eqref{eq:SPDE-1} converges as $t\to\infty$
to $\rho_U(u)$,
we may study the invariant measure of the SPDE:
\begin{equation}\label{eq:SPDE-3}
\partial_t\Psi(t,u) = A_U\Psi(t,u) + \sqrt{2g_U(u)} \dot{W}(t,u),
\end{equation}
where
\[
A_U\Psi(u) :=\Bigl(\frac{\Psi'(u)} {(1+\rho_U(u))^2}\Bigr)' + \a
\frac{\Psi'(u)}{(1+\rho_U(u))^2} \quad \text{and}\quad
g_U(u)=\frac{\rho_U(u)}{1+\rho_U(u)}.
\]
Note that one can rewrite the operator $A_U$ as
\[
A_U\Psi(u) = - g_U(u) Q_U\Psi(u).
\]
In particular, $A_U$ is symmetric in the space $L_U^2 :=
L^2(\R_+^\circ, 1/g_U(u) du)$. Let $e^{tA_U}$ be the semigroup generated by
$A_U$ on $L_U^2$. Then, the solution of the SPDE \eqref{eq:SPDE-3} can be
written in the mild form:
\[
\Psi_t = e^{tA_U}\Psi_0 + \int_0^t e^{(t-s)A_U}\sqrt{2g_U} dW_s.
\]
In particular, for every $\psi\in L_U^2$, we have
\[
\lan\Psi_t,\psi\ran_{L_U^2}  = \lan e^{tA_U}\Psi_0,\psi\ran_{L_U^2} +
\int_0^t \lan dW_s, \tfrac{1}{g_U} \sqrt{2g_U} e^{(t-s)A_U}\psi\ran_{L^2}
=: m_t + I_t.
\]
However, since $A_U$ on $L_U^2$ is unitary equivalent to $-Q_U$ on
$L^2(\R_+^\circ)$, Lemma \ref{lem:5.3} below implies $A_U\le -c$ with $c>0$
and therefore $m_t\to 0$ as $t\to\infty$, while
\begin{align*}
E \bigl[I_t^2 \bigr] & = \int_0^t \|\sqrt{\tfrac{2}{g_U}}e^{(t-
s)A_U}\psi \|_{L^2}^2 ds = 2 \int_0^t \|e^{sA_U}\psi \|_{L_U^2}^2 ds  \\
& = 2 \int_0^t \lan e^{2sA_U}\psi,\psi\ran_{L_U^2} ds \to 2 \lan
(-2A_U)^{-1}\psi,\psi\ran_{L_U^2} = \lan (-A_U)^{-1}\psi,\psi\ran_{L_U^2}
\end{align*}
as $t\to\infty$. This proves that $\lan\Psi_t,\psi\ran_{L_U^2}$ converges
weakly to $N(0, \lan (-A_U)^{-1}\psi,\psi\ran_{L_U^2})$ for every
$\psi\in L_U^2$, which is an equivalent formulation to $\lan\Psi_t,
\fa\ran_{L^2}$ converging weakly to $N(0, \lan (-A_U)^{-1}(\fa{g_U}),
\fa\ran_{L^2})$ by taking $\fa= {\psi}/{g_U}$.  However, $(-A_U)^{-1}
(\fa{g_U}) = Q_U^{-1}\fa$ and this implies the conclusion.
\end{proof}

\begin{rem}
Since $C_U(u,v)$ is the Green kernel of $Q_U^{-1}$,
this gives another proof of the static result in U-case.
\end{rem}

\begin{lem}(Poincar\'e inequality; U-case) \label{lem:5.3}
There exists $c>0$ such that $(f,Q_Uf)\ge c\| f\|^2$ holds for every $f\in
C^1(\R_+^\circ) \cap L^2(\R_+^\circ,du)$, where the inner product and the
norm are those of the space $L^2(\R_+^\circ,du)$.
\end{lem}

\begin{proof}
We divide $\| f\|^2$ into a sum of integrals over $(1,\infty)$ and $(0,1]$,
and estimate them separately. We begin with the integral over $(1,\infty)$.
Set
\[
a_U(u) = \{u \,\rho_U(u)(1+\rho_U(u)) \}^{-1}, \quad u\in \R_+.
\]
Note that $a_U(u)>0$ and $C=\int_1^\infty a_U(u)^{-1}du<\infty$ and by
Schwarz's inequality, we have for every $f\in C_0^1(\R_+^\circ)$ that
\begin{align*}
\int_1^\infty f^2(u) du & = \int_1^\infty \Bigl(\int_u^\infty
f'(v)dv\Bigr)^2 du \le C \int_1^\infty du \int_u^\infty f'(v)^2 a_U(v) dv
\\
& \le C \int_1^\infty f'(v)^2 va_U(v) dv = C \int_1^\infty \frac{f'(u)^2}
{\rho_U(u)(1+\rho_U(u))} du.
\end{align*}
Next, we study the integral over $(0,1)$. By Schwarz's inequality,
\[
\int_0^1 f^2(u) du = \int_0^1 \Bigl(f(u)-f(1) +f(1) \Bigr) ^2 du \le 2
\int_0^1 \Bigl(f(u)-f(1) \Bigr) ^2 du +2 f(1)^2.
\]
We estimate two terms in the last expression separately. The first term is
estimated by Schwarz's inequality again as
\begin{align*}
f(1)^2 &= \Bigl(\int_1^\infty f'(u)du\Bigr)^2 \le  \Bigl( \int_1^\infty
\frac{f'(u)^2}{\rho_U(u)(1+\rho_U(u))} du \Bigr) \Bigl( \int_1^\infty
\rho_U(u)(1+\rho_U(u)) du \Bigr) \\
& \le C \int_1^\infty \frac{f'(u)^2}{\rho_U(u)(1+\rho_U(u))} du.
\end{align*}
To bound the remaining term, we need more detailed estimates. First, we
obtain the following bound
\begin{gather*}
\int_0^1 \bigl(f(u)-f(1) \bigr)^2 du = \int_0^1 \Bigl(\int_u^1
f'(v)dv\Bigr)^2 du \le  \int_0^1 \Bigl( \int_u^1 f'(v)^2 v^{\frac32} dv
\Bigr) \Bigl(\int_u^1 v^{-\frac32} dv \Bigr) du\\
\le 2 \int_0^1 u^{-\frac12} \Bigl( \int_u^1 f'(v)^2 v^{\frac32} dv
\Bigr) du = 2 \int_0^1 f'(v)^2 v^{\frac32} \Bigl( \int_0^v u^{-\frac12}
du \Bigr) dv = 4 \int_0^1 f'(v)^2 v^2 dv.
\end{gather*}
Inserting the relation
$\{\rho_U(u)(1+\rho_U(u))\}^{-1}=(e^{\a
u}-1)^2 e^{-\a u} \ge \a^2e^{-\a} u^2$, $ u \in [0,  1]$,
into the last term of the above inequality, we have that
\[
\int_0^1 \bigl(f(u)-f(1) \bigr) ^2 du \le 4\a^{-2} e^{\a} \int_0^1
\frac{f'(u)^2}{\rho_U(u)(1+\rho_U(u))} du.
\]
Combining inequalities obtained up to this point, we conclude that
\[
\int_0^\infty f^2(u) du \le \tilde{C} \int_0^\infty \frac{f'(u)^2}
{\rho_U(u)(1+\rho_U(u))} du =\tilde{C}(f,Q_Uf)
\]
where $\tilde{C}=3C + 8\a^{-2} e^{\a}$. The last equality follows by
integration by parts with $f\in C_0^1(\R_+^\circ)$ in mind. One can extend
the class of functions $f$.
\end{proof}
\begin{rem}
It is also possible to obtain the invariant measure for the U-case
from the one for the RU-case in Theorem \ref{thm:Rinvm} by
using the transformation used in Section \ref{sec:uniformproof} .
\end{rem}
\subsection{Restricted Uniform Case}
Let $Q_R$ be the differential operator
\[
Q_R= - \frac{\partial}{\partial u} \Bigl\{ \frac1{\rho_R(u)
(1-\rho_R(u))} \frac{\partial}{\partial u} \Bigr\}
\]
defined on $L^2(\R_+,du)$ with the Neumann boundary condition at $u=0$.

\begin{thm}\label{thm:Rinvm}
The Gaussian measure $N(0,Q_R^{-1})$ is the unique invariant measure of the
SPDE \eqref{eq:SPDE-2}, which appeared in Theorem \ref{thm:2.2}.
\end{thm}

\begin{proof}
Since $\rho(t,u)$ in the SPDE \eqref{eq:SPDE-2} converges as $t\to\infty$
to $\rho_R(u)$,
we may study the invariant measure of the SPDE:
\begin{equation}\label{eq:SPDE-4}
\left\{
\begin{aligned}
\partial_t\Psi(t,u) & = A_R\Psi(t,u) + \sqrt{2g_R(u)} \dot{W}(t,u), \\
\Psi'(t,0) & = 0,
\end{aligned}
\right.
\end{equation}
where
\[
A_R\Psi(u) :=\Psi''(u) + \b(1-2\rho_R(u))\Psi'(u) \quad \text{and} \quad
g_R(u)=\rho_R(u)(1-\rho_R(u)).
\]
Note that one can rewrite the operator $A_R$ as
\[
A_R\Psi(u) = - g_R(u) Q_R\Psi(u).
\]
In particular, $A_R$ is symmetric in the space $L_R^2 := L^2(\R_+,
1/g_R(u) du)$. Let $e^{tA_R}$ be the semigroup generated by $A_R$ on
$L_R^2$. Then, the solution of the SPDE \eqref{eq:SPDE-4} can be written in
the mild form:
\[
\Psi_t = e^{tA_R}\Psi_0 + \int_0^t e^{(t-s)A_R}\sqrt{2g_R} dW_s.
\]
In particular, for every $\psi\in L_R^2$, we have
\[
\lan\Psi_t,\psi\ran_{L_R^2}  = \lan e^{tA_R}\Psi_0,\psi\ran_{L_R^2} +
\int_0^t \lan dW_s, \tfrac{1}{g_R} \sqrt{2g_R} e^{(t-s)A_R}\psi\ran_{L^2}
=: m_t + I_t.
\]
However, since $A_R\le -c$ from Lemma \ref{lem:5.5} below, $m_t\to 0$ as
$t\to\infty$, while
\begin{align*}
E \bigl[I_t^2 \bigr] & = \int_0^t \|\sqrt{ \tfrac{2}{g_R}}e^{(t-
s)A_R}\psi \|_{L^2}^2 ds = 2 \int_0^t \|e^{sA_R}\psi \|_{L_R^2}^2 ds\\
& = 2 \int_0^t \lan e^{2sA_R}\psi,\psi\ran_{L_R^2} ds \to 2 \lan
(-2A_R)^{-1}\psi,\psi\ran_{L_R^2} = \lan (-A_R)^{-1}\psi,\psi\ran_{L_R^2}
\end{align*}
as $t\to\infty$. This proves that $\lan\Psi_t,\psi\ran_{L_R^2}$ converges
weakly to $N(0, \lan (-A_R)^{-1}\psi,\psi\ran_{L_R^2})$ for every
$\psi\in L_R^2$, which is an equivalent formulation of $\lan\Psi_t,
\fa\ran_{L^2}$  converging weakly to $N(0, \lan (-A_R)^{-1}(\fa{g_R}),
\fa\ran_{L^2})$ by taking $\fa= {\psi}/{g_R}$. However, $(-A_R)^{-1}
(\fa{g_R}) = Q_R^{-1}\fa$ and this implies the conclusion.
\end{proof}

\begin{rem}
Since $C_R(u,v)$ is the Green kernel of $Q_R^{-1}$,
this gives another proof of static result in RU-case.
\end{rem}

\begin{lem}(Poincar\'e inequality; RU-case) \label{lem:5.5}
There exists $c>0$ such that $(f,Q_Rf)\ge c\| f\|^2$ holds for every $f\in
C^1(\R_+) \cap L^2(\R_+,du)$ satisfying $f'(0)=0$, where the inner product
and the norm are those of the space $L^2(\R_+,du)$.
\end{lem}
\begin{proof}
Set
\[ a_R(u) = \{u \,\rho_R(u)(1-\rho_R(u))\}^{-1}, \quad u\in \R_+.\]
Note that $a_R(u)>0$ and $C=\int_0^\infty a_R(u)^{-1}du<\infty$ and by
Schwarz's inequality, we have for every $f\in C_0^1(\R_+)$ that
\begin{align*}
\int_0^\infty f^2(u) du & = \int_0^\infty \Bigl(\int_u^\infty
f'(v)dv\Bigr)^2 du   \le C \int_0^\infty du \int_u^\infty f'(v)^2 a_R(v)
dv\\
& = C \int_0^\infty f'(v)^2 va_R(v) dv = C(f,Q_Rf).
\end{align*}
The last equality follows by integration by parts with $f'(0)=0$ in mind.
One can extend the class of functions $f$.
\end{proof}

\end{document}